\documentclass[10t,a4paper]{article}
\usepackage[a4paper]{geometry}
\usepackage{amssymb,latexsym,amsmath,amsfonts,amsthm}
\usepackage{graphicx}

\usepackage{epsfig}
\usepackage{comment}
\usepackage{subfigure}
\usepackage{graphicx,ebezier}
\usepackage{overpic}
\usepackage{mathrsfs,enumerate}

\DeclareMathOperator{\diag}{diag}
\DeclareMathOperator{\Res}{Res}
\newcommand{\lam}{\lambda}
\newcommand{\Lam}{\Lambda}
\newcommand{\bol}{\hfill\square\\}
\newcommand{\til}{\tilde}
\newcommand{\wtil}{\widetilde}
\newcommand{\crit}{\textrm{crit}}

\newcommand{\ep}{\epsilon}
\renewcommand{\Re}{\mathrm{Re}\,}
\renewcommand{\Im}{\mathrm{Im}\,}
\newcommand{\ds}{\displaystyle}

\newcommand{\Ai}{\mathrm{Ai}}

\newcommand{\supp}{\mathrm{supp}}
\newcommand{\Mdec}{M^{(1)}}
\newcommand{\Mmod}{M^{(2)}}
\newcommand{\Mhom}{M^{(3)}}
\newcommand{\Mhomdef}{M^{(4)}}
\newcommand{\Mhomtris}{M^{(5)}}

\newtheorem{theorem}{Theorem}[section]
\newtheorem{lemma}[theorem]{Lemma}
\newtheorem{proposition}[theorem]{Proposition}

\newtheorem{corollary}[theorem]{Corollary}

\newtheorem{rhp}[theorem]{RH problem}

\theoremstyle{definition}
\newtheorem{definition}[theorem]{Definition}

\theoremstyle{remark}

\newtheorem{remark}[theorem]{Remark}

\numberwithin{equation}{section}

\title{Critical behavior of non-intersecting Brownian motions at a tacnode}

\author{Steven Delvaux\footnotemark[1],\quad Arno B.~J.~Kuijlaars\footnotemark[1],\quad Lun Zhang\footnotemark[1]}
\date{\today}

\begin{document}

\maketitle
\renewcommand{\thefootnote}{\fnsymbol{footnote}}
\footnotetext[1]{Department of Mathematics, Katholieke Universiteit Leuven,
Celestijnenlaan 200B, B-3001 Leuven, Belgium. email:
\{steven.delvaux, arno.kuijlaars, lun.zhang\}\symbol{'100}wis.kuleuven.be.}

\begin{abstract}
We study a model of $n$ one-dimensional non-intersecting Brownian
motions with two prescribed starting points at time $t=0$ and two prescribed
ending points at time $t=1$ in a critical regime where the paths fill two
tangent ellipses in the time-space plane as $n \to \infty$.
The limiting mean density for the positions of the Brownian paths
at the time of tangency consists of two touching semicircles, possibly of different sizes.
We show that in an appropriate double scaling limit, there is a new familiy of limiting
determinantal point processes with  integrable correlation kernels that are
expressed in terms of a new Riemann-Hilbert problem of size $4\times 4$.
We prove solvability of the Riemann-Hilbert problem and establish
a remarkable connection with the Hastings-McLeod solution of the
Painlev\'e~II equation. We show that this Painlev\'e~II transcendent also
appears in the critical limits of the recurrence coefficients of the
multiple Hermite polynomials that are associated with the non-intersecting Brownian motions.
Universality suggests that the new limiting kernels apply to more general
situations whenever a limiting mean density vanishes according to two touching
square roots, which represents a new universality class.


\end{abstract}

\setcounter{tocdepth}{2} \tableofcontents

\section{Introduction}
\label{section:introduction}

In recent years the model of non-intersecting Brownian
motions has been studied in various regimes, see e.g.\
\cite{ADV,AVV,DelKui1,DelKui2,KT1,KT2,KT3,TW4},
where many connections with determinantal point processes and random matrix
theory were found, see also \cite{KIK,KMW,KO,TW4} for non-intersecting Bessel
paths and Brownian excursions.  Their discrete counterparts, the non-intersecting random walks,
have important connections with tiling and random growth models,
see e.g.\ \cite{Baik,BorKuan,Joh1,Joh2,PS}.

In this paper we consider one-dimensional non-intersecting Brownian
motions with two prescribed starting points at time $t=0$ and two
prescribed ending points at time $t=1$. We assume that the number of
paths that leave from the topmost (bottommost) starting point is the
same as the number of paths that arrive at the topmost (bottommost)
ending point. As the number of paths increases and simultaneously
the overall variance of the Brownian transition probability
decreases, we may create various situations that are illustrated in
Figure~\ref{fig:3cases}.

\begin{figure}[tbp]
\begin{center}
\subfigure{\label{figlargesep}}\includegraphics[scale=0.3]{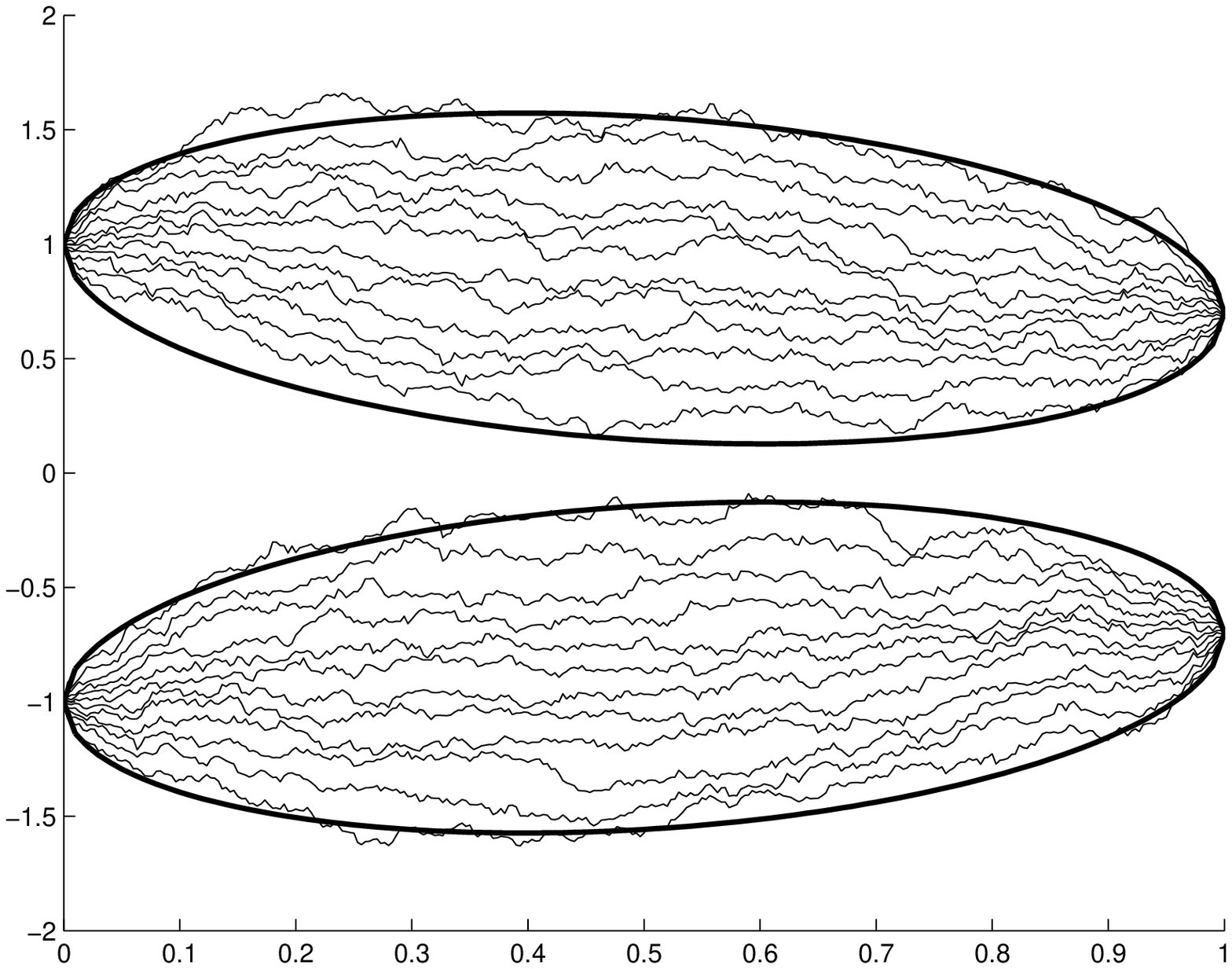}\hspace{5mm}
\subfigure{\label{figsmallsep}}\includegraphics[scale=0.3]{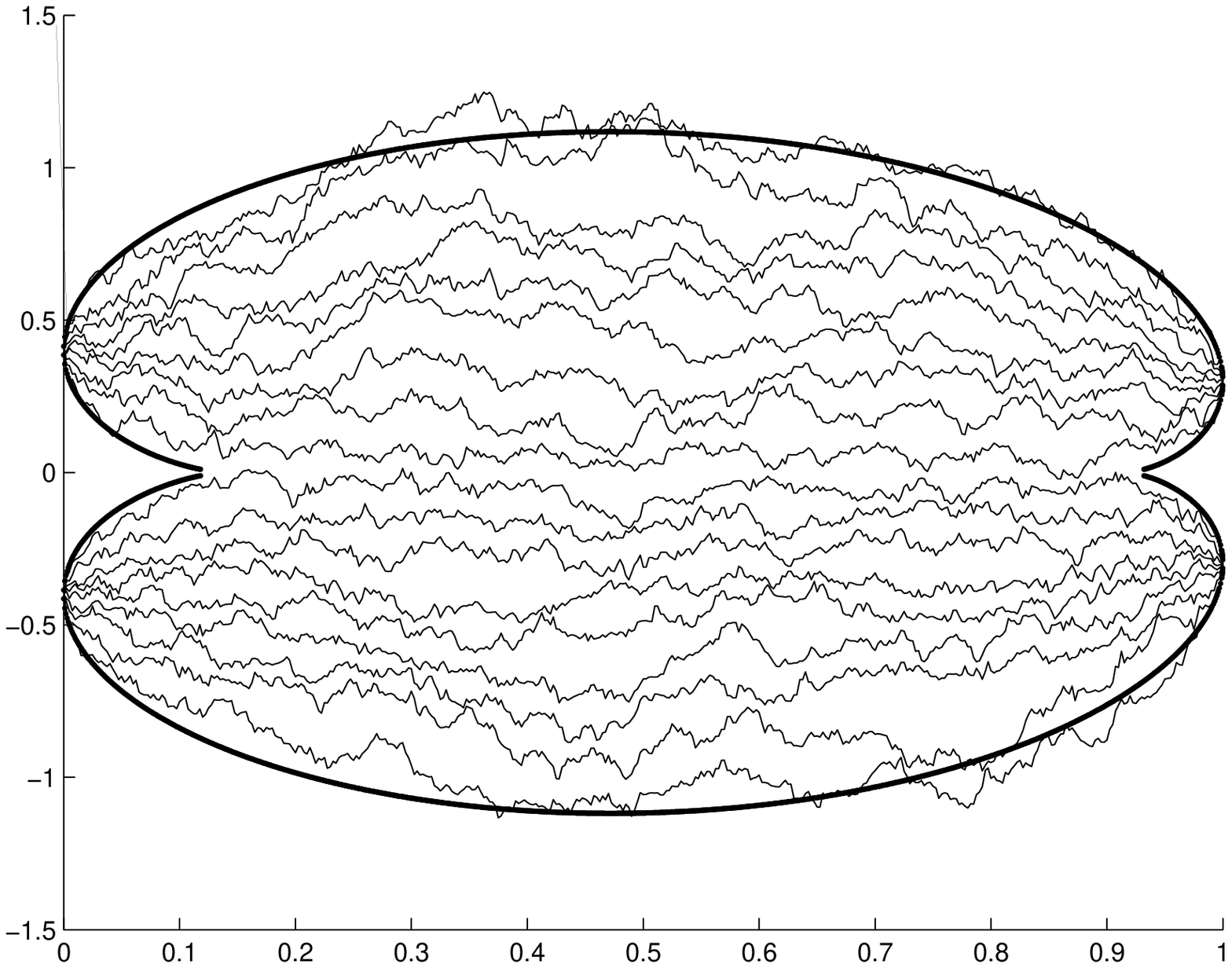}
\subfigure{\label{figcriticalsep}}\includegraphics[scale=0.3]{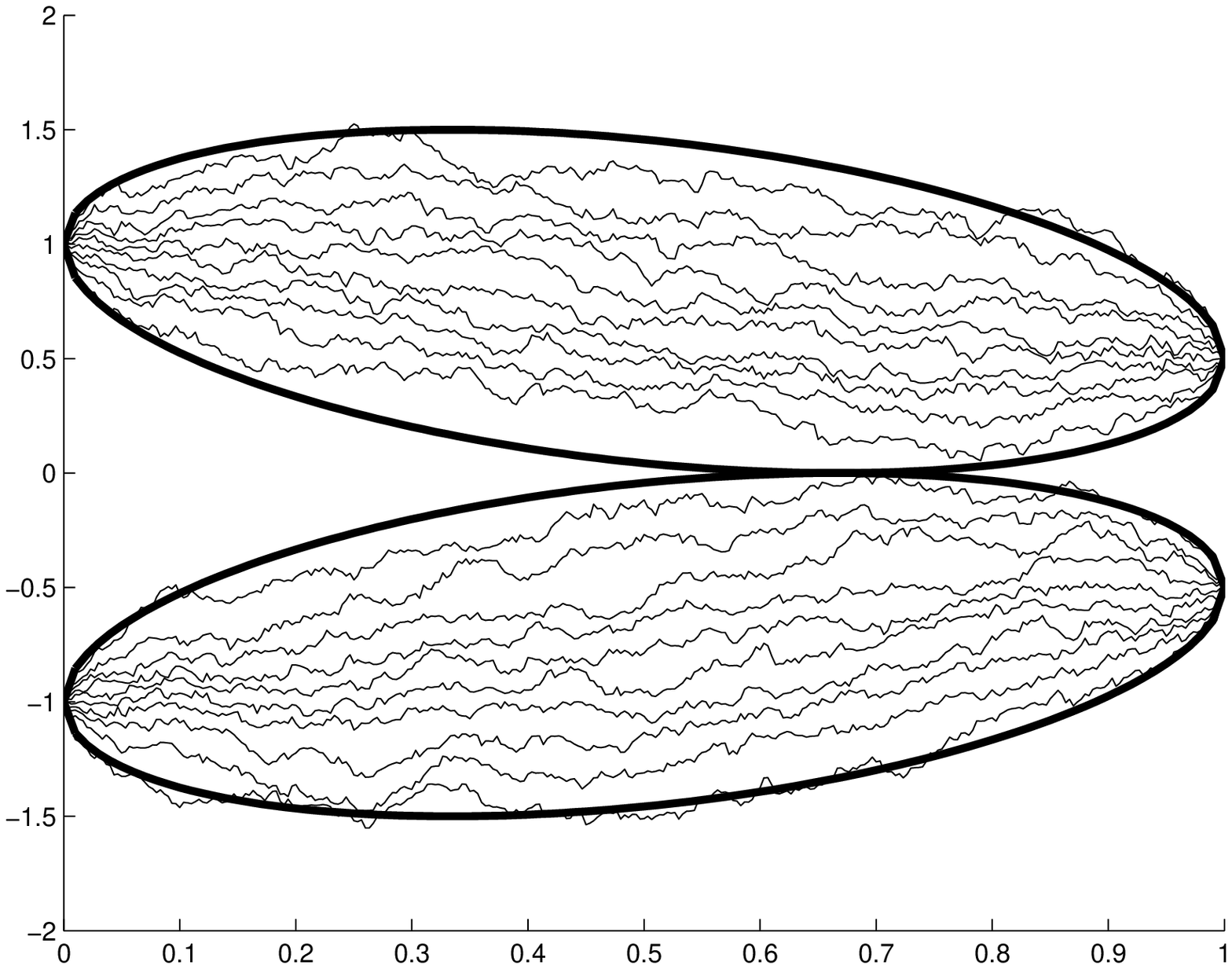}
\end{center}
\caption{Non-intersecting Brownian motions with two starting and two ending
positions in case of (a) large, (b) small, and (c) critical separation between
the endpoints. Here the horizontal axis denotes the time, $t\in[0,1]$, and for
each fixed $t$ the positions of the $n$ non-intersecting Brownian motions at
time $t$ are denoted on the vertical line through $t$. Note that for
$n\to\infty$ the positions of the Brownian motions fill a prescribed region in
the time-space plane, which is bounded by the boldface lines in the figures.
Here we have chosen $N=n=20$ and $p_1=p_2=1/2$ in each of the figures, and (a)
$a_1=-a_2=1$, $b_1=-b_2=0.7$, (b) $a_1=-a_2=0.4$, $b_1=-b_2=0.3$, and (c)
$a_1=-a_2=1$, $b_1=-b_2=1/2$, in the cases of large, small and critical
separation, respectively.} \label{fig:3cases}
\end{figure}

In case of large separation of the starting and ending points we have a
situation as in Figure~\ref{fig:3cases}(a). The paths are in two disjoint
groups, where one group of paths goes from the topmost starting point to the
topmost ending point, and the other group goes from the bottommost starting
point to the bottommost ending point. In the large $n$ limit the paths fill out
two disjoint ellipses.

In case of small separation of the starting and ending points we have a
situation as in Figure~\ref{fig:3cases}(b). Here the two groups of paths that
emanate from the two starting points merge at a certain time, stay together for
a while, and separate at a later time. In the large $n$ limit the paths fill
out a region that is bounded by a more complicated curve with two cusp
points. The critical behavior at the cusp point is known to be described
by the Pearcey process, see \cite{AvM1, BK3,  OR, TW4}.

In the transitional case of critical separation the paths fill out two ellipses
that are tangent at a critical time as shown in Figure~\ref{fig:3cases}(c). The
case of critical separation was already considered by the first two authors
\cite{DelKui1}, but at the non-critical time. Here we consider the
behavior at the critical time. Note that the tangent point of two ellipses is
called a tacnode in the classification scheme of singular points of
algebraic curves, whence the title of this paper; see also \cite{BorodinDuits}
where another model is analyzed with a tacnode, but with markedly different
properties.

The phase transition at critical separation can already be observed
at times different from the critical time $t_{\crit}$. In
\cite{DelKui1} a connection was found with the Hastings-McLeod
solution of the Painlev\'e~II equation
\begin{equation}\label{def:Painleve2}
    q''(s) = sq(s)+2q^3(s).
\end{equation}
Here the prime denotes the derivative with respect to $s$. The Hasting-McLeod
solution \cite{HML} is the special solution $q(s)$ of \eqref{def:Painleve2},
which is real for real $s$ and satisfies
\begin{equation}\label{def:HastingMcLeod}
    q(s) = \Ai(s)(1+o(1)), \qquad \text{as } s\to +\infty,
\end{equation}
where $\Ai$ denotes the usual Airy function.

The Hastings-McLeod solution $q(s)$ (or functions related to it) do
not appear in the local scaling limits of the correlation kernel at
time $t \neq t_{\crit}$. The Hastings-McLeod solution, however, does
appear in the asymptotics of the recurrence coefficients of the
multiple Hermite polynomials related to the non-intersecting
Brownian motions, see \cite{DelKui1}.

The results in \cite{DelKui1} were obtained from the Deift-Zhou steepest
descent analysis of the $4 \times 4$ matrix-valued Riemann-Hilbert (RH) problem
for multiple Hermite polynomials. During the analysis we had to construct a
local parametrix at a point $x_0 \in \mathbb R$ that lies strictly between the
two intervals where the two groups of Brownian paths accumulate. The point
$x_0$ does not have a physical meaning. However, the local parametrix affects
the recurrence coefficients of the multiple Hermite polynomials, as mentioned
before.

The aim of this paper is to perform a similar steepest descent analysis for the
critical time $t=t_{\crit}$. This multicritical situation, after appropriate
scaling  of the parameters will be locally described by the solution of a model
RH problem of size $4\times 4$. The RH problem can be considered as a
combination of two  RH problems for the Airy function with an
additional non-trivial coupling in the jump matrices.

Using this new $4\times 4$ model RH problem, we obtain an expression for the
limit of the correlation kernel at the critical time. We show that the kernel
has an integrable form determined by the solution to the RH problem, see
Theorem~\ref{theorem:kernelpsi}.

We find it remarkable that this new $4\times 4$ RH problem is again related to
the Hastings-McLeod solution of the Painlev\'e~II equation. More precisely, we
prove that the Hastings-McLeod solution shows up in the residue matrix in the
asymptotic series at infinity of the $4\times 4$ model RH problem, see
Theorem~\ref{theorem:Painleve2modelrhp}. This is very similar to the situation
for the classical $2\times 2$ RH problem for Painlev\'e~II due to
Flaschka-Newell \cite{FN,FIKN}. This suggests that our $4\times 4$ problem may
be expressible in terms of this smaller $2\times 2$ problem; however we could
not find such an expression. So our RH problem might lead to a genuinely new
Lax pair for the Hastings-McLeod solution to Painlev\'e~II. See also the paper
\cite{JKT} for Lax pairs with matrices of size $3\times 3$.

As a consequence, we are able to show that the results in \cite{DelKui1} on
the recurrence coefficients remain valid at the critical time $t=t_{\crit}$,
i.e., the asymptotic behavior of the recurrence coefficients of the multiple
Hermite polynomials is still governed by the Hastings-McLeod solution to
Painlev\'e~II with exactly the same formulas as in \cite{DelKui1}. We find
this a surprising fact.

Very recently, a model of non-intersecting random walks was studied by
Adler, Ferrari and Van Moerbeke \cite{AFV2} in a situation that is
very similar to ours. There are two groups or random walks in \cite{AFV2},
that in the  scaling limit fill out two domains that are tangent in one point.
It is shown that there is a limiting correlation kernel  at the tacnode,
and two expressions for it are given in terms of multiple integrals
involving  Airy functions and the Airy kernel resolvent.
It seems very likely that this limiting kernel should be equivalent to the one
that we obtain in Theorem~\ref{theorem:kernelpsi} below for the symmetric case
(i.e., $p_1^* = p_2^*= 1/2$ in Theorem~\ref{theorem:kernelpsi}), but we have
not been able to make this identification.

It would indeed be interesting to see how the Painlev\'e~II equation arises in the
framework of \cite{AFV2}, and conversely, to see how our formula can be reduced to
integrals with Airy functions and the Airy kernel resolvent.
It would also be interesting to have a process version with an extended tacnode kernel.

\section{Statement of results}
\label{section:modelRHP}

We now give a precise statement of our results.
In Sections~\ref{subsection:RHintro}-\ref{subsection:doublescaling} we
describe our situation and the connection with a $4 \times 4$ matrix valued
RH problem. In Section~\ref{subsection:modelrhp} we formulate the the new
$4\times 4$ RH problem and its properties, in particular the connection with
the Painlev\'e II equation. Finally, in  Sections~\ref{subsection:kernellim}--\ref{subsection:additional}
we state the main results about the  limiting behavior
of the correlation kernels and recurrence coefficients.

\subsection{Correlation kernel and the Riemann-Hilbert problem}
\label{subsection:RHintro} We consider $n$ one-dimensional non-intersecting
Brownian motions with two starting points $a_1>a_2$ at time $t=0$ and two
ending points $b_1>b_2$ at time $t=1$. We assume that $n_1$ of the particles
move from the topmost starting point $a_1$ to the topmost ending point $b_1$,
and $n_2$ particles move from the bottommost starting point $a_2$ to the
bottommost ending point $b_2$, with $n_1+n_2=n$.

The transition probability density of the Brownian motions is
\begin{equation}\label{transitionprob}
    P_N(t,x,y) = \frac{1}{\sqrt{2\pi t} \, \sigma_N} \exp\left(- \frac{1}{2t \sigma_N^2}(x-y)^2\right),
    \qquad \sigma_N^2 = 1/N,
\end{equation}
with an overall variance $\sigma_N^2 = 1/N$, that decreases as $n$
increases such that
\begin{align}
\label{defnN}
    T =  n \sigma_N^2 = \frac{n}{N}  > 0
    \end{align}
remains fixed. We interpret $T$ as a temperature variable.

Using the above setting, it is known that the positions $x_1,\ldots,x_n$ of the
Brownian paths at time $t \in (0,1)$ have a joint probability density
\cite{KMcG}
\begin{equation}\label{jpdf} p(x_1, \ldots, x_n) = \frac{1}{Z_n} \det \left(
f_i(x_j) \right)_{i,j=1}^n \,
    \det \left( g_i(x_j) \right)_{i,j=1}^n \end{equation}
with functions
\begin{align}
    f_i(x) & =  \frac{\partial^{i-1}}{\partial x^{i-1}} P_N(t, a_1, x),
    \qquad i = 1, \ldots, n_1, \\
    f_{n_1+i}(x) & = \frac{\partial^{i-1}}{\partial x^{i-1}} P_N(t, a_2, x),
    \qquad i = 1, \ldots, n_2, \\
    g_i(x) & =  \frac{\partial^{i-1}}{\partial x^{i-1}} P_N(1-t, x, b_1),
    \qquad i = 1, \ldots, n_1, \\
    g_{n_1+i}(x) & = \frac{\partial^{i-1}}{\partial x^{i-1}} P_N(1-t, x, b_2),
    \qquad i = 1, \ldots, n_2,
    \end{align}
and with $Z_n$ a normalization constant. Note that \eqref{jpdf} is a
biorthogonal ensemble \cite{Bor}. In particular, it is a determinantal point
process with correlation kernel
\begin{equation} \label{kernel:doublesum}
    K_n(x,y) = \sum_{i,j=1}^n f_i(x) \left( A^{-1} \right)_{i,j} g_j(y), \end{equation}
where $\left( A^{-1} \right)_{i,j}$ denotes the $(i,j)$th entry of the inverse of the
matrix
\[ A = \left( \int_{\mathbb R} f_i(x) g_j(x) dx \right)_{i,j=1}^n. \]

The double sum formula \eqref{kernel:doublesum} for the kernel is not very
tractable for asymptotic analysis. However there is a convenient representation
of $K_n$ in terms of the solution of a Riemann-Hilbert problem. It is of size
$4 \times 4$, since $4$ is the total number of starting and ending positions.

Define the weight functions
\begin{align}
\label{gaussianw1} w_{1,k}(x) & = \exp\left(-\frac{n}{2Tt}(x^2-2a_kx)\right), &&  k=1,2, \\
\label{gaussianw2} w_{2,l}(x) & =
\exp\left(-\frac{n}{2T(1-t)}(x^2-2b_lx)\right), &&  l=1,2,
\end{align}
and consider the following RH problem which was introduced in
\cite{DK2} as a generalization of the RH problem for orthogonal polynomials
\cite{FIK}, see also \cite{VAGK}.

\begin{rhp} \label{rhp:Y} We look for a $4\times 4$ matrix-valued function
$Y : \mathbb C \setminus \mathbb R \to \mathbb C^{4 \times 4}$
satisfying
\begin{enumerate}
\item[\rm (1)] $Y(z)$ is analytic for $z\in\mathbb C\setminus\mathbb R$.
\item[\rm (2)] $Y$ has limiting values $Y_{\pm}$ on $\mathbb R$,
where $Y_+$ ($Y_-$) denotes the limiting value from the upper
(lower) half-plane, and
\begin{equation}\label{defjumpmatrix0}
Y_{+}(x) = Y_{-}(x)
\begin{pmatrix} I_2 & W(x)\\
0 & I_2
\end{pmatrix}, \qquad  x \in \mathbb R,
\end{equation}
where $I_2$ denotes the $2 \times 2$ identity matrix, and  $W(x)$ is
the rank-one matrix (outer product of two vectors)
\begin{equation}\label{defWblock0}
W(x) = \begin{pmatrix} w_{1,1}(x) \\ w_{1,2}(x)
\end{pmatrix}\begin{pmatrix} w_{2,1}(x) & w_{2,2}(x)
\end{pmatrix}.
\end{equation}
\item[\rm (3)] As $z\to\infty$, we have that
\begin{equation}\label{asymptoticconditionY0}
    Y(z) = \left(I+\frac{Y_1}{z}+O\left(\frac{1}{z^2}\right)\right)
    \diag(z^{n_1},z^{n_2},z^{-n_1},z^{-n_2}).
\end{equation}
\end{enumerate}
\end{rhp}

The RH problem has a unique solution that can be described in terms of certain
multiple orthogonal polynomials (actually multiple Hermite polynomials of mixed
type), see \cite{DK2,DelKui1} for  details. According to \cite{DK2},
the correlation kernel $K_n$ from \eqref{kernel:doublesum} is equal to
\begin{equation} \label{correlationkernel}
    K_n(x,y) = \frac{1}{2\pi i(x-y)}\begin{pmatrix} 0 & 0 & w_{2,1}(y) &
    w_{2,2}(y)\end{pmatrix} Y_{+}^{-1}(y)Y_{+}(x)\begin{pmatrix} w_{1,1}(x)\\
    w_{1,2}(x)\\ 0 \\ 0 \end{pmatrix}.
\end{equation}

The kernel $K_n$ in \eqref{correlationkernel} not only depends on $n$, but also
on $n_1$, $n_2$, $a_1,a_2$, $b_1, b_2$, $t$ and $T$. In what follows, we will
take the variables $T$ and $t$ fixed while the other variables will be
varying with $n$, see Section~\ref{subsection:doublescaling}.

\subsection{Separation of the starting and ending points}
\label{subsection:sepintro}

We now discuss in more detail the three situations in Figure~\ref{fig:3cases}.
Given $n_1$ and $n_2$, we denote the corresponding fractions of particles by
\begin{align} \label{defp1p2}
    p_1 := \frac{n_1}{n}, \qquad p_2 := \frac{n_2}{n} = 1-p_1,
\end{align}
which are varying with $n$, and we assume that
\begin{align}
\label{defp1}
    p_1  = p_1^*+O(1/n), \qquad p_2 = p_2^* + O(1/n), \qquad n\to\infty,
\end{align}
for certain limiting values $p_1^*,p_2^* \in (0,1)$. Of course, $p_2^* =
1-p_1^*$.

For given $T$ and $p_1^*, p_2^*$, the three cases of large, small and critical
separation of the starting and ending points are distinguished as follows; see
\cite{DelKui1}. There is large separation in case the inequality
\begin{equation}\label{largeseparation}
    (a_1-a_2)(b_1-b_2) > T \left(\sqrt{p_1^*}+\sqrt{p_2^*}\right)^2
\end{equation}
holds. Then the Brownian motion paths remain in two separate groups, and the
limiting hull in the $tx$-plane consists of two ellipses. For any $t\in(0,1)$,
the limiting distribution of the positions of the paths at time $t$ is
supported on the two disjoint intervals $[\alpha_1^*,\beta_1^*]$ and
$[\alpha_2^*,\beta_2^*]$, where the endpoints satisfy
\begin{equation}\label{chain:ineq}\alpha_2^*<\beta_2^*<\alpha_1^*<\beta_1^*\end{equation} and are given
explicitly by
\begin{align} \label{def:alphaj}
    \alpha_j^* =  (1-t)a_j+tb_j- 2\sqrt{p_j^* Tt(1-t)},\\
    \label{def:betaj}
    \beta_j^* = (1-t)a_j+tb_j+ 2\sqrt{p_j^* Tt(1-t)},
\end{align}
with the limiting density on these intervals given by the semicircle
laws
\begin{equation}\label{wignerdensity}
  \frac{1}{2\pi Tt(1-t)}\sqrt{(\beta_j^*(t)-x)(x-\alpha_j^*(t))},\qquad
   x\in[\alpha_j^*(t),\beta_j^*(t)],
\end{equation}
for $j=1,2$. See again Figure~\ref{fig:3cases}(a).

We are in the case of small separation if the inequality
\eqref{largeseparation} is reversed, so that
\begin{equation}\label{smallseparation}
    (a_1-a_2)(b_1-b_2) < T \left(\sqrt{p_1^*}+\sqrt{p_2^*}\right)^2.
\end{equation}
Then the limiting hull in the $tx$-plane is an algebraic curve and the limiting
distribution of the paths at any given time $t \in (0,1)$ is not given by
semicircle laws anymore. This was shown in \cite{DKV} for the case where
$p_1^*=p_2^*=1/2$. See Figure~\ref{fig:3cases}(b).

The case of critical separation corresponds to the equality sign in
\eqref{largeseparation}, that is,
\begin{equation}\label{criticalseparation}
    (a_1-a_2)(b_1-b_2) = T \left(\sqrt{p_1^*}+\sqrt{p_2^*}\right)^2.
\end{equation}
Then the limiting distribution is still given by two semicircle
densities \eqref{wignerdensity}, which are now however touching at
the critical time $t_{\crit}\in(0,1)$ defined by
\begin{equation}\label{tcrit}
    t_{\crit} = \frac{a_1-a_2}{(a_1-a_2)+(b_1-b_2)},
\end{equation}
in the sense that $\beta_2^*=\alpha_1^*$ in \eqref{chain:ineq} when
$t=t_{\crit}$. See Figure~\ref{fig:3cases}(c).

\subsection{Double scaling limit}
\label{subsection:doublescaling}

In order to study the case of critical separation we choose without
loss of generality
\begin{equation}\label{fixed tem}
 T = 1,
\end{equation}
and take $a_1^* > a_2^*$, $b_1^*
> b_2^*$ so that
\begin{equation}  \label{criticalseparation:tris}
    (a_1^*-a_2^*)(b_1^*-b_2^*) = \left(\sqrt{p_1^*}+\sqrt{p_2^*}\right)^2,
\end{equation}
cf.~\eqref{criticalseparation}. We also take
\begin{equation}\label{tcrit:bis}
    t=t_{\crit} = \frac{a_1^*-a_2^*}{(a_1^*-a_2^*)+(b_1^*-b_2^*)}
\end{equation}
and
\begin{equation} \label{xcrit}
    x_{\crit} =  (1-t_{\crit}) \frac{a_1^* + a_2^*}{2} + t_{\crit} \frac{ b_1^* + b_2^*}{2}
\end{equation}
as the critical time and place of tangency.

We can assume that $T$ and $t$ remain fixed as in \eqref{fixed tem} and
\eqref{tcrit:bis} without loss of generality, since we cannot create
different limiting behavior by varying $T$ and $t$ as well, see
Remarks~\ref{remark:varying:T} and \ref{remark:varying:t} in
Section~\ref{subsection:kernellim}.

We let the starting points $a_1$, $a_2$ and ending points $b_1$, $b_2$ be
varying with $n$ in such a way that for certain real constants $L_1, L_2, L_3,
L_4$,
\begin{align}
    \label{doublescaling:a} a_{1} &= a_1^* + L_1 n^{-2/3}, \quad &
    b_{1} &= b_1^* + L_3 n^{-2/3},\\
    \label{doublescaling:b}
    a_{2} &= a_2^* + L_2 n^{-2/3}, \quad &
     b_{2} &= b_2^* + L_4 n^{-2/3},
\end{align}
as $n \to \infty$. It turns out that the constants only appear in
our results through the combinations
\begin{align}
\label{doublescaling:L5} L_5 &:= (b_1^*-b_2^*)L_1+(a_1^*-a_2^*)L_3, \\
\label{doublescaling:L6} L_6 &:= (b_1^*-b_2^*)L_2+(a_1^*-a_2^*)L_4.
\end{align}

Near the critical value $x_{\crit}$, the asymptotics of $K_n$ for $n\to\infty$
will be described by a family of limiting kernels
\[ K^{tacnode}(u, v; r_1, r_2, s_1, s_2), \qquad u, v \in \mathbb R, \]
depending on four variables $r_1, r_2 > 0$ and $s_1, s_2 \in \mathbb
R$ that depend on the values $p_1^*$, $p_2^*$, $L_5$ and $L_6$.
Because of a dilation and translation symmetry
\[ c^2 K^{tacnode}(c^2u,c^2 v; r_1, r_2, s_1, s_2)
    = K^{tacnode}(u,v; c^3 r_1, c^3 r_2, c s_1, c s_2 ), \qquad c > 0, \]
\[ K^{tacnode}(u + 2c, v+ 2c; r_1, r_2, s_1, s_2)
    = K^{tacnode}(u,v; r_1,r_2,s_1 - c r_1, s_2 + cr_2),
    \qquad c \in \mathbb R, \]
the family essentially depends on two parameters only, and we could
for example choose $r_1 = 1$, $s_1 = s_2$. However, in order to
preserve the symmetry in the formulas, we prefer to use four
parameters.

The limiting kernels are given in terms of a RH problem that we
discuss next.

\subsection{A new $4\times 4$ Riemann-Hilbert problem}
\label{subsection:modelrhp}

The $4\times 4$ matrix-valued RH problem has jumps on a contour in
the complex plane consisting of $10$ rays emanating from the origin.
The rays are determined by two numbers $\varphi_1,\varphi_2$ such
that
\begin{equation}\label{def:angles}
0<\varphi_1<\varphi_2<\pi/3.
\end{equation}
The value $\pi/3$ in \eqref{def:angles} is not the optimal one, but
it will be sufficient for our purposes. We define the half-lines
$\Gamma_k$, $k=0,\ldots,9$, by
\begin{equation}\label{def:rays1}
\begin{aligned}
    \Gamma_0  =\mathbb R^+, \qquad \Gamma_1 & =e^{i\varphi_1}\mathbb R^+,
    && \qquad
    \Gamma_2 =e^{i\varphi_2}\mathbb R^+, \\
    \Gamma_3 & =e^{i(\pi-\varphi_2)}\mathbb R^+, &&\qquad
\Gamma_4=e^{i(\pi-\varphi_1)}\mathbb R^+,
\end{aligned}
\end{equation}
and
\begin{equation}\label{def:rays2}
\Gamma_{5+k}=-\Gamma_k,\qquad k=0,\ldots,4.
\end{equation}
All rays are oriented towards infinity, as shown in
Figure~\ref{fig:modelRHP}.

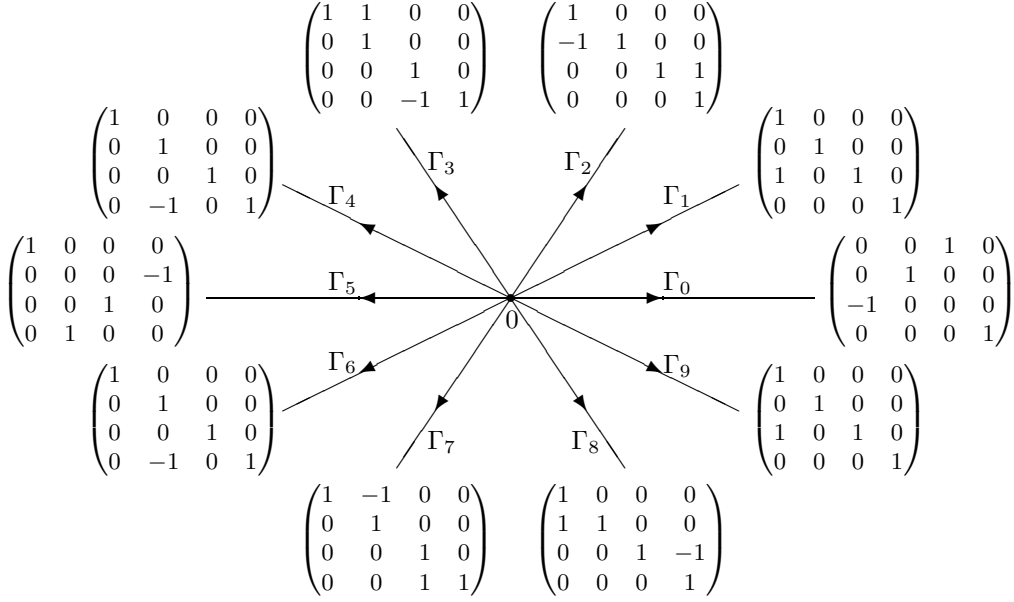
\begin{figure}[t]
\vspace{14mm}
\begin{center}
   \setlength{\unitlength}{1truemm}
   \begin{picture}(100,70)(-5,2)
       \put(40,40){\line(1,0){40}}
       \put(40,40){\line(-1,0){40}}
       \put(40,40){\line(2,1){30}}
       \put(40,40){\line(2,-1){30}}
       \put(40,40){\line(-2,1){30}}
       \put(40,40){\line(-2,-1){30}}
       \put(40,40){\line(2,3){15}}
       \put(40,40){\line(2,-3){15}}
       \put(40,40){\line(-2,3){15}}
       \put(40,40){\line(-2,-3){15}}
       \put(40,40){\thicklines\circle*{1}}
       \put(39.3,36){$0$}
       \put(60,40){\thicklines\vector(1,0){.0001}}
       \put(20,40){\thicklines\vector(-1,0){.0001}}
       \put(60,50){\thicklines\vector(2,1){.0001}}
       \put(60,30){\thicklines\vector(2,-1){.0001}}
       \put(20,50){\thicklines\vector(-2,1){.0001}}
       \put(20,30){\thicklines\vector(-2,-1){.0001}}
       \put(50,55){\thicklines\vector(2,3){.0001}}
       \put(50,25){\thicklines\vector(2,-3){.0001}}
       \put(30,55){\thicklines\vector(-2,3){.0001}}
       \put(30,25){\thicklines\vector(-2,-3){.0001}}

       \put(60,41){$\Gamma_0$}
       \put(60,52.5){$\Gamma_1$}
       \put(47,57){$\Gamma_2$}
       \put(29,57){$\Gamma_3$}
       \put(16,52.5){$\Gamma_4$}
       \put(16,41){$\Gamma_5$}
       \put(16,30.5){$\Gamma_6$}
       \put(29,20){$\Gamma_7$}
       \put(48,20){$\Gamma_8$}
       \put(60,30){$\Gamma_{9}$}

       \put(81,40){$\small{\begin{pmatrix}0&0&1&0\\ 0&1&0&0\\ -1&0&0&0\\ 0&0&0&1 \end{pmatrix}}$}
       \put(71.5,57){$\small{\begin{pmatrix}1&0&0&0\\ 0&1&0&0\\ 1&0&1&0\\ 0&0&0&1 \end{pmatrix}}$}
       \put(43,71){$\small{\begin{pmatrix}1&0&0&0\\ -1&1&0&0\\ 0&0&1&1\\ 0&0&0&1 \end{pmatrix}}$}
       \put(12,71){$\small{\begin{pmatrix}1&1&0&0\\ 0&1&0&0\\ 0&0&1&0\\ 0&0&-1&1 \end{pmatrix}}$}
       \put(-16,57){$\small{\begin{pmatrix}1&0&0&0\\ 0&1&0&0\\ 0&0&1&0\\ 0&-1&0&1 \end{pmatrix}}$}
       \put(-27,40){$\small{\begin{pmatrix}1&0&0&0\\ 0&0&0&-1\\ 0&0&1&0\\ 0&1&0&0 \end{pmatrix}}$}
       \put(-16,23){$\small{\begin{pmatrix}1&0&0&0\\ 0&1&0&0\\ 0&0&1&0\\ 0&-1&0&1 \end{pmatrix}}$}
       \put(12,7){$\small{\begin{pmatrix}1&-1&0&0\\ 0&1&0&0\\ 0&0&1&0\\ 0&0&1&1 \end{pmatrix}}$}
       \put(43,7){$\small{\begin{pmatrix}1&0&0&0\\ 1&1&0&0\\ 0&0&1&-1\\ 0&0&0&1 \end{pmatrix}}$}
       \put(71.5,23){$\small{\begin{pmatrix}1&0&0&0\\ 0&1&0&0\\ 1&0&1&0\\ 0&0&0&1 \end{pmatrix}}$}

  \end{picture}
   \vspace{0mm}
   \caption{The figure shows the jump contours $\Gamma_k$ in the complex $\zeta$-plane and the corresponding jump matrices
   $J_k$, $k=0,\ldots,9$, in the RH problem for  $M = M(\zeta)$.}
   \label{fig:modelRHP}
\end{center}
\end{figure}

\begin{rhp}\label{rhp:modelM}
We look for a $4\times 4$ matrix-valued function $M(\zeta)$ (which
also depends parametrically on the parameters $r_1, r_2 > 0$ and
$s_1, s_2 \in \mathbb C$) satisfying
\begin{enumerate}
\item[\rm (1)] $M(\zeta)$ is analytic for $\zeta\in\mathbb C\setminus\left(\bigcup_{k=0}^{9}
\Gamma_k\right)$.
\item[\rm (2)] For $\zeta\in\Gamma_k$, the limiting values
\[ M_+(\zeta) = \lim_{z \to \zeta, \,  z\textrm{ on $+$-side of }\Gamma_k} M(z), \qquad
    M_-(\zeta) = \lim_{z \to \zeta, \, z\textrm{ on $-$-side of }\Gamma_k} M(z) \]
exist, where the $+$-side and $-$-side of $\Gamma_k$ are the sides
which lie on the left and right of $\Gamma_k$, respectively, when
traversing $\Gamma_k$ according to its orientation. These limiting
values satisfy the jump relation
\begin{equation}\label{jumps:M}
    M_{+}(\zeta) = M_-(\zeta)J_k,\qquad \zeta \in \Gamma_k, \quad k=0,\ldots,9,
\end{equation}
where the jump matrices $J_k$ are shown in
Figure~\ref{fig:modelRHP}.
\item[\rm (3)] As $\zeta\to\infty$, we have that
\begin{multline}
\label{M:asymptotics} M(\zeta) =
\left(I+\frac{M_1}{\zeta}+\frac{M_2}{\zeta^2}+O\left(\frac{1}{\zeta^3}\right)\right)
\diag((-\zeta)^{-1/4},\zeta^{-1/4},(-\zeta)^{1/4},\zeta^{1/4})
\\ \times \frac{1}{\sqrt{2}}
\begin{pmatrix} 1 & 0 & -i & 0 \\
0 & 1 & 0 & i \\
-i & 0 & 1 & 0 \\
0 & i & 0 & 1 \\
\end{pmatrix}
\diag\left(e^{-\theta_1(\zeta)},e^{-\theta_2(\zeta)},e^{\theta_1(\zeta)},e^{\theta_2(\zeta)}\right),
\end{multline}
where the coefficient matrices $M_1,M_2,\ldots$ depend on the
parameters $r_1, r_2, s_1, s_2$, but not on $\zeta$, and where we
define
\begin{equation}\label{def:theta1}
\theta_1(\zeta) = \frac{2}{3} r_1 (-\zeta)^{3/2}+2s_1 (-\zeta)^{1/2}
\end{equation}
and
\begin{equation}\label{def:theta2}
\theta_2(\zeta) = \frac{2}{3} r_2 \zeta^{3/2}+2s_2 \zeta^{1/2}.
\end{equation}
\item[\rm (4)] $M(\zeta)$ is bounded near $\zeta=0$.
\end{enumerate}
\end{rhp}

In \eqref{M:asymptotics}--\eqref{def:theta2}, we use the principal
branches of the fractional powers, so that for example
$(-\zeta)^{1/4}$ is defined and analytic for $\zeta \in \mathbb C
\setminus [0,\infty)$ with real and positive values on $\mathbb
R^-$. We write
\[ M(\zeta; r_1, r_2, s_1, s_2), \]
in case we want to emphasize the dependence of $M$ on the
parameters.

It follows from standard arguments (e.g.~\cite{Dei}) that the
solution to the RH problem~\ref{rhp:modelM} is unique if it exists.
The existence issue is our first main theorem.

\begin{theorem}\label{theorem:solvability} (Existence:)
Assume that $r_1, r_2>0$ and $s_1, s_2 \in \mathbb R$. Then the RH
problem~{\rm\ref{rhp:modelM}} for $M(\zeta)$ is uniquely solvable.

\end{theorem}

Theorem~\ref{theorem:solvability} will be proved in
Section~\ref{section:vanishing} by using the technique of a
vanishing lemma \cite{DKMVZ1,FZ,Zhou}.

Our next result is about the residue matrix $M_1$ in
\eqref{M:asymptotics}. We show that its top right $2\times 2$ block
is related to the Hastings-McLeod solution $q(s)$ of the
Painlev\'e~II equation and the associated Hamiltonian $u(s)$.

\begin{theorem}\label{theorem:Painleve2modelrhp}
($M_1$ vs.~the Painlev\'e~II equation:) Let the parameters $r_1, r_2>0$ and
$s_1, s_2 \in \mathbb R$ in \eqref{M:asymptotics}--\eqref{def:theta2} be fixed.
Then the matrix $M_1$ in \eqref{M:asymptotics} can be written as
\begin{equation} \label{eq:defM1}
    M_1 = \begin{pmatrix} a & b & ic & id \\ - \wtil b & - \wtil a & id & i \wtil c \\
    i e & i f & -a & \wtil b \\ if & i\wtil e & -b & \wtil a
    \end{pmatrix}
    \end{equation}
for certain real valued numbers $a, \wtil a, b, \wtil b,c,\wtil c, d,e,\wtil
e,f$ that depend on $r_1, r_2, s_1, s_2$. In addition, we have
\begin{align}
\label{d:Painleve2} d &= \frac{(r_1 r_2)^{1/6}}{(r_1^2 +
r_2^2)^{1/3}} q\left(\sigma\right),
\\
\label{c:Hamiltonian} c &= -\frac{r_2^{2/3}}{r_1^{1/3} (r_1^2+
r_2^2)^{1/3}} u(\sigma)+ \frac{s_1^2}{r_1},
\\
\label{ctil:Hamiltonian} \wtil c &=
-\frac{r_1^{2/3}}{r_2^{1/3}(r_1^2+r_2^2)^{1/3}}u(\sigma)+\frac{s_2^2}{r_2},
\end{align}
where $q(s)$ is the Hastings-McLeod solution of the Painlev\'e~II
equation \eqref{def:Painleve2}--\eqref{def:HastingMcLeod}, $u(s)$ is
the Hamiltonian
\begin{equation} \label{def:Hamiltonian}
    u(s) := \left( q'(s)\right)^2 - s q^2(s) - q^4(s)
    \end{equation}
and
\begin{equation}\label{def:sigma}
    \sigma:=\frac{2(r_1 s_2 + r_2 s_1)}{(r_1 r_2)^{1/3}(r_1^2 + r_2^2)^{1/3}}.
\end{equation}
\end{theorem}

Theorem~\ref{theorem:Painleve2modelrhp} shows that the top right $2\times 2$
block in \eqref{eq:defM1} behaves in a similar way as the residue matrix in the
classical $2\times2$ matrix-valued RH problem for Painlev\'e~II due to Flaschka and Newell
\cite{FN,FIKN}. We do not think that our $4\times 4$ RH problem can be reduced
to the $2\times 2$ problem.

The proof of Theorem~\ref{theorem:Painleve2modelrhp} is given in
Section~\ref{section:proofP2M}.

\begin{remark} \label{rem:Laxpair}
In Section~\ref{section:proofP2M} we also find identities relating the
entries in the top left block of $M_1$ to the entries in the top right block in \eqref{eq:defM1},
namely
\[ (2a+c^2)r_1 = r_2 d^2 + s_1, \quad (2\wtil{a} + \wtil{c}^2) r_2 = r_1 d^2 + s_2, \quad
    \frac{\partial d}{\partial s_1} = 2 cd - 2 \wtil b = \frac{r_2}{r_1} (2 \wtil c d - 2 b) \]
see \eqref{threeextrarelations:1}, \eqref{threeextrarelations:2}, and \eqref{eq:compat4}.
Hence also $a, \wtil a, b$ and $\wtil b$ can be directly expressed in terms
of $q(s)$.

We also obtain the Lax pair equations
\begin{align}
    \frac{\partial M}{\partial \zeta} = U M, \qquad
    \frac{\partial M}{\partial s_j} = V_j M, \quad j =1,2
    \end{align}
    with matrices $U$ and $V_j$ that are explicitly given in terms of $M_1$,
    see Propositions \ref{prop:diffeq1} and \ref{prop:diffeq2}. In fact the matrices
    only depend on the entries $a, \wtil{a}, b, \wtil{b}, c, \wtil{c}$ and $d$
    of $M_1$, and so $U$ and $V_j$ depend  only on the Hastings-McLeod solution
    $q(s)$ of Painlev\'e II.

In terms of $U$ one may therefore view $M$ as
a solution of $\frac{\partial M}{\partial \zeta} = UM$ in each sector
 with the asymptotic behavior \eqref{M:asymptotics} in that sector.
\end{remark}

\subsection{Critical limit of correlation kernel}
\label{subsection:kernellim} The limiting kernels are defined in terms of the
solution of the model RH problem \ref{rhp:modelM} as follows.
Let $M(\zeta; r_1, r_2, s_1, s_2)$ be the solution of the RH problem
\ref{rhp:modelM} for fixed $r_1, r_2 > 0$ and $s_1, s_2 \in \mathbb R$. Thus,
$M$ is analytic in each of the sectors determined by the contours $\Gamma_k$,
and the restriction to one such sector has an analytic continuation to the
entire complex $\zeta$ plane. In other words, the entries of $M$ are entire
functions. Consider the sector around the positive imaginary axis, bounded by
$\Gamma_2$ and $\Gamma_3$; we denote the analytic continuation of the
restriction of $M$ to this sector by $\widehat M$.

\begin{definition} For $u, v \in \mathbb
R$, the kernel $K^{tacnode}(u,v; r_1, r_2, s_1, s_2)$ is defined by
\begin{multline} \label{eq:tacnode kernel}
    K^{tacnode}(u,v; r_1, r_2, s_1, s_2)  \\
     = \frac{1}{2\pi i(u-v)} \begin{pmatrix} 0 & 0 & 1 & 1 \end{pmatrix}
    \widehat M^{-1}(u; r_1, r_2, s_1, s_2)
    \widehat M (v; r_1, r_2, s_1, s_2) \begin{pmatrix} 1 \\ 1 \\ 0 \\ 0 \end{pmatrix}.
\end{multline}
\end{definition}

We can rewrite the kernel in terms of the limiting values $M_+(u)$ and $M_+(v)$
of $M$ on the real axis, by using the jump relations in the RH problem for $M$.
For example, for $u, v > 0$, we have
\begin{equation}\label{tacnodekernel:pos}
    K^{tacnode}(u,v; r_1, r_2, s_1, s_2)
    = \frac{1}{2\pi i(u-v)} \begin{pmatrix} -1 & 0 & 1 & 0 \end{pmatrix}
     M_+^{-1}(u)  M_+(v) \begin{pmatrix} 1 \\ 0 \\ 1 \\ 0 \end{pmatrix},
\end{equation}
with a different expression in case $u$ and/or $v$ are negative.
Here, we have dropped the dependence of $M$ on the parameters $r_1,
r_2, s_1, s_2$.

From \eqref{eq:tacnode kernel}, it is also easily seen that the
kernel has the integrable form
\[ \frac{f_1(u) g_1(v) + f_2(u) g_2(v) + f_3(u) g_3(v) + f_4(u) g_4(v)}{u-v}, \]
for certain entire functions $f_j$ and $g_j$, $j=1, \ldots,4$, with
$ \sum_{j=1}^4 f_j(u) g_j(u) = 0$.
The functions of course depend on $r_1, r_2, s_1, s_2$.

The following is the main theorem of this paper.

\begin{theorem} \label{theorem:kernelpsi}
(Correlation kernel at the tacnode:) Consider $n$
non-intersecting Brownian motions on $[0,1]$ with transition
probability density \eqref{transitionprob}  with two given starting
points $a_1 > a_2$ and two given endpoints $b_1 > b_2$. Suppose
$n_1$ paths start in $a_1$ and end in $b_1$, and $n_2 = n-n_1$ paths
start in $a_2$ and end in $b_2$. Assume that $T=1$ and that
\eqref{defp1}, \eqref{doublescaling:a}--\eqref{doublescaling:b} hold
with values $p_1^*, p_2^*, a_1^*, a_2^*, b_1^*, b_2^*$ such that
\eqref{criticalseparation:tris} holds.

Let $t_{\crit}$ and $x_{\crit}$ be the critical time and place of tangency as given
in \eqref{tcrit:bis} and \eqref{xcrit}, respectively. Then the correlation
kernels $K_n$ for the positions of the paths at the critical time $t_{\crit}$
satisfy
\begin{multline} \label{kernel at tacnode}
    \lim_{n \to \infty} \frac{\exp\left(c_2 n^{1/3} (u-v) \right)}{c n^{2/3}}
    K_n \left( x_{\crit} +  \frac{u}{cn^{2/3}}, x_{\crit} + \frac{v}{cn^{2/3}}\right) \\
        = K^{tacnode}(u,v; r_1, r_2, s_1, s_2),
\end{multline}
with $K^{tacnode}$ given by \eqref{eq:tacnode kernel} and
\begin{align}
\label{doublescaling:r} r_1 & = \left(p_1^*\right)^{1/4}, & r_2 & = \left(p_2^*\right)^{1/4}, \\
\label{doublescaling:s} s_1 &=
\frac{(p_1^*)^{1/4}}{2(\sqrt{p_1^*}+\sqrt{p_2^*})} L_5, &
     s_2 &= - \frac{(p_2^*)^{1/4}}{2(\sqrt{p_1^*} + \sqrt{p_2^*})} L_6,
\end{align}
and the constants $c$ and $c_2$ in \eqref{kernel at tacnode} are given by
\begin{align} \label{eq:c}
    c & =   \frac{1}{\sqrt{t_{\crit}(1-t_{\crit})}}, \\
    \label{eq:c2}
    c_2  & := c \left[- (1-t_{\crit}) \frac{a_1^*}{2}  + t_{\crit} \frac{b_1^*}{2} \right] =
        c \left[-(1-t_{\crit}) \frac{a_2^*}{2} +  t_{\crit} \frac{b_2^*}{2}
        \right].
    \end{align}
\end{theorem}

Theorem~\ref{theorem:kernelpsi} will be proved in
Section~\ref{subsection:proofkernel}.

\begin{remark}
The extra factor $\exp(c_2 n^{1/3}(u-v))$ in \eqref{kernel at tacnode} is irrelevant
as it does not change any of the determinantal correlation
functions $\det \left( K_n(x_i, x_j) \right)$ of the determinanal
point process.
\end{remark}

\begin{remark}
In view of \eqref{criticalseparation:tris}, \eqref{doublescaling:L5} and
\eqref{doublescaling:L6}, we may rewrite \eqref{doublescaling:s} as
\begin{align*}
 s_1 = \frac{r_1}{2} \left( \sqrt{\frac{b_1^*-b_2^*}{a_1^* - a_2^*}} L_1
    +  \sqrt{\frac{a_1^*-a_2^*}{b_1^* - b_2^*}} L_3 \right), \quad
 s_2 = - \frac{r_2}{2} \left( \sqrt{\frac{b_1^*-b_2^*}{a_1^* - a_2^*}} L_2
    +  \sqrt{\frac{a_1^*-a_2^*}{b_1^* - b_2^*}} L_4 \right).
\end{align*}
\end{remark}

\begin{remark}
With the values \eqref{doublescaling:r}--\eqref{doublescaling:s}, the parameter
$\sigma$ of \eqref{def:sigma} that is the argument of the Painlev\'e II
transcendent is related to $p_1^*$, $p_2^*$, $L_5$ and $L_6$ by
\begin{equation}\label{doublescaling:sigma}
\sigma = \frac{(p_1^* p_2^*)^{1/6}}{(\sqrt{p_1^*} +
\sqrt{p_2^*})^{4/3}} \left(L_5 - L_6\right).
\end{equation}
\end{remark}

\begin{remark}\label{remark:varying:T}
Instead of varying the endpoints, we could have used the temperature
$T$ as a scaling parameter with the same scale $O(n^{-2/3})$. If we
keep $a_1, a_2, b_1, b_2$ fixed at the values $a_1^*, a_2^*, b_1^*, b_2^*$
and vary $T$ with $n$ such that
\[ T = 1 - L_7 n^{-2/3} + o(n^{-2/3}) \qquad \text{as } n \to \infty, \]
then it can be checked that the same local behavior at the
tacnode can be created with a change
in the endpoints while keeping the temperature fixed. Indeed, to
that end, we vary  the endpoints as in
\eqref{doublescaling:a}--\eqref{doublescaling:b} with
\begin{equation}
    L_1  = \frac{1}{2} a_1^* L_7, \quad L_2  = \frac{1}{2} a_2^* L_7,  \quad  L_3  = \frac{1}{2} b_1^* L_7,
    \quad L_4  = \frac{1}{2} b_2^* L_7.
\end{equation}
Thus, by \eqref{doublescaling:L5} and \eqref{doublescaling:L6}, it follows that
\begin{align*}
    L_5 & =  \frac{1}{2} \left((b_1^* - b_2^*) a_1^* + (a_1^* - a_2^*) b_1^*\right) L_7, \\
    L_6 & =  \frac{1}{2} \left((b_1^* - b_2^*) a_2^* + (a_1^* - a_2^*) b_2^* \right) L_7,
\end{align*}
and
\[ L_5 - L_6 = (a_1^*-a_2^*)(b_1^* - b_2^*) L_7
    = (\sqrt{p_1^*} + \sqrt{p_2^*})^2 L_7, \]
so that
\[ \sigma = (p_1^* p_2^*)^{1/6}(\sqrt{p_1^*} + \sqrt{p_2^*})^{2/3} L_7, \]
with the aid of \eqref{doublescaling:sigma}. This expression for $\sigma$ is
compatible with the one in \cite{DelKui1}.
\end{remark}

\begin{remark}\label{remark:varying:t}
We could similarly vary the time $t$ around the critical time $t_{\crit}$ while
keeping other parameters fixed, say,
\[ t = t_{\crit} + L_8 n^{-2/3}+ o(n^{-2/3}). \]
This effect can  be modeled by varying the endpoints as in
\eqref{doublescaling:a}--\eqref{doublescaling:b} with $t = t_{\crit}$ and $T =
1$ fixed, with now
\begin{equation}
\begin{aligned}
    L_1 & = -\frac{a_1^*}{2t_{\crit}(1-t_{\crit})}  L_8, \qquad &L_3 & = \frac{b_1^*}{2t_{\crit}(1-t_{\crit})}  L_8, \\
    L_2 & = -\frac{a_2^*}{2t_{\crit}(1-t_{\crit})}  L_8, \qquad  &L_4 & = \frac{b_2^*}{2t_{\crit}(1-t_{\crit})} L_8.
\end{aligned}
\end{equation}
Hence $L_5 - L_6 = 0$ and so $\sigma = 0$ by \eqref{doublescaling:sigma}.
Thus the argument \eqref{def:sigma} of the Painlev\'e transcendent
does not change if we only vary $t$.
\end{remark}

\subsection{Additional results}
\label{subsection:additional}

\subsubsection{Semicircle densities}
\label{subsubsection:semicircles}

There are two other results that come out of our asymptotic analysis. The first
one is about the limiting density of  the non-intersecting Brownian paths.

\begin{theorem}\label{st2ellipses}
$($Touching semicircle densities:$)$ Consider the double scaling limit as
described in Section~\ref{subsection:doublescaling}. Then as $n \to\infty$, the
Brownian particles at the critical time $t$ in \eqref{tcrit:bis} are
asymptotically supported on the two touching intervals $[\alpha_1^*,\beta_1^*]$
and $[\alpha_2^*,\beta_2^*]$ \eqref{def:alphaj}--\eqref{def:betaj}, with
$\beta_2^*=\alpha_1^*$, and with limiting densities given by the semicircles
\eqref{wignerdensity}.
\end{theorem}

Theorem~\ref{st2ellipses} can be proved from the steepest descent analysis in
Section~\ref{section:steepestdescent} in quite the same way as in
\cite{DelKui1}. We will not go into the details.

\subsubsection{Critical limit of recurrence coefficients}
\label{subsubsection:reccrit}

The second result concerns the recurrence coefficients of the multiple Hermite
polynomials. In the critical limit, their behavior is governed by the
Hastings-McLeod solution of the Painlev\'e~II equation, in exactly the same way
as in \cite{DelKui1}. For the $4\times 4$ matrix $Y_1$ in
\eqref{asymptoticconditionY0}, the combinations $(Y_1)_{i,j}(Y_1)_{j,i}$,
$i,j=1,\ldots,4$ with $i\neq j$ are called the \lq off-diagonal recurrence
coefficients\rq\ in \cite{DelKui1}, where we use the notation $(M)_{i,j}$ to
denote the $(i,j)$th entry of any given matrix $M$. It is also shown in
\cite{DelKui1} that the quantities $(Y_1)_{i,j}(Y_1)_{j,i}$ with $i\neq j$ are
determined by $(Y_1)_{1,2}(Y_1)_{2,1}$ and $(Y_1)_{1,4}(Y_1)_{4,1}$ alone. We
then have the following result.

\begin{theorem}\label{theorem:Painleve2rec}
$($Asymptotics of off-diagonal recurrence coefficients:$)$ Consider the double
scaling limit as described in Section~{\rm\ref{subsection:doublescaling}}. Then we
have
\begin{align}
    \label{c12asymptotics}
    (Y_1)_{1,2}(Y_1)_{2,1} &= - K^2 t_{\crit}^2(b_1^*-b_2^*)^2q^2(\sigma)n^{-2/3}+O(n^{-1}),\\
    \label{c14asymptotics}
    (Y_1)_{1,4}(Y_1)_{4,1} &=  K^2  t_{\crit}(1-t_{\crit})(a_1^*-a_2^*)(b_1^*-b_2^*)q^2(\sigma)
    n^{-2/3}+O(n^{-1}),
\end{align}
as $n\to\infty$, where $q(s)$ is the Hastings-McLeod solution to the
Painlev\'e~II equation, $\sigma$ is given by \eqref{doublescaling:sigma} and
\begin{equation}\label{doublescaling:K}
K= \frac{(p_1^* p_2^*)^{1/6}}{(\sqrt{p_1^*} + \sqrt{p_2^*})^{4/3}}.
\end{equation}
\end{theorem}

Theorem~\ref{theorem:Painleve2rec} will be proved in
Section~\ref{subsection:proofrec}.

\begin{remark}
For ease of comparison, we have stated the above theorem in exactly the
same way as in \cite{DelKui1}. We note however that the expressions \eqref{c12asymptotics}
and \eqref{c14asymptotics} can be
simplified in the present case, since we are now looking exactly
at the critical time $t=t_{\crit}$. This means that the variable $t$ in the
above formulas can be substituted by \eqref{tcrit:bis}. Further simplification
can be obtained by substituting \eqref{criticalseparation:tris}.
\end{remark}

\begin{remark}
In \cite{DelKui1} there is also a result on the Painlev\'e~II asymptotics of
the so-called \lq diagonal recurrence coefficients\rq. One can show that
exactly the same result holds in the present case. This will be briefly
commented at the end of Section~\ref{subsection:proofrec}.
\end{remark}

\subsection{About the proofs}
\label{subsection:aboutproofs}

The rest of this paper contains the proofs of the theorems and is organized in
two parts: Part~I (Sections~\ref{section:Lun}--\ref{section:proofP2M}) and
Part~II
(Sections~\ref{section:xilamfunctions}--\ref{section:proofsmaintheorems}).
Part~I deals with the RH problem~\ref{rhp:modelM} for $M(\zeta)$. It consists
of three sections: In Section~\ref{section:Lun} we perform an asymptotic
analysis of this RH problem for $s_1\to +\infty$; in
Section~\ref{section:vanishing} we establish the existence result in
Theorem~\ref{theorem:solvability}; and in Section~\ref{section:proofP2M} we
prove the connection with the Painlev\'e~II equation in
Theorem~\ref{theorem:Painleve2modelrhp}.

Part~II deals with the critical asymptotics of the non-intersecting Brownian
motions at the tacnode.
The proofs of Theorems~\ref{theorem:kernelpsi} and \ref{theorem:Painleve2rec}
will be given in Section~\ref{section:proofsmaintheorems}. They are based on
the Deift-Zhou steepest descent analysis of the RH problem~\ref{rhp:Y}, which
will be discussed in Section~\ref{section:steepestdescent}. In the proofs an
important role is played by two modified equilibrium problems that give rise to
two functions $\xi_1$, $\xi_2$ and their anti-derivatives $\lambda_1$,
$\lambda_2$. The modified equilibrium problems will be discussed in
Section~\ref{section:xilamfunctions}.

The two parts are largely independent.
Both parts contain the steepest descent analysis of a $4 \times 4$ matrix valued RH problem,
that we give in some detail, although it makes the paper rather lengthy.
Some of our notation will have
different meanings in the two parts. For example, $V_1$ and $V_2$
are used to denote certain $4 \times 4$ matrices in part I, see \eqref{def:V1Lax} and \eqref{def:V2Lax},
while in part II they denote two external fields in an equilibrium problem, see \eqref{eq:V1V2}.
We trust that this will not lead to any confusion.

\part{Analysis of the Riemann-Hilbert problem for $M(\zeta)$}
\label{part:modelRHP}

\section{Asymptotic analysis of $M(\zeta)$ for $s_1\to +\infty$}
\label{section:Lun}

In this section we analyze the model RH problem \ref{rhp:modelM} for $M(\zeta)$
as $s_1 \to +\infty$. Our goal is twofold: we want to prove the solvability of
the RH problem for $s_1$ sufficiently large, and establish the large $s_1$
asymptotics for the quantities $c$, $\wtil c$ and $d$ in \eqref{eq:defM1}. The
goal of this section is to prove the following proposition.

\begin{proposition}\label{prop:large s1 solvability}
    Let $r_1, r_2 > 0$ and $s_2 \in \mathbb R$ be fixed.
    Then for large enough $s_1 \in \mathbb R$, the RH problem \ref{rhp:modelM} for
    $M(\zeta; r_1, r_2, s_1, s_2)$ is uniquely solvable.

    Moreover, we have
    \begin{align} \nonumber
    d & := - i \left(M_1(r_1, r_2, s_1, s_2)\right)_{1,4} \\
      & =
        \frac{1}{2\sqrt{\pi} s_1^{1/4}}
        \left(\frac{r_1}{2(r_1^2 + r_2^2)} \right)^{1/4}
        \exp \left(- \frac{2}{3} \sigma^{3/2} \right)
        (1 + O(s_1^{-3/2})), \label{eq:asy of d}
        \\
     c & := -i \left(M_1(r_1, r_2, s_1, s_2) \right)_{1,3} =
        \frac{s_1^2}{r_1} + O\left(s_1^{-1} \right), \label{eq:asy of c}\\
     \wtil c & := -i \left(M_1(r_1, r_2, s_1, s_2) \right)_{2,4} =
        \frac{s_2^2}{r_2} + O\left(s_1^{-1} \right),  \label{eq:asy of til c}
        \end{align}
        as $s_1 \to +\infty$,
        where $\sigma$ is given by \eqref{def:sigma}.
\end{proposition}

We will use Proposition~\ref{prop:large s1 solvability} in the proofs of
Theorems \ref{theorem:solvability} and \ref{theorem:Painleve2modelrhp}.

In the proof of Proposition~\ref{prop:large s1 solvability}, we will first show
that we can restrict ourselves to the case
\begin{equation}\label{eq:assumption}
r_1 = 1, \qquad s_2 = 0.
\end{equation}
This is due to certain dilation and translation symmetries for $M$. The further
analysis will then be based on the Deift-Zhou steepest descent method, using
ideas of \cite{DZ,IKO}. For the problem at hand, the analysis consists of a
series of transformations $M \mapsto A \mapsto B \mapsto C \mapsto D$, so that
the matrix-valued function $D$ uniformly tends to the identity matrix as $s_1
\to +\infty$.

\begin{remark}\label{remark:asymptotic:s}
Although we will perform the transformations for $s_1>0$ and
sufficiently large in what follows, it is also possible to apply an
asymptotic analysis as $s_1 \to -\infty$. But in that case the
`$g$-functions' will be more complicated: They must be constructed
from a $4$-sheeted Riemann surface of genus zero, which has a
somewhat similar flavor as the one in \cite{DKV}. We will not go
into the details.
\end{remark}

\subsection{Reduction to the case $r_1 = 1$ and $s_2 = 0$}
\label{subsection:reduction}

The four parameters $r_1, r_2, s_1, s_2$ in the RH problem \ref{rhp:modelM} for $M$ are
not independent. Indeed, we have the following dilation symmetry
\begin{multline} \label{eq:Msymmetry1}
     \begin{pmatrix} \gamma^{1/2} & 0 & 0 & 0 \\ 0 &  \gamma^{1/2} & 0 & 0 \\
     0 & 0 & \gamma^{-1/2} & 0 \\ 0 & 0 & 0 & \gamma^{-1/2} \end{pmatrix}
     M(\gamma^2 \zeta; r_1, r_2, s_1, s_2) \\
    =
    M(\zeta; \gamma^3 r_1, \gamma^3 r_2, \gamma s_1, \gamma s_2),
        \qquad \gamma > 0,
    \end{multline}
and the translational symmetry
\begin{multline} \label{eq:Msymmetry2}
    \begin{pmatrix} 1 & 0 & 0 & 0 \\ 0 & 1 & 0 & 0 \\
    i\delta(2s_1-r_1\delta) & 0 & 1 & 0 \\
    0 & i \delta(2s_2+r_2 \delta) & 0 & 1 \end{pmatrix}M(\zeta + 2\delta; r_1, r_2, s_1, s_2) \\
    \approx
    M(\zeta; r_1, r_2, s_1 - r_1 \delta, s_2 + r_2 \delta),
    \qquad \delta \in \mathbb R,
    \end{multline}
where the symbol $\approx$ in \eqref{eq:Msymmetry2} is used in
the sense of `equality up to contour deformation'. The equality
\eqref{eq:Msymmetry1} is immediate, due to the fact that both sides
satisfy the same RH problem. To show \eqref{eq:Msymmetry2}, we note that the
left-hand side of \eqref{eq:Msymmetry2} satisfies a RH problem with
the same jumps as $M$, but on a shifted contour. By an easy
transformation of the RH problem, we may shift the contour back to
the original contour. After this transformation, the equality
\eqref{eq:Msymmetry2} holds, since both sides of
\eqref{eq:Msymmetry2} also have the same asymptotic behavior as
$\zeta \to \infty$.

In view of \eqref{eq:Msymmetry1} and \eqref{eq:Msymmetry2}, we may
restrict ourselves to $r_1 = 1$ and $s_2 = 0$ in order to prove
the solvability statement in Proposition \ref{prop:large s1 solvability}.

For the residue matrix $M_1$ in \eqref{M:asymptotics}, we find from
\eqref{eq:Msymmetry1} that
\begin{multline} \label{eq:M1symmetry1}
    M_1(\gamma^3 r_1, \gamma^3 r_2, \gamma s_1, \gamma s_2) = \\
    \frac{1}{\gamma^2} \diag \begin{pmatrix} \gamma^{1/2}, \gamma^{1/2}, \gamma^{-1/2}, \gamma^{-1/2} \end{pmatrix}
    M_1( r_1, r_2, s_1, s_2) \diag \begin{pmatrix} \gamma^{-1/2}, \gamma^{-1/2}, \gamma^{1/2}, \gamma^{1/2}
    \end{pmatrix},
     \end{multline}
and from \eqref{eq:Msymmetry2}, after straightforward calculations,
\begin{multline} \label{eq:M1symmetry2}
    M_1(r_1, r_2, s_1 - r_1 \delta, s_2 + r_2 \delta)
    = M_{1,\delta}    \\
    +    \begin{pmatrix} 1 & 0 & 0 & 0 \\
        0 & 1 & 0 & 0 \\
        i\delta(2s_1-r_1\delta) & 0 & 1 & 0 \\
        0 & i \delta(2s_2+r_2 \delta) & 0 & 1 \end{pmatrix}
    M_1(r_1, r_2, s_1, s_2) \\
    \times
     \begin{pmatrix} 1 & 0 & 0 & 0 \\
        0 & 1 & 0 & 0 \\
        -i\delta(2s_1-r_1\delta) & 0 & 1 & 0 \\
        0 & -i \delta(2s_2+r_2 \delta) & 0 & 1 \end{pmatrix},
    \end{multline}
where
\begin{align*}
    (M_{1,\delta})_{1,1} & = - (M_{1,\delta})_{3,3} =
        -\frac{1}{2} \delta - 2 s_1^2 \delta^2 + 2 r_1 s_1 \delta^3 - \frac{1}{2} r_1^2 \delta^4, \\
    (M_{1,\delta})_{2,2} & = - (M_{1,\delta})_{4,4}  =
        -\frac{1}{2} \delta + 2 s_2^2 \delta^2 + 2 r_2 s_2 \delta^3 + \frac{1}{2}  r_2^2 \delta^4,  \\
    (M_{1,\delta})_{1,3} & = i (-2s_1 \delta +r_1 \delta^2), \\
    (M_{1,\delta})_{2,4} & = i (2 s_2 \delta + r_2 \delta^2), \\
    (M_{1,\delta})_{3,1} & = i \left(- s_1 \delta^2 + \frac{1}{3} (2 r_1 - 8s_1^3) \delta^3
    + 4 s_1^2 r_1 \delta^4 - 2 s_1 r_1^2 \delta^5 + \frac{1}{3} r_1^3 \delta^6\right), \\
    (M_{1,\delta})_{4,2} & = i\left(- s_2 \delta^2 +\frac{1}{3} (-2 r_2 + 8 s_2^3) \delta^3 +
    4 r_2 s_2^2 \delta^4  + 2 r_2^2 s_2 \delta^5 + \frac{1}{3}  r_2^3 \delta^6\right),
    \end{align*}
    and the other entries of $M_{1,\delta}$ are zero.

In view of the structure of $M_1$ in \eqref{eq:defM1}, a combination
of \eqref{eq:M1symmetry1} and \eqref{eq:M1symmetry2} yields
\begin{align*}
    d(\gamma^3 r_1, \gamma^3 r_2, \gamma(s_1 - r_1 \delta), \gamma(s_2 + r_2\delta)) &
        = \frac{1}{\gamma} d(r_1, r_2, s_1, s_2), \\
    c(\gamma^3 r_1, \gamma^3 r_2, \gamma(s_1 - r_1 \delta), \gamma(s_2 + r_2\delta)) &
        = \frac{1}{\gamma} \left( c(r_1, r_2, s_1, s_2)  - 2 s_1 \delta + r_1 \delta^2 \right), \\
    \wtil c(\gamma^3 r_1, \gamma^3 r_2, \gamma(s_1 - r_1 \delta), \gamma(s_2 + r_2\delta)) &
        = \frac{1}{\gamma} \left( \wtil c(r_1, r_2, s_1, s_2) + 2s_2 \delta + r_2 \delta^2
        \right),
    \end{align*}
for $\gamma > 0$ and $\delta \in \mathbb R$. These relations are
consistent with the formulas for $d, c$ and $\wtil c$ given in \eqref{eq:asy of d}--\eqref{eq:asy of til c}
as well as those in
\eqref{d:Painleve2}--\eqref{def:sigma}. Thus, we may indeed restrict to the
case $r_1=1$ and $s_2 = 0$ in the proof of Proposition \ref{prop:large s1 solvability}

\subsection{First transformation: $M \mapsto A$}

We consider the RH problem \ref{rhp:modelM} for $M$ with parameters $r_1 = 1$
and $s_2 = 0$.
The first transformation is a rescaling of the RH problem.
\begin{definition}
We define
\begin{equation}\label{M to A}
A(\zeta)=\diag(s_1^{1/4},s_1^{1/4},s_1^{-1/4},s_1^{-1/4})M(s_1\zeta),\qquad
\zeta\in\mathbb{C} \setminus \left( \bigcup_{k=0}^9 \Gamma_k \right).
\end{equation}
\end{definition}
Then $A$ satisfies the following RH problem similar to that for $M$.
\begin{rhp}
\textrm{ }
\begin{enumerate}[\rm (1)]
  \item $A(\zeta)$ is analytic for $\zeta\in\mathbb{C}\setminus
  \left( \bigcup_{k=0}^9 \Gamma_k \right)$, where $\Gamma_k$, $k=1,...,9$ is shown in Figure
  \ref{fig:modelRHP}.
  \item $A$ has the same jump matrix $J_k$ on $\Gamma_k$ as $M$.
  \item As $\zeta\to\infty$, we have
  \begin{align}\label{Aasy}
  A(\zeta)=&\left(I+\frac{A_1}{\zeta}+
  O\left(\frac{1}{\zeta^2}\right)\right)\diag((-\zeta)^{-1/4},\zeta^{-1/4},
  (-\zeta)^{1/4},\zeta^{1/4}) \nonumber \\
  &\times \frac{1}{\sqrt{2}}\begin{pmatrix}
  1 & 0 & -i & 0\\
  0 & 1 & 0 & i\\
  -i & 0 & 1 & 0\\
  0 & i & 0 & 1
  \end{pmatrix} \diag(e^{-\lambda \tilde \theta_1(\zeta)},e^{-\lambda \tilde\theta_2(\zeta)} ,e^{\lambda \tilde\theta_1(\zeta)},
  e^{\lambda \tilde\theta_2(\zeta)}),
  \end{align}
  with
  \begin{equation}\label{def:lambda:s}
    \lambda = s_1^{3/2},\end{equation}
  and
  \begin{align}\label{tilde theta}
  \tilde \theta_1(\zeta)=\frac{2}{3}(-\zeta)^{3/2}+2(-\zeta)^{1/2},\qquad
  \tilde \theta_2(\zeta)=\frac{2}{3} r_2 \zeta^{3/2}.
  \end{align}
\end{enumerate}
\end{rhp}

The residue matrix $A_1$ in \eqref{Aasy} is given by
\begin{equation}\label{A_1 reprentation}
A_1=
\frac{1}{s_1}\diag(s_1^{1/4},s_1^{1/4},s_1^{-1/4},s_1^{-1/4})M_1
\diag(s_1^{-1/4},s_1^{-1/4},s_1^{1/4},s_1^{1/4}).
\end{equation}
In view of \eqref{A_1 reprentation}, we find the following
expressions for the numbers $c,\wtil c$ and $d$ in \eqref{eq:defM1}:
\begin{equation}\label{d in A}
c =-i\sqrt{s_1}(A_1)_{1,3},\qquad \wtil c
=-i\sqrt{s_1}(A_1)_{2,4},\quad d =-i\sqrt{s_1}(A_1)_{1,4}.
\end{equation}

\subsection{Second transformation: $A \mapsto B$}

In the second transformation we apply contour deformations; see also
\cite{IKO}. The six rays $\Gamma_k$, $k=1,2,3,7,8,9$, emanating from
the origin are replaced by their parallel lines emanating from some
special points on the real line. More precisely, we replace
$\Gamma_1$ and $\Gamma_9$ by their parallel rays $\tilde\Gamma_1$
and $\tilde\Gamma_9$ emanating from the point~$2$, replace
$\Gamma_2$ and $\Gamma_8$ by their parallel rays $\til\Gamma_2$ and
$\til\Gamma_8$ emanating from the point $2-\epsilon$, and $\Gamma_3$
and $\Gamma_7$ by their parallel rays $\til\Gamma_3$ and
$\til\Gamma_7$ emanating from $\epsilon$, where $\epsilon$ is a
small positive number. See Figure~\ref{fig:ContourB}.

\begin{definition}
Denoting by $E_{i,j}$ the $4 \times 4$ elementary matrix with entry
$1$ at the $(i,j)$th position and all other entries equal to zero,
we then successively define
\begin{equation}\label{A to B1}
     B^{(1)}(\zeta)=\left\{
                \begin{array}{ll}
                  A(\zeta)(I+E_{3,1}), & \text{for $\zeta$ between $\Gamma_1$ and $\til\Gamma_1$}, \\
                  A(\zeta)(I-E_{3,1}), & \text{for $\zeta$  between $\Gamma_9$ and $\til\Gamma_9$}, \\
                  A(\zeta), & \text{elsewhere},
                \end{array}
              \right.
\end{equation}
\begin{equation}
    B^{(2)}(\zeta) = \left\{
                \begin{array}{ll}
                  B^{(1)}(\zeta)(I-E_{2,1} + E_{3,4}), &
                  \text{for $\zeta$ between $\Gamma_2$ and $\til \Gamma_2$} \\
        & \text{ and for $\zeta$ between $\Gamma_8$ and $\til \Gamma_8$}, \\
                  B^{(1)}(\zeta), & \text{elsewhere},
                \end{array}
              \right.
\end{equation}
and
\begin{equation}\label{B2 to B}
B(\zeta)= \left\{\begin{array}{ll} B^{(2)}(\zeta)
    \left( I + E_{1,2} - E_{4,3} \right), & \text{for $\zeta$ between $\Gamma_{3}$ and $\til \Gamma_3$}\\
    & \text{ and for $\zeta$ between $\Gamma_7$ and $\til \Gamma_7$}, \\
    B^{(2)}(\zeta), & \text{elsewhere}.
\end{array}\right.
\end{equation}
\end{definition}

\begin{figure}[t]
\begin{center}
   \setlength{\unitlength}{1truemm}
   \begin{picture}(100,70)(-5,2)
       \put(40,40){\line(1,0){50}}
       \put(40,40){\line(-1,0){50}}
       \put(60,40){\line(2,1){30}}
       \put(60,40){\line(2,-1){30}}
       \put(20,40){\line(-2,1){30}}
       \put(20,40){\line(-2,-1){30}}
       \put(55,40){\line(2,3){10}}
       \put(55,40){\line(2,-3){10}}
       \put(25,40){\line(-2,3){10}}
       \put(25,40){\line(-2,-3){10}}
       \put(20,40){\thicklines\circle*{1}}
       \put(60,40){\thicklines\circle*{1}}
       \put(25,40){\thicklines\circle*{1}}
       \put(55,40){\thicklines\circle*{1}}
       \put(19,36){$0$}
       \put(59.3,36){$2$}
       \put(24,36){$\epsilon$}
       \put(48,36){$2-\epsilon$}
       \put(75,40){\thicklines\vector(1,0){.0001}}
       \put(40,40){\thicklines\vector(1,0){.0001}}
       \put(5,40){\thicklines\vector(1,0){.0001}}
       \put(87,40.5){$\mathbb R$}
       \put(80,50){\thicklines\vector(2,1){.0001}}
       \put(80,30){\thicklines\vector(2,-1){.0001}}
       \put(0,50){\thicklines\vector(2,-1){.0001}}
       \put(0,30){\thicklines\vector(2,1){.0001}}
       \put(60.2,48){\thicklines\vector(2,3){.0001}}
       \put(60.3,32){\thicklines\vector(-2,3){.0001}}
       \put(19.8,48){\thicklines\vector(-2,3){.0001}}
       \put(19.7,32){\thicklines\vector(2,3){.0001}}

       \put(80,52.5){$\til\Gamma_1$}
       \put(65,52){$\til \Gamma_2$}
       \put(11,52){$\til \Gamma_3$}
       \put(0.5,50){$\Gamma_4$}
       \put(0,27){$\Gamma_6$}
       \put(11,25){$\til \Gamma_7$}
       \put(65,25){$\til \Gamma_8$}
       \put(80,25){$\til\Gamma_{9}$}
  \end{picture}
  \vspace{-20mm}
   \caption{Contour $\Sigma_B$ for the RH problem \ref{rhp:B} for $B$.
   Note the reversion of the orientation of some of the rays.}
   \label{fig:ContourB}
\end{center}
\end{figure}
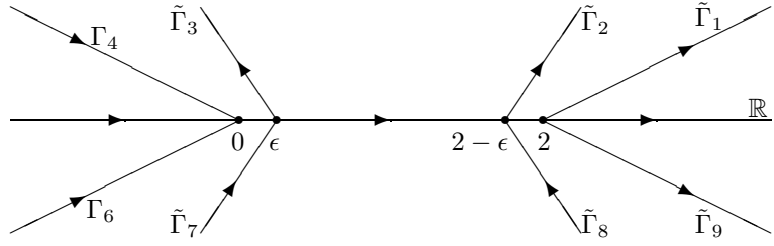

It is easily seen that $B$ is analytic in $\mathbb C \setminus
\Sigma_B$, where $\Sigma_B$ is the contour shown in
Figure~\ref{fig:ContourB}. Note that we reverse the orientation on
some of the rays, in particular the real line is oriented from left
to right; compare Figure~\ref{fig:ContourB} with
Figure~\ref{fig:modelRHP}. Moreover, $B$ satisfies the following RH
problem.
\begin{rhp}\label{rhp:B}
\textrm{}
\begin{enumerate}[\rm (1)]
  \item $B(\zeta)$ is analytic for $\zeta\in\mathbb{C}\setminus
  \Sigma_B$.
  \item $B$ has the following jumps on $\Sigma_B$:
  \begin{equation*}
  B_{+}(\zeta)=B_-(\zeta)J_B(\zeta),
  \end{equation*}
  where $J_B$ is defined by
  \begin{align*}
  J_B(\zeta)&=
  \begin{pmatrix}
  0 & 0 & 1 & 0\\
  0 & 1 & 0 & 0\\
  -1 & 0 & 0 & 0\\
  0 & 0 & 0 & 1
  \end{pmatrix}, \qquad \text{for $\zeta \in (2,+\infty)$},
  \\
  J_B(\zeta)&=(I+E_{3,1}),
\qquad \text{for $\zeta\in \tilde \Gamma_1\cup
  \tilde \Gamma_9$},
  \\
  J_B(\zeta)&=(I+E_{1,3}),
  \qquad \text{for $\zeta\in(2-\ep,2)$},\\
  J_B(\zeta)&= (I - E_{2,1} + E_{3,4}), \qquad \text{for $\zeta\in\tilde\Gamma_2
  \cup\tilde\Gamma_8$},
  \\
  J_B(\zeta)&=
  \begin{pmatrix}
  1 & 0 & 1 & 1\\
  0 & 1 & 1 & 1\\
  0 & 0 & 1 & 0\\
  0 & 0 & 0 & 1
  \end{pmatrix},
  \qquad \text{for $\zeta\in(\epsilon,2-\epsilon)$},
  \\
  J_B(\zeta)&= (I + E_{1,2} - E_{4,3}), \qquad \text{for $\zeta\in\tilde\Gamma_3
  \cup\tilde\Gamma_7$},
  \\
  J_B(\zeta)&= (I+E_{2,4}),
  \qquad \text{for $\zeta\in(0, \epsilon)$},
  \\
  J_B(\zeta)&=(I+E_{4,2}),
  \qquad \text{for $\zeta\in \Gamma_4
  \cup \Gamma_6$},
  \\
  J_B(\zeta)&=
  \begin{pmatrix}
  1 & 0 & 0 & 0\\
  0 & 0 & 0 & 1\\
  0 & 0 & 1 & 0\\
  0 & -1 & 0 & 0
  \end{pmatrix}, \qquad \text{for $\zeta \in (-\infty, 0)$}.
  \end{align*}

  \item As $\zeta\to\infty$, we have
  \begin{align}\label{asy of B}
  B(\zeta)=&\left(I+\frac{B_1}{\zeta}+
  O\left(\frac{1}{\zeta^2}\right)\right)\diag((-\zeta)^{-1/4},\zeta^{-1/4},
  (-\zeta)^{1/4},\zeta^{1/4}) \nonumber \\
  &\times \frac{1}{\sqrt{2}}\begin{pmatrix}
  1 & 0 & -i & 0\\
  0 & 1 & 0 & i\\
  -i & 0 & 1 & 0\\
  0 & i & 0 & 1
  \end{pmatrix} \diag(e^{-\lambda \tilde \theta_1(\zeta)},e^{-\lambda \tilde\theta_2(\zeta)} ,e^{\lambda \tilde\theta_1(\zeta)},
  e^{\lambda \tilde\theta_2(\zeta)}),
  \end{align}
  where $\lambda = s_1^{3/2}$ and $ \tilde \theta_{1}(\zeta)$, $\tilde \theta_2(\zeta)$ is
  given in \eqref{tilde theta}.
\end{enumerate}
\end{rhp}

By \eqref{d in A} and the transformation $A \mapsto B$ in \eqref{A
to B1}--\eqref{B2 to B}, it is readily seen that
\begin{equation}\label{d in B}
c =-i\sqrt{s_1}(B_1)_{1,3},\quad \wtil c =-i\sqrt{s_1}(B_1)_{2,4},\quad d
=-i\sqrt{s_1}(B_1)_{1,4}.
\end{equation}

\subsection{Third transformation: $B \mapsto C$}
In this transformation we partially normalize the RH problem
\ref{rhp:B} for $B$ at infinity. For this purpose, we introduce the
following \lq $g$-functions\rq:
\begin{equation}\label{g1}
g_1(\zeta)=\frac{2}{3}(2-\zeta)^{3/2},\qquad \zeta\in\mathbb{C}
\setminus [2,\infty),
\end{equation}
and
\begin{equation}\label{g2}
g_2(\zeta)=\tilde \theta_2(\zeta)=\frac{2}{3} r_2 \zeta^{3/2},
\qquad \zeta\in\mathbb{C} \setminus (-\infty, 0].
\end{equation}
Note that
\begin{equation}\label{asy of g1}
g_1(\zeta)=\tilde
\theta_1(\zeta)+(-\zeta)^{-1/2}-\frac{1}{3}(-\zeta)^{-3/2}+O(\zeta^{-5/2})
\end{equation}
as $\zeta \to \infty$, where $\tilde \theta_{1}$ is given in
\eqref{tilde theta}.

\begin{definition} We define
\begin{equation}\label{B to C}
C(\zeta)= (I+  i\lambda E_{3,1}) B(\zeta)\diag(e^{\lambda g_1(\zeta)},
e^{\lambda g_2(\zeta)},e^{-\lambda g_1(\zeta)},e^{-\lambda g_2(\zeta)}).
\end{equation}
\end{definition}
Then $C$ satisfies the following RH problem.

\begin{rhp}\label{rhpforC}
\textrm{}
\begin{enumerate}[\rm (1)]
  \item $C(\zeta)$ is analytic for $\zeta\in\mathbb{C}\setminus
  \Sigma_B$.
  \item $C$ has the following jumps on $\Sigma_B$:
  \[ C_+(\zeta) = C_-(\zeta) J_C(\zeta), \]
   where $J_C$ is given by
  \begin{align}
  J_C(\zeta)& =
  \begin{pmatrix}
  0 & 0 & 1 & 0\\
  0 & 1 & 0 & 0\\
  -1 & 0 & 0 & 0\\
  0 & 0 & 0 & 1
  \end{pmatrix}, \qquad \text{for $\zeta \in (2,\infty)$}, \nonumber
  \\ \nonumber
  J_C(\zeta) &=
  (I+e^{2\lambda g_1(\zeta)}E_{3,1}), \qquad \text{for $\zeta\in \tilde \Gamma_1\cup
  \tilde \Gamma_9$},
  \\   \nonumber
  J_C(\zeta) &=
  (I+e^{-2\lambda g_1(\zeta)}E_{1,3}),
  \qquad \text{for $\zeta\in(2-\ep,2)$},\\
  \nonumber
  J_C(\zeta) &= (I + e^{\lambda (g_1 - g_2)(\zeta)} (E_{3,4}-E_{2,1})),
  \qquad \text{for $\zeta\in\tilde\Gamma_2 \cup\tilde\Gamma_8$},
   \\ \nonumber
  J_C(\zeta) &=
  \begin{pmatrix}
  1 & 0 & e^{-2\lambda g_1(\zeta)} & e^{-\lambda (g_1+ g_2)(\zeta)}\\
  0 & 1 & e^{-\lambda (g_1 + g_2)(\zeta)} & e^{-2\lambda g_2(\zeta)}\\
  0 & 0 & 1 & 0\\
  0 & 0 & 0 & 1
  \end{pmatrix},
  \qquad \text{for $\zeta\in(\ep,2-\ep)$},\\
  \nonumber
  J_C(\zeta) &= (I + e^{-\lambda (g_1 - g_2)(\zeta)}(E_{1,2} - E_{4,3})),
  \qquad \text{for $\zeta\in\tilde\Gamma_3
  \cup\tilde\Gamma_7$},
  \nonumber
  \\
  \nonumber
  J_C(\zeta) &= (I+e^{-2\lambda g_2(\zeta)}E_{2,4}),
  \qquad \text{for $\zeta\in(0, \ep )$},
  \\
  \nonumber
  J_C(\zeta) &= (I+e^{2\lambda g_2(\zeta)}E_{4,2}), \qquad \text{for $\zeta\in \Gamma_4 \cup \Gamma_6$},
  \\
  \nonumber 
  J_C(\zeta) &=
  \begin{pmatrix}
  1 & 0 & 0 & 0\\
  0 & 0 & 0 & 1\\
  0 & 0 & 1 & 0\\
  0 & -1 & 0 & 0
  \end{pmatrix}, \qquad \text{for $\zeta \in (-\infty, 0)$}.
  \end{align}

  \item As $\zeta\to\infty$, we have
  \begin{align}\label{asy of C}
  C(\zeta)=\left(I+\frac{C_1}{\zeta}+
  O\left(\frac{1}{\zeta^2}\right)\right)\diag((-\zeta)^{-1/4},\zeta^{-1/4},
  (-\zeta)^{1/4},\zeta^{1/4})
  \frac{1}{\sqrt{2}}\begin{pmatrix}
  1 & 0 & -i & 0\\
  0 & 1 & 0 & i\\
  -i & 0 & 1 & 0\\
  0 & i & 0 & 1
  \end{pmatrix},
  \end{align}
  where
 \begin{equation}\label{C1}
 C_1 = (I + i \lambda E_{3,1})
  \left[B_1 (I - i \lambda E_{3,1})
  + \begin{pmatrix}
  - \frac{\lambda^2}{2} & 0 & -i\lambda  & 0\\
  0 & 0 & 0 & 0 \\
  -\frac{i\lam}{3} & 0 & -\frac{\lambda^2}{2} & 0\\
  0 & 0 & 0 & 0
  \end{pmatrix}\right].
 \end{equation}
\end{enumerate}
\end{rhp}

\begin{proof}
The jump condition for $C$ in item (2) follows from straightforward
calculations, where we have made use of the facts that
$(g_{1,+}+g_{1,-})(\zeta)=0$ for $\zeta \in [2,\infty)$ and
$(g_{2,+}+g_{2,-})(\zeta)=0$ for $\zeta \in (-\infty, 0]$.

To establish the large $\zeta$ behavior of $C$ shown in item (3), we
first observe from \eqref{asy of g1} and \eqref{g2} that
\begin{align*}
e^{\pm\lambda (g_1-\tilde \theta_1)(\zeta)} & =1\pm\lambda
(-\zeta)^{-1/2}+\frac{\lambda^2}{2} (-\zeta)^{-1}\mp
\frac{\lam}{3}(-\zeta)^{-3/2}+O(\zeta^{-2}),
\end{align*}
as $\zeta \to \infty$, and
\begin{align*}
e^{ \pm \lambda (g_2-\tilde \theta_2)(\zeta)} & \equiv 1.
\end{align*}
Hence, by \eqref{asy of B}, it is readily seen that
\begin{align} \label{B diag g }
&B(\zeta)\diag(e^{\lambda g_1(\zeta)}, e^{\lambda
g_2(\zeta)},e^{-\lambda g_1(\zeta)},e^{-\lambda g_2(\zeta)})
  \nonumber \\
&=\left(I+\frac{B_1}{\zeta}+
  O\left(\frac{1}{\zeta^2}\right)\right)\diag((-\zeta)^{-1/4},\zeta^{-1/4},
  (-\zeta)^{1/4},\zeta^{1/4})\frac{1}{\sqrt{2}}\begin{pmatrix}
  1 & 0 & -i & 0\\
  0 & 1 & 0 & i\\
  -i & 0 & 1 & 0\\
  0 & i & 0 & 1
  \end{pmatrix} \nonumber \\
  &~~~~\times  \left( I + \left(\frac{\lambda}{\sqrt{-\zeta}}-
  \frac{\lambda^2}{2\zeta}-\frac{\lam}{3(-\zeta)^{3/2}}\right)E_{1,1}
  -\left(\frac{\lambda}{\sqrt{-\zeta}}+\frac{\lambda^2}{2\zeta}-\frac{\lam}{3(-\zeta)^{3/2}}\right) E_{3,3}+O(\zeta^{-3/2}) \right),
\end{align}
as $\zeta\to\infty$. Note that
\begin{align*}
&\diag((-\zeta)^{-1/4},\zeta^{-1/4},
  (-\zeta)^{1/4},\zeta^{1/4})
 \frac{1}{\sqrt{2}}\begin{pmatrix}
  1 & 0 & -i & 0\\
  0 & 1 & 0 & i\\
  -i & 0 & 1 & 0\\
  0 & i & 0 & 1
  \end{pmatrix}
  \\
 &~~~~\times \left( I + \left(\frac{\lambda}{\sqrt{-\zeta}}
 -\frac{\lambda^2}{2\zeta}-\frac{\lam}{3(-\zeta)^{3/2}}\right)E_{1,1}
 -\left(\frac{\lambda}{\sqrt{-\zeta}}+\frac{\lambda^2}{2\zeta}-\frac{\lam}{3(-\zeta)^{3/2}}\right) E_{3,3}\right)
 \\
&=\left[\begin{pmatrix}
  1 & 0 & 0 & 0\\
  0 & 1 & 0 & 0\\
  -i\lambda & 0 & 1 & 0\\
  0 & 0 & 0 & 1
  \end{pmatrix}
  + \begin{pmatrix}
  -\frac{\lambda^2}{2} & 0 & -i\lambda  & 0\\
  0 & 0 & 0 & 0 \\
   -\frac{i\lam}{3} & 0&-\frac{\lambda^2}{2} & 0\\
  0 & 0 & 0 & 0
  \end{pmatrix}\frac{1}{\zeta}-\frac{i\lam}{3\zeta^2}E_{1,3}\right]
 \\
&~~~~\times \diag((-\zeta)^{-1/4},\zeta^{-1/4},
  (-\zeta)^{1/4},\zeta^{1/4})
 \frac{1}{\sqrt{2}}\begin{pmatrix}
  1 & 0 & -i & 0\\
  0 & 1 & 0 & i\\
  -i & 0 & 1 & 0\\
  0 & i & 0 & 1
  \end{pmatrix}.
\end{align*}
This, together with \eqref{B to C} and \eqref{B diag g }, implies
\eqref{asy of C} and \eqref{C1}.
\end{proof}

It follows from \eqref{d in B}, \eqref{C1} and $\lambda=s_1^{3/2}$
that
\begin{equation}\label{cctild:D}
c = -i\sqrt{s_1}(C_1)_{1,3}+s_1^2, \quad \wtil c =
-i\sqrt{s_1}(C_1)_{2,4},\quad d = -i\sqrt{s_1}(C_1)_{1,4}.
\end{equation}

\subsection{Asymptotic behavior of jump matrices for $C$}

The jump matrices of the RH problem \ref{rhpforC} for $C$ all tend
to the identity matrix exponentially fast as $\lambda \to +\infty$,
except for the jumps on $(-\infty,0)$ and $(2,\infty)$. This is
easily seen for the jumps on $\tilde \Gamma_1 \cup \Gamma_4
\cup(0,2)\cup \Gamma_6 \cup \tilde \Gamma_9 $, due to the facts that
\begin{equation}
\Re~g_1(\zeta)\left\{
                \begin{array}{ll}
                  <0, & \quad\text{if $\frac{2\pi}{3}<\arg (\zeta-2)<\frac{4 \pi}{3}$},
                  \\
                  >0, & \quad\text{if $0<\arg (\zeta-2)<\frac{2 \pi}{3}$ or
                              $\frac{2\pi}{3}<\arg (\zeta-2)<2\pi$},
                \end{array}
              \right.
\end{equation}
and
\begin{equation}
\Re~g_2(\zeta)\left\{
\begin{array}{ll}
<0, & \quad\text{if $\frac{\pi}{3}<\arg \zeta <\pi$ or $-\pi< \arg
\zeta < -\frac{\pi}{3}$},
\\
>0, & \quad\text{if $-\frac{\pi}{3}<\arg \zeta <\frac{\pi}{3}$}.
\end{array}
              \right.
\end{equation}

For the jump matrices on $\tilde \Gamma_2 \cup \tilde \Gamma_8$ and
$\tilde \Gamma_3 \cup \tilde \Gamma_7$, it is necessary to analyze
the function $(g_1-g_2)(\zeta)$. Recall that $\til\Gamma_2$,
$\til\Gamma_8$ ($\til\Gamma_3$, $\til\Gamma_7$) intersect the real
line at the point $2-\ep$ ($\ep$, respectively) with $\ep>0$ a small
number. It is easily seen that for $\ep\to 0$,
\begin{align*}
(g_1 - g_2)(2-\ep) &= -\frac{4\sqrt{2}}{3} r_2 +O(\ep),\\
(g_1 - g_2)\left(\ep\right) &= \frac{4\sqrt{2}}{3} +O(\ep).
\end{align*}
By choosing $\ep$ sufficiently small, we can then guarantee that
$(g_1-g_2)(2-\ep) < 0$ and $(g_1 - g_2)(\ep)>0$. Hence by deforming the
contours if necessary, we may assume that $\Re(g_1-g_2)(\zeta)<0$ on $\tilde
\Gamma_2 \cup \tilde \Gamma_8$, while $\Re(g_1-g_2)(\zeta)>0$ on $\tilde
\Gamma_3 \cup \tilde \Gamma_7$, which ensures that the jump matrices on these
contours uniformly tend to the identity matrix, exponentially fast as $\lambda
\to +\infty$.

\subsection{Construction of global parametrix}

Away from the  points $2$ and $0$, we expect that $C$ should be well
approximated by the solution $C^{(\infty)}$ of the following RH
problem, which is obtained from the RH problem \ref{rhpforC} for $C$
by removing all exponentially decaying entries in the jump matrices:
\begin{rhp}
\textrm{}
\begin{enumerate}[\rm (1)]
  \item $C^{(\infty)}(\zeta)$ is analytic for $\zeta\in\mathbb{C}\setminus
  \left( (-\infty,0]\cup[2,\infty) \right)$.

  \item $C^{(\infty)}(\zeta)$ has the jumps
  \begin{align*}
  C^{(\infty)}_{+}(\zeta)&=C^{(\infty)}_-(\zeta)
  \begin{pmatrix}
  0 & 0 & 1 & 0\\
  0 & 1 & 0 & 0\\
  -1 & 0 & 0 & 0\\
  0 & 0 & 0 & 1
  \end{pmatrix}, \qquad \text{for $\zeta>2$},
  \\
  C^{(\infty)}_{+}(\zeta)&=C^{(\infty)}_-(\zeta)
  \begin{pmatrix}
  1 & 0 & 0 & 0\\
  0 & 0 & 0 & 1\\
  0 & 0 & 1 & 0\\
  0 & -1 & 0 & 0
  \end{pmatrix}, \qquad \text{for $\zeta<0$}.
  \end{align*}

  \item As $\zeta\to\infty$, we have
  \begin{align}\label{Pinfty:asy}
  C^{(\infty)}(\zeta)=&\left(I+\frac{C^{(\infty)}_1}{\zeta}+O\left(\frac{1}{\zeta^2}\right)
  \right)\diag((-\zeta)^{-1/4},\zeta^{-1/4},
  (-\zeta)^{1/4},\zeta^{1/4}) \nonumber \\
   & \times
  \frac{1}{\sqrt{2}}
  \begin{pmatrix}
  1 & 0 & -i & 0\\
  0 & 1 & 0 & i\\
  -i & 0 & 1 & 0\\
  0 & i & 0 & 1
  \end{pmatrix}.
  \end{align}

\end{enumerate}
\end{rhp}

It is worthwhile to point out that $C^{(\infty)}$ is independent of $\lambda$.
The RH problem for $C^{(\infty)}$ can be solved explicitly, and its solution is
given by
\begin{equation}\label{C infty}
  C^{(\infty)}(\zeta)
  =\diag\left((2-\zeta)^{-1/4}, \zeta^{-1/4},
  (2-\zeta)^{1/4}, \zeta^{1/4}\right)
   \frac{1}{\sqrt{2}}
  \begin{pmatrix}
  1 & 0 & -i & 0\\
  0 & 1 & 0 & i\\
  -i & 0 & 1 & 0\\
  0 & i & 0 & 1
  \end{pmatrix},
  \end{equation}
where we take the branch cuts of $\zeta^{1/4}$ and $(2-\zeta)^{1/4}$
along $(-\infty,0]$ and $[2,\infty)$, respectively.

The residue matrix $C^{(\infty)}_1$ in \eqref{Pinfty:asy} is
diagonal:
\begin{equation}\label{Pinfty1}
C^{(\infty)}_1 = \diag\left(\frac 12, 0,-\frac 12,0 \right).
\end{equation}

\subsection{Construction of local parametrices}
\label{subsec:localparametrices}
 The global parametrix
$C^{(\infty)}$ is a good approximation for $C$ only for $\zeta$ bounded away
from the points $0$ and $2$. Here we will construct the parametrices $C^{(0)}$
and $C^{(2)}$ near these points. Since the local parametrices around $0$ and
$2$ can be built in a similar manner, we will only consider the local
parametrix $C^{(0)}$ near $0$. Let $D(0,\delta)$ be a fixed disk centered at
$0$ with radius $\delta<\ep$, and let $\partial D(0,\delta)$ denote its
boundary. We look for a $4 \times 4$ matrix-valued function $C^{(0)}$ defined
in $D(0,\delta)$ which satisfies the following.

\begin{rhp}
\textrm{}
\begin{enumerate}[\rm (1)]
  \item $C^{(0)}(\zeta)$ is analytic for $\zeta\in D(0,\delta)\setminus
  (\mathbb{R}\cup  \Gamma_4 \cup  \Gamma_6)$.

  \item For $\zeta\in D(0,\delta) \cap
  (\mathbb{R}\cup \tilde \Gamma_4 \cup \tilde \Gamma_6)$, $C^{(0)}$ has the jumps
  \begin{align}
   C^{(0)}_{+}(\zeta)&=  C^{(0)}_-(\zeta) J_C(\zeta),
  \end{align}
  where $J_C$ is the jump matrix in the RH problem for $C$.
  \item As $\lambda \to +\infty $, we have
  \begin{align}
  C^{(0)} (\zeta)= C^{(\infty)}\left(I+
  O\left(\frac{1}{\lambda}\right)\right),
  \end{align}
  uniformly for $\zeta \in \partial D(0,\delta) \setminus
  (\mathbb{R}\cup  \Gamma_4 \cup \Gamma_6)$.

\end{enumerate}
\end{rhp}

The solution of the above RH problem for $C^{(0)}$ can be built via
Airy functions and their derivatives in a standard way, we follow
the theme in \cite{Dei,DKMVZ1}.

Let $y_0$, $y_1$ and $y_2$ be the functions defined by
\begin{equation}
y_0(\zeta)=\Ai(\zeta), \qquad y_1(\zeta)=\omega\Ai(\omega \zeta),
\qquad y_2(\zeta)=\omega^2\Ai(\omega^2 \zeta),
\end{equation}
where $\Ai$ is the usual Airy function and $\omega=e^{2\pi i/3}$.
Consider the following $4\times4$ matrix-valued function $\Psi$:
\begin{align}
\Psi(\zeta)&=\begin{pmatrix} 1 & 0 & 0 & 0 \\
0 & y_0(\zeta) & 0 & -y_2(\zeta)\\
0 & 0 & 1 & 0 \\
0 & y_0'(\zeta) & 0 & -y_2'(\zeta)
\end{pmatrix}
,\quad \arg \zeta \in (0,2\pi/3), \nonumber
\\
\Psi(\zeta)&=\begin{pmatrix} 1 & 0 & 0 & 0 \\
0 & -y_1(\zeta) & 0 & -y_2(\zeta) \\
0 & 0 & 1 & 0 \\
0 & -y_1'(\zeta) & 0 & -y_2'(\zeta)
\end{pmatrix}
,\quad \arg \zeta \in (2\pi/3,\pi), \nonumber
\\
\Psi(\zeta)&=\begin{pmatrix}
1 & 0 & 0 & 0 \\
0 & -y_2(\zeta) & 0 & y_1(\zeta) \\
0 & 0 & 1 & 0 \\
0 & -y_2'(\zeta) & 0 & y_1'(\zeta)
\end{pmatrix}
,\quad \arg \zeta \in (-\pi,-2\pi/3), \nonumber
\\
 \Psi(\zeta)&=\begin{pmatrix} 1 & 0 & 0 & 0 \\
0 & y_0(\zeta) & 0 & y_1(\zeta) \\
0 & 0 & 1 & 0 \\
0 & y_0'(\zeta) & 0 & y_1'(\zeta)
\end{pmatrix}
,\quad \arg \zeta \in (-2\pi/3,0). \nonumber
\end{align}
The parametrix $C^{(0)}$ is then built in the following form
\begin{equation}\label{para near left point}
C^{(0)} (\zeta)=E(\zeta)\Psi\left(\lambda \zeta
\right)\diag(1,e^{\lambda g_2(\zeta)}, 1,e^{-\lambda g_2(\zeta)}),
\end{equation}
where $E(\zeta)$ is analytic in $D(0,\delta)$. It is straightforward
(cf. \cite[Sec.~7.6]{Dei}) to verify that $C^{(0)}$ given in
\eqref{para near left point} satisfies items (1) and (2) of the RH
problem for $C^{(0)}$. To determine the analytic prefactor $E$, one
needs to use the matching condition (3) and the asymptotics of Airy
function $\Ai(\zeta)$ as $\zeta\to\infty$. A direct calculation
gives
\begin{align*}
E(\zeta)&=C^{(\infty)}(\zeta)
\begin{pmatrix} 1 & 0 & 0 & 0 \\
0 & \sqrt{\pi} & 0 & -\sqrt{\pi} \\
0 & 0 & 1 & 0 \\
0 & -i\sqrt{\pi} & 0 & -i\sqrt{\pi}
\end{pmatrix}
\begin{pmatrix} 1 & 0 & 0 & 0 \\
0 & \left(\lambda \zeta\right)^{1/4} & 0 & 0 \\
0 & 0 & 1 & 0 \\
0 & 0 & 0 & \left(\lambda \zeta \right)^{-1/4}
\end{pmatrix}.
\end{align*}
This completes the construction of the local parametrix $C^{(0)}$.

\subsection{Final transformation: $C \mapsto D$}
\begin{definition}
We define the final transformation
\begin{equation}\label{def:R:Lun}
D(\zeta)=\left\{
       \begin{array}{ll}
         C(\zeta)(C^{(0)})^{-1}(\zeta), & \text{in a  $\delta$-neighborhood of $0$,} \\
         C(\zeta)(C^{(2)})^{-1}(\zeta), & \text{in a  $\delta$-neighborhood of $2$,} \\
         C(\zeta)(C^{(\infty)})^{-1}(\zeta), & \text{elsewhere.}
       \end{array}
     \right.
\end{equation}
\end{definition}

From the construction of the parametrices, it follows that $D$
satisfies the following RH problem.
\begin{rhp}\label{rhp:D}
\textrm{}
\begin{enumerate}[\rm (1)]
\item $D$ is analytic in $\mathbb{C} \setminus \Sigma_D$, where $\Sigma_D$ is
shown in Figure \ref{fig:ContourD}.

\item $D$ has jumps $D_+(\zeta)=D_-(\zeta)J_D(\zeta)$ for
 $\zeta\in \Sigma_D$, where
\begin{equation}\label{jump for D}
J_D(\zeta)=\left\{
      \begin{array}{ll}
        C^{(\infty)}(\zeta) (C^{(0)} )^{-1}(\zeta), & \text{on $|\zeta| = \delta$,} \\
        C^{(\infty)}(\zeta) (C^{(2)} )^{-1}(\zeta), & \text{on $|\zeta - 2| = \delta$,} \\
        C^{(\infty)}(\zeta) J_C(\zeta) (C^{(\infty)})^{-1}(\zeta), & \text{on the rest of $\Sigma_D$}.
      \end{array}
    \right.
\end{equation}
\item As $\zeta \to
\infty$, we have
\begin{equation}\label{eq: D asy}
D(\zeta)=I+\frac{D_1}{\zeta}+O\left(\frac{1}{\zeta^2}\right).
\end{equation}

\end{enumerate}
\end{rhp}

\begin{figure}[t]
\begin{center}
   \setlength{\unitlength}{1truemm}
   \begin{picture}(100,70)(-5,2)
       \put(40,40){\line(1,0){17.5}}
       \put(40,40){\line(-1,0){17.5}}
       \put(62.5,41){\line(2,1){28}}
       \put(62.5,39){\line(2,-1){28}}
       \put(17.5,41){\line(-2,1){28}}
       \put(17.5,39){\line(-2,-1){28}}
       \put(55,40){\line(1,2){10}}
       \put(55,40){\line(1,-2){10}}
       \put(25,40){\line(-1,2){10}}
       \put(25,40){\line(-1,-2){10}}
       \put(20,40){\thicklines\circle*{1}}
       \put(20,40){\circle{5}}
       \put(60,40){\thicklines\circle*{1}}
       \put(60,40){\circle{5}}
       \put(25,40){\thicklines\circle*{1}}
       \put(55,40){\thicklines\circle*{1}}
       \put(19.1,42.3){\thicklines\vector(-1,0){.0001}}
       \put(59.1,42.3){\thicklines\vector(-1,0){.0001}}
       \put(19,34){$0$}
       \put(59.3,34){$2$}
       \put(40,40){\thicklines\vector(1,0){.0001}}
       \put(80,50){\thicklines\vector(2,1){.0001}}
       \put(80,30){\thicklines\vector(2,-1){.0001}}
       \put(0,50){\thicklines\vector(2,-1){.0001}}
       \put(0,30){\thicklines\vector(2,1){.0001}}
       \put(61,52){\thicklines\vector(1,2){.0001}}
       \put(61,28){\thicklines\vector(-1,2){.0001}}
       \put(19,52){\thicklines\vector(-1,2){.0001}}
       \put(19,28){\thicklines\vector(1,2){.0001}}

       \put(80,52.5){$\til\Gamma_1$}
       \put(66.5,60){$\til\Gamma_2$}
       \put(15,60){$\til\Gamma_3$}
       \put(0.5,50){$\Gamma_4$}
       \put(0,32){$\Gamma_6$}
       \put(16.5,20){$\til\Gamma_7$}
       \put(65,20){$\til\Gamma_8$}
       \put(80,30){$\til\Gamma_{9}$}

  \end{picture}
   \vspace{-16mm}
   \caption{Contour  $\Sigma_D$ for the RH problem \ref{rhp:D} for $D$.}
   \label{fig:ContourD}
\end{center}
\end{figure}
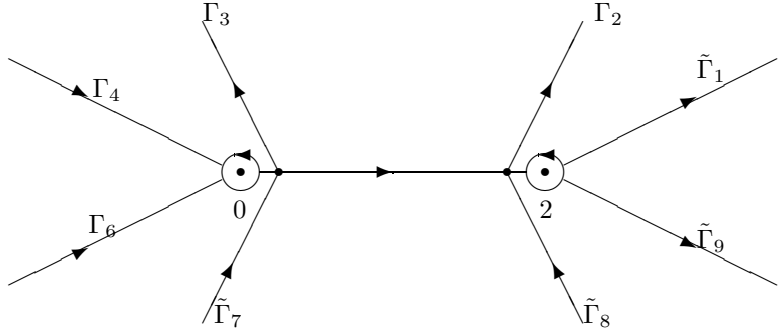

The jump matrix for $D$ satisfies
\begin{equation*}
J_D(\zeta)=I+O(1/\lambda), \qquad \text{as $\lambda \to +\infty$},
\end{equation*}
uniformly for $z$ on the circles $|z| = \delta$ and $|z-2| =
\delta$, and the jumps on the remaining contours of $\Sigma_D$ are
exponentially close to the identity matrix. In particular, we note
that
\begin{align}\label{R:expsmall1}
    J_D(\zeta) &= I+ e^{\lambda(g_1-g_2)(\zeta)} C^{(\infty)}(\zeta)
        \left(E_{3,4}-E_{2,1} \right) (C^{(\infty)})^{-1}(\zeta) ,
    ~~~ \text{for } \zeta \in \tilde{\Gamma}_2 \cup \tilde{\Gamma}_8, \\
    J_D(\zeta) &= I+ e^{-\lambda(g_1-g_2)(\zeta)} C^{(\infty)}(\zeta)
        \left(E_{1,2} -E_{4,3} \right) (C^{(\infty)})^{-1}(\zeta) ,
       ~~~ \text{for } \zeta \in \tilde{\Gamma}_3 \cup \tilde{\Gamma}_7,
    \label{R:expsmall2}
\end{align}
and
\begin{align}\label{jumpR:explicit}
J_D(\zeta) =&~I + \frac{e^{-2 \lambda g_1(\zeta)}}{2}
    \begin{pmatrix} i & 0 & (2-\zeta)^{-1/2} & 0 \\
    0 & 0 & 0 & 0 \\
    (2-\zeta)^{1/2} & 0 & -i & 0 \\
    0 & 0 & 0 & 0 \end{pmatrix}
    + \frac{e^{-2 \lambda g_2(\zeta)}}{2}
    \begin{pmatrix}  0 & 0 & 0 & 0 \\
    0 & -i & 0 & \zeta^{-1/2}  \\
    0 & 0 & 0 & 0 \\
    0 & \zeta^{1/2} & 0 & i
    \end{pmatrix}
    \nonumber \\
     &+ \frac{e^{-\lambda(g_1 + g_2)(\zeta)}}{2}
    \begin{pmatrix}
    0 & -i \left( \frac{\zeta}{2-\zeta} \right)^{1/4} & 0 & \left(\zeta (2-\zeta) \right)^{-1/4} \\
    i \left( \frac{2-\zeta}{\zeta} \right)^{1/4} & 0 & \left(\zeta (2-\zeta)\right)^{-1/4} & 0 \\
    0 & -\left(\zeta (2-\zeta) \right)^{1/4} & 0 & - i \left( \frac{2-\zeta}{\zeta} \right)^{1/4} \\
    -\left(\zeta (2-\zeta) \right)^{1/4} & 0 &   i \left( \frac{\zeta}{2-\zeta} \right)^{1/4} & 0
    \end{pmatrix},
\end{align}
for $\zeta\in(\ep,2-\ep)$, with the aid of \eqref{jump for D} and
\eqref{C infty}.

By standard arguments in \cite{Dei,DKMVZ1}, we then conclude that
\begin{equation}\label{RLun:small}
D(\zeta)=I+O\left(\frac{1}{\lambda(|\zeta|+1)}\right),
\end{equation}
as $\lambda\to +\infty$, uniformly for $\zeta\in\mathbb{C} \setminus
\Sigma_D$.

It follows from \eqref{cctild:D} and the above constructions that
\begin{equation}\label{cctild:R} c =
-i\sqrt{s_1}(D_1)_{1,3}+s_1^2,\quad \wtil c =
-i\sqrt{s_1}(D_1)_{2,4}, \quad d =-i\sqrt{s_1}(D_1)_{1,4},
\end{equation}
where we used \eqref{def:R:Lun} for large $\zeta$ and the fact that
$C^{(\infty)}_1$ in \eqref{Pinfty1} is diagonal.

\subsection{Proof of Proposition \ref{prop:large s1 solvability}}

\subsubsection*{Solvability of the RH problem for $s_1$ large}

In the above, we applied a series of invertible transformations $M \mapsto A
\mapsto B \mapsto C \mapsto D$, so that the matrix-valued function $D$ exists
and uniformly tends to the identity matrix as $s_1 \to +\infty$. This
immediately implies the solvability of the RH problem for $M$ for $s_1$
sufficiently large.

\subsection*{Asymptotics of $c$, $\wtil c$ and $d$ as $s_1\to+\infty$}
Recalling $\lambda = s_1^{3/2}$, it follows from \eqref{RLun:small} that $D_1 =
O(\lambda^{-1})
    = O(s_1^{-3/2})$ as $s_1 \to +\infty$.
Then by \eqref{cctild:R}, we have
\[ c = s_1^2 + O(s_1^{-1}), \quad \wtil c = O(s_1^{-1}), \quad d = O(s_1^{-1}), \]
as $s_1 \to +\infty$. This, together with our assumption
\eqref{eq:assumption}, implies \eqref{eq:asy of c} and \eqref{eq:asy
of til c}.

It remains to establish \eqref{eq:asy of d}. This requires more
effort. First we make some observations on the jump matrix $J_D$ in
\eqref{jump for D}. The matrix $J_D$ has the \lq checkerboard\rq\
sparsity pattern
\begin{equation}\label{chessboard} \begin{pmatrix}*&0&*&0\\
0&*&0&*\\ *&0&*&0\\ 0&*&0&*\end{pmatrix},
\end{equation}
on the circles $|z| = \delta$ and $|z-2| = \delta$, on $(-\infty,
\epsilon) \cup (2-\epsilon, 2)$ and on the contours $\tilde \Gamma_1
\cup \tilde \Gamma_9 \cup \Gamma_4 \cup \Gamma_6$. Here, $*$ denotes
an unspecified matrix entry.

Let $D^{(1)}$ be the solution to the following RH problem.

\begin{rhp}
\textrm{}
\begin{enumerate}[\rm(1)]
\item $D^{(1)}$ is analytic in $\mathbb{C} \setminus \Sigma_D$, where $\Sigma_D$ is
shown in Figure \ref{fig:ContourD}.

\item $D^{(1)}$ has jumps $D^{(1)}_+(\zeta)=D^{(1)}_-(\zeta)J_{D^{(1)}}(\zeta)$ for
 $\zeta\in \Sigma_D$, where
\begin{equation*}\label{jump for D1}
    J_{D^{(1)}}(\zeta) = \left\{
    \begin{array}{ll}
    I + \frac{e^{-2 \lambda g_1(\zeta)}}{2}
    \begin{pmatrix} i & 0 & (2-\zeta)^{-1/2} & 0 \\
    0 & 0 & 0 & 0 \\
    (2-\zeta)^{1/2} & 0 & -i & 0 \\
    0 & 0 & 0 & 0 \end{pmatrix} \\
    ~~+ \frac{e^{-2 \lambda g_2(\zeta)}}{2}
    \begin{pmatrix}
    0 & 0 & 0 & 0 \\
    0 & -i & 0 & \zeta^{-1/2}  \\
    0 & 0 & 0 & 0 \\
    0 & \zeta^{1/2} & 0 & i
    \end{pmatrix}, & \text{for } \zeta \in (\ep, 2-\ep), \\
    I, &  \text{for } \zeta \in \til \Gamma_2 \cup \til \Gamma_3 \cup \til \Gamma_7 \cup \til \Gamma_8, \\
    J_D(\zeta), & \text{elsewhere}.
      \end{array}
    \right.
\end{equation*}

\item As $\zeta \to
\infty$, we have
\begin{equation}\label{D1:asy}
D^{(1)}(\zeta)=I+\frac{D_1^{(1)}}{\zeta}+O\left(\frac{1}{\zeta^2}\right).
\end{equation}
\end{enumerate}
\end{rhp}

It is easily seen that $J_{D^{(1)}}$ still tends to the identity
matrix as $\lambda \to +\infty$, hence, there is a unique solution
$D^{(1)}$ to this RH problem if $\lambda$ is large enough. Moreover,
we have
\begin{equation}\label{eq:D1 est}
D^{(1)}(\zeta)=I+O\left(\frac{1}{\lambda(|\zeta|+1)}\right),
\end{equation}
as $\lambda\to +\infty$, uniformly for $\zeta\in\mathbb{C} \setminus
\Sigma_D$.

Since all the jump matrices for $D^{(1)}$ have the checkerboard
pattern \eqref{chessboard}, the solution $D^{(1)}$ has the same
pattern. Obviously, the residue matrix $D_1^{(1)}$ in \eqref{D1:asy}
also has the pattern \eqref{chessboard}.

\begin{definition}
We define
\begin{equation}\label{def:R2}
R (\zeta) =D(\zeta)(D^{(1)})^{-1}(\zeta).
\end{equation}
\end{definition}
The matrix-valued function $R(\zeta)$ satisfies the following RH
problem.

\begin{rhp}\label{rhp:R}
\textrm{}
\begin{enumerate}[\rm (1)]
\item $R$ is analytic in $\mathbb C \setminus \Sigma_R$,
where $\Sigma_R$ is as in Figure \ref{fig:ContourR}.

\item $R$ has the jumps
$R_+(\zeta)=R_-(\zeta) J_{R}(\zeta)$ for $\zeta\in\Sigma_R$, where
\begin{equation}\label{eq:JR}
J_R(\zeta) = D^{(1)}_-(\zeta) J_D(\zeta) ( D^{(1)}_+)^{-1}(\zeta).
\end{equation}

\item As $\zeta\to\infty$, we have
\begin{equation}\label{R:asy}
R(\zeta) = I+\frac{R_1}{\zeta}+O\left(\frac{1}{\zeta^2}\right).
\end{equation}
\end{enumerate}
\end{rhp}

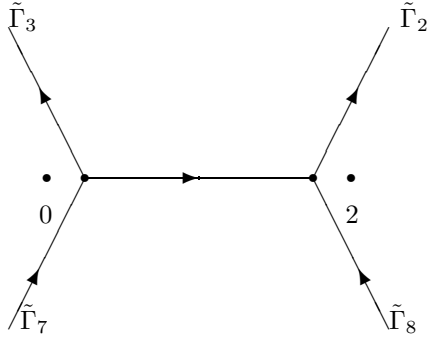
\begin{figure}[t]
\vspace{14mm}
\begin{center}
   \setlength{\unitlength}{1truemm}
   \begin{picture}(100,70)(-5,2)
       \put(40,40){\line(1,0){15}}
       \put(40,40){\line(-1,0){15}}
       \put(55,40){\line(1,2){10}}
       \put(55,40){\line(1,-2){10}}
       \put(25,40){\line(-1,2){10}}
       \put(25,40){\line(-1,-2){10}}
       \put(20,40){\thicklines\circle*{1}}
       \put(60,40){\thicklines\circle*{1}}
       \put(25,40){\thicklines\circle*{1}}
       \put(55,40){\thicklines\circle*{1}}

       \put(19,34){$0$}
       \put(59.3,34){$2$}
       \put(40,40){\thicklines\vector(1,0){.0001}}

       \put(61,52){\thicklines\vector(1,2){.0001}}
       \put(61,28){\thicklines\vector(-1,2){.0001}}
       \put(19,52){\thicklines\vector(-1,2){.0001}}
       \put(19,28){\thicklines\vector(1,2){.0001}}

       \put(66.5,60){$\til\Gamma_2$}
       \put(15,60){$\til\Gamma_3$}
       \put(16.5,20){$\til\Gamma_7$}
       \put(65,20){$\til\Gamma_8$}

  \end{picture}
   \vspace{-16mm}
   \caption{Contour  $\Sigma_R$ for the RH problem \ref{rhp:R} for $R$.}
   \label{fig:ContourR}
\end{center}
\end{figure}

Since $R = D ( D^{(1)} )^{-1}$, we obtain from \eqref{eq: D asy},
\eqref{D1:asy} and \eqref{R:asy} that
\[ R_1 = D_1 - D^{(1)}_1. \]
This, together with \eqref{cctild:R} and the checkerboard pattern of
$D^{(1)}_1$, gives us
\begin{equation}\label{R2:res14}
    d = -i\sqrt{s_1}(R_1)_{1,4}.
\end{equation}

Next, we give some estimates for $J_R$ defined in \eqref{eq:JR} on $\Sigma_R$.
By \eqref{R:expsmall1}, \eqref{R:expsmall2} and the above definitions, it is
readily seen that
\begin{align*}
    J_R(\zeta) & = I + O\left(e^{\lambda(g_1-g_2)(\zeta)}\right),
       \qquad \text{uniformly for } \zeta \in \tilde \Gamma_{2} \cup\tilde \Gamma_{8}, \\
    J_R(\zeta) & = I + O\left(e^{-\lambda(g_1-g_2)(\zeta)}\right),
        \qquad \text{uniformly for } \zeta \in \tilde \Gamma_{3}\cup\tilde \Gamma_{7},
        \end{align*}
        as $\lambda \to +\infty$.

On the interval $(\ep,2-\ep)$, we have
\[ J_R(\zeta)  = D_-^{(1)}(\zeta)
    J_D(\zeta) J_{D^{(1)}}^{-1}(\zeta)
    (D_-^{(1)})^{-1}(\zeta), \]
with
\begin{align*}
 &J_D(\zeta) J_{D^{(1)}}^{-1}(\zeta)
 \\
 &= I + \frac{e^{-\lambda(g_1 + g_2)(\zeta)}}{2}
    \begin{pmatrix}
    0 & -i \left( \frac{\zeta}{2-\zeta} \right)^{1/4} & 0 & \left(\zeta (2-\zeta)\right)^{-1/4} \\
    i \left( \frac{2-\zeta}{\zeta} \right)^{1/4} & 0 & \left(\zeta (2-\zeta)\right)^{-1/4} & 0 \\
    0 & -\left(\zeta (2-\zeta) \right)^{1/4} & 0 & - i \left( \frac{2-\zeta}{\zeta} \right)^{1/4} \\
    -\left(\zeta (2-\zeta) \right)^{1/4} & 0 &   i \left( \frac{\zeta}{2-\zeta} \right)^{1/4} & 0
    \end{pmatrix}\\
    &\times (I+O(e^{-2\lam g_1(\zeta)})+O(e^{-2\lam g_2(\zeta)})),
\end{align*}
as $\lam\to+\infty$; see \eqref{jumpR:explicit} and the definition of
$J_{D^{(1)}}$. Hence,
\begin{align*}
    J_R(\zeta)  =&~ I +  \frac{e^{-\lambda(g_1 +
g_2)(\zeta)}}{2}D_-^{(1)}(\zeta)\\
    &\times
    \begin{pmatrix}
    0 & -i \left( \frac{\zeta}{2-\zeta} \right)^{1/4} & 0 & \left(\zeta(2-\zeta) \right)^{-1/4} \\
    i \left( \frac{2-\zeta}{\zeta} \right)^{1/4} & 0 & \left(\zeta (2-\zeta) \right)^{-1/4} & 0 \\
    0 & -\left(\zeta (2-\zeta) \right)^{1/4} & 0 & - i \left( \frac{2-\zeta}{\zeta} \right)^{1/4} \\
    -\left(\zeta (2-\zeta) \right)^{1/4} & 0 &   i \left( \frac{\zeta}{2-\zeta} \right)^{1/4} & 0
    \end{pmatrix}\\ & \times (I+O(e^{-2\lam g_1(\zeta)})+O(e^{-2\lam g_2(\zeta)}))
    (D_-^{(1)})^{-1}(\zeta),
    \end{align*}
for $\zeta\in(\ep,2-\ep)$. In particular, it then follows from
\eqref{eq:D1 est} and the above formula that, for the $(1,4)$ entry,
\begin{equation}\label{jumpR2:14entry}
    (J_R)_{1,4}(\zeta) = \frac{e^{-\lambda (g_1 + g_2)(\zeta)}}{2(\zeta(2-\zeta))^{1/4}}
    \left(1+O(\lam^{-1}) \right),\qquad \zeta\in(\ep,2-\ep),
\end{equation}
as $\lam \to +\infty$.

Now we observe that the function $g_1 + g_2$ has a unique minimum on
the interval $(0,2)$, which is attained at the point
\begin{equation}\label{xstar}
    \zeta^* :=\frac{2}{r_2^2+1} \in \left(0,2\right).
\end{equation}
Indeed, straightforward calculations using \eqref{g1} and \eqref{g2}
yield
\begin{align}\label{htil:saddle1}
    (g_1 + g_2)(\zeta^*) &=
    \frac{4\sqrt{2}}{3}\frac{r_2}{\sqrt{r_2^2+1}},\\
\label{htil:saddle2}
    (g_1 + g_2)'(\zeta^*) &= 0,\\
\label{htil:saddle3}
    (g_1 + g_2)''(\zeta^*) & =
        \frac{\sqrt{2}}{4}\frac{(r_2^2+1)^{3/2}}{r_2}.
\end{align}

We may assume that the number $\ep>0$ is small enough so that
\[ (g_1 + g_2)(\zeta^*) < \Re (g_2 - g_1)(\zeta),
    \qquad \text{for } \zeta \in \til \Gamma_2 \cup \til \Gamma_8, \]
and
\[ (g_1 + g_2)(\zeta^*)  < \Re (g_1 - g_2)(\zeta),
    \qquad \text{for } \zeta \in \til \Gamma_3 \cup \til \Gamma_7. \]
Thus the exponents in the exponential estimates
\eqref{R:expsmall1}--\eqref{R:expsmall2} are strictly smaller than
$-\lam  (g_1 + g_2)(\zeta^*)$ with \eqref{htil:saddle1}. Therefore,
the jump matrix $J_R$ behaves as
\begin{equation}\label{jumpR2:expbound}
    J_{R}(\zeta) = I + O\left(e^{-\lam (g_1 + g_2)(\zeta^*)}\right),
   \qquad \lam\to+\infty,
\end{equation}
uniformly on $\Sigma_R$. From \eqref{jumpR2:expbound}, standard
theory implies a similar estimate for $R$ itself:
\begin{equation}\label{eq:est for R}
    R(\zeta) = I+O\left(e^{-\lam (g_1 + g_2)(\zeta^*)}\right),
    \qquad \lam\to +\infty,
\end{equation}
uniformly for $\zeta\in\mathbb C \setminus \Sigma_R$.

Since $R_+ = R_- + (J_R - I) + (R_- - I)(J_R-I)$, it then follows
from the Sokhotski-Plemelj formula that
\begin{align*}
    R(\zeta) & = I + \frac{1}{2\pi i} \int_{\Sigma_R} \frac{(J_R-I)(x)}{x-\zeta} dx
    +  \frac{1}{2\pi i} \int_{\Sigma_R} \frac{(R_- - I)(J_R-I)(x)}{x-\zeta} dx \\
    & = I + \frac{1}{2\pi i} \int_{\Sigma_R} \frac{(J_R-I)(x)}{x-\zeta} dx
    + O\left( \frac{e^{-2\lam(g_1+g_2)(\zeta^*)}}{\zeta} \right),\qquad \lam\to +\infty,
\end{align*}
on account of \eqref{jumpR2:expbound} and \eqref{eq:est for R}. For
the $(1,4)$ entry, we get
\begin{equation} \left(R_1 \right)_{1,4} =
    - \frac{1}{2\pi i} \int_{\Sigma_R} (J_R)_{1,4}(x) dx +
    O\left( e^{-2 \lam(g_1+g_2)(\zeta^*)}\right),
\end{equation}
and by \eqref{R2:res14},
\begin{equation}\label{eq:d by R}
  d =  \frac{\sqrt{s_1}}{2 \pi}  \int_{\Sigma_R} (J_R)_{1,4}(x) dx
    +O\left( \sqrt{s_1} e^{-2 \lam(g_1+g_2)(\zeta^*)}\right),
\qquad \lam\to +\infty.
\end{equation}

The main contribution to the integral in \eqref{eq:d by R} comes from a
neighborhood of $\zeta^*$. For any given $\delta > 0$ we have
\begin{align} \nonumber
 d & =  \frac{\sqrt{s_1}}{4 \pi} \int_{\zeta^*-\delta}^{\zeta^* + \delta}
    \frac{e^{-\lambda (g_1 + g_2)(x)}}{(x(2-x))^{1/4}} dx( 1 + O(1/\lambda))
\end{align}
as $\lam\to +\infty$, by virtue of \eqref{jumpR2:14entry}. Now a
standard saddle point approximation (Laplace method, cf.
\cite{Olver,Wong}) yields
\begin{align}\label{eq:d appro}
 d & = \frac{\sqrt{s_1}}{4 \pi} \cdot \sqrt{\frac{2\pi}{\lambda (g_1+g_2)''(\zeta^*)}}
 \cdot \frac{e^{-\lam (g_1+g_2)(\zeta^*)}}{(\zeta^*(2-\zeta^*))^{1/4}}  (1 + O(1/\lambda)),
\end{align}
as $\lam\to +\infty$. Inserting \eqref{xstar}, \eqref{htil:saddle1},
\eqref{htil:saddle3} and $\lambda = s_1^{3/2}$ into \eqref{eq:d
appro}, we finally obtain
\begin{align*}
    d & = \frac{\sqrt{2}}{4 \sqrt{\pi} s_1^{1/4}}
        \cdot \left( \frac{4}{\sqrt{2}} \frac{r_2}{(r_2^2+1)^{3/2}} \right)^{1/2}
        \cdot \left( \frac{r_2^2+1}{2} \frac{r_2^2+1}{2r_2^2} \right)^{1/4}
        \cdot  e^{-\frac{4 \sqrt{2}}{3} \frac{r_2}{\sqrt{r_2^2+1}} s_1^{3/2}}
        (1 + O(s_1^{-3/2})) \\
     & = \frac{1}{2 \sqrt{\pi} s_1^{1/4}} \left(\frac{1}{2(r_2^2+1)}\right)^{1/4}
     \cdot  e^{-\frac{4 \sqrt{2}}{3} \frac{r_2}{\sqrt{r_2^2+1}} s_1^{3/2}}
        (1 + O(s_1^{-3/2})), \qquad s_1 \to +\infty,
    \end{align*}
which is \eqref{eq:asy of d} with $r_1=1$ and $s_2=0$, as desired.

This completes the proof of Proposition \ref{prop:large s1
solvability}. $\bol$


\section{Proof of Theorem~\ref{theorem:solvability}}
\label{section:vanishing}

In this section we prove Theorem~\ref{theorem:solvability} on the solvability
of the RH problem for $M(\zeta)$ by using the technique of a vanishing lemma.
To this end we follow basically the scheme laid out in \cite{DKMVZ1,FZ,Zhou},
although the argument is somewhat more involved because our RH problem is of
size $4\times 4$ whereas the usual dimensions treated in the literature is
$2\times 2$.

Following \cite{DKMVZ1}, the proof consists of three steps, which as
in \cite{DKMVZ1} are called Step 1, Step 2 and Step 3.

\subsubsection*{Step 1: Fredholm property}

Standard theory show that the RH problem for $M$ is associated to a
singular integral operator. The first step is to show that this
operator is Fredholm. To this end we are going to apply a series of
transformations $M\mapsto \Mdec\mapsto \Mmod$. The first
transformation $M\mapsto \Mdec$ is defined by
\begin{equation}
\label{def:Mdec} \Mdec(\zeta) = M(\zeta)\Theta(\zeta),
\end{equation}
where
\begin{equation}\label{def:Theta}
\Theta(\zeta) =
\diag\left(e^{\theta_1(\zeta)},e^{\theta_2(\zeta)},e^{-\theta_1(\zeta)},e^{-\theta_2(\zeta)}\right).
\end{equation}
This transformation will kill the exponential factor in the
asymptotics \eqref{M:asymptotics}, at the expense of complicating
the jump matrices.

The second transformation $\Mdec\mapsto \Mmod$ is defined by
\begin{equation}
\label{def:Mmod} \Mmod(\zeta; r_1, r_2, s_1, s_2) = \Mdec(\zeta;
r_1, r_2, s_1, s_2)
    (\Mdec)^{-1}(\zeta; r_1, r_2, s_1^*,s_2),
\end{equation}
where $s_1^*$ is a fixed but sufficiently large real number for
which we already know that $M$ and therefore $\Mdec$ exists; see
Proposition \ref{prop:large s1 solvability}. Then $\Mmod$ satisfies
the following RH problem.

\begin{rhp}\label{rhp:Mmod}
\textrm{}
\begin{itemize}
\item[\rm (1)] $\Mmod(\zeta)$ is analytic for $\zeta\in\mathbb C\setminus\left(\bigcup_{k=1}^{4}\Gamma_k\cup \bigcup_{k=6}^{9}\Gamma_k\right)$.
\item[\rm (2)] For $\zeta\in\Gamma_k$, $\Mmod(\zeta)$
has the jump
$$ \Mmod_+(\zeta) =
\Mmod_-(\zeta)J_k^{(2)}(\zeta),
$$
where
$$ J_k^{(2)}(\zeta)=\Mdec_-(\zeta;r_1,r_2,s_1^*,s_2)\Theta_-^{-1}(\zeta)J_k
    \Theta_+(\zeta)(\Mdec_+)^{-1}(\zeta;r_1,r_2,s_1^*,s_2),
$$
and where $J_k$ is the original jump matrix on $\Gamma_k$ denoted in
Figure~\ref{fig:modelRHP}.
\item[\rm (3)] As $\zeta\to\infty$, we have $\Mmod(\zeta)=I + O(\zeta^{-1})$.
\end{itemize}
\end{rhp}

Since the transformations $M\mapsto \Mdec\mapsto \Mmod$ are
invertible, we have that the original RH problem for $M$ is solvable
if and only if the one for $\Mmod$ is solvable.

One checks that the jump matrices $J_k^{(2)}(\zeta)$ converge to the
identity matrix whenever $\zeta\to\infty$ or $\zeta\to 0$, along any
of the $8$ rays $\Gamma_k$, $k=1,\ldots,4,6,\ldots,9$. In other
words, the jump matrices are normalized at infinity and continuous
at zero. Then it follows from the techniques of Deift et al.
\cite{DKMVZ1} that the singular integral operator associated to the
RH problem for $\Mmod(\zeta)$ is Fredholm. The same statement then
holds for the original RH problem for $M(\zeta)$.

\subsubsection*{Step 2: the Fredholm index is zero}

The next step is to show that the Fredholm index of the Fredholm
operator associated to $\Mmod$ is zero. This follows by a continuity
argument in the same way as in \cite{DKMVZ1}. Indeed, we have that
the Fredholm index is a continuous function of $s_1$ which takes
only integer values. For large values of $s_1$, the Fredholm index
is equal to zero, and therefore this must hold for all $s_1$.

\subsubsection*{Step 3: triviality of solution to the homogeneous
version of the Riemann-Hilbert problem}

The third and final step is to show that the \lq homogeneous\rq\
version of the RH problem has only the trivial solution. This is
also known as a vanishing lemma. As in \cite[Page~1402]{DKMVZ1}, this reduces to studying the following RH problem,
which is the homogeneous version of the RH problem for $M(\zeta)$.

\begin{rhp}\label{rhp:Mhom} We look for a $4\times 4$ matrix-valued function
$\Mhom(\zeta)$ satisfying
\begin{itemize}
\item[\rm (1)] $\Mhom(\zeta)$ is analytic for $\zeta\in\mathbb C\setminus\left(\bigcup_{k=0}^{9}
\Gamma_k\right)$.
\item[\rm (2)] For $\zeta\in\Gamma_k$, $\Mhom(\zeta)$ has the same jumps as
$M(\zeta)$, see Figure~\ref{fig:modelRHP}.
\item[\rm (3)] As $\zeta\to\infty$, we have
\begin{align}
\label{Mhom:asymptotics} \Mhom(\zeta) =& \left(O(\zeta^{-1})\right)
\diag((-\zeta)^{-1/4},\zeta^{-1/4},(-\zeta)^{1/4},\zeta^{1/4})
\nonumber
\\
&\times \frac{1}{\sqrt{2}}
\begin{pmatrix} 1 & 0 & -i & 0 \\
0 & 1 & 0 & i \\
-i & 0 & 1 & 0 \\
0 & i & 0 & 1 \\
\end{pmatrix}
\Theta^{-1}(\zeta),
\end{align}
where we recall the notations \eqref{def:Theta} and
\eqref{def:theta1}--\eqref{def:theta2}.
\item[\rm (4)] $\Mhom(\zeta)$ is bounded near $\zeta=0$.
\end{itemize}
\end{rhp}

Note that the only difference between the RH problems for $M(\zeta)$
and $\Mhom(\zeta)$ is in the leftmost factor of the asymptotics;
compare \eqref{M:asymptotics} with \eqref{Mhom:asymptotics}.

We need to show that the only solution to RH problem~\ref{rhp:Mhom}
is when $\Mhom$ is the zero matrix.

First we apply a contour deformation to bring all jumps to the real
axis. Denote by $\Omega_k$, $k=0,\ldots,9$ the region between the
rays $\Gamma_k$ and $\Gamma_{k+1}$ in Figure~\ref{fig:modelRHP},
with $\Gamma_{10}:=\Gamma_0$. We define $\widehat\Mhomdef(\zeta)$
for $\zeta\in\mathbb C\setminus\bigcup_{k=0}^{9} \Gamma_k$ by
\begin{equation}\label{M3hat}
\widehat\Mhomdef = \left\{\begin{array}{ll}\Mhom J_1J_2,& \qquad \textrm{ for $\zeta\in\Omega_0$},\\
\Mhom J_2,& \qquad\textrm{ for $\zeta\in\Omega_1$},\\
\Mhom,& \qquad\textrm{ for $\zeta\in\Omega_2$},\\
\Mhom J_3^{-1},& \qquad\textrm{ for $\zeta\in\Omega_3$},\\
\Mhom J_4^{-1}J_3^{-1},& \qquad \textrm{ for $\zeta\in\Omega_4$},\\
\Mhom J_6J_7,& \qquad\textrm{ for $\zeta\in\Omega_5$,}\\
\Mhom J_7,& \qquad\textrm{ for $\zeta\in\Omega_6$,}\\
\Mhom,& \qquad\textrm{ for $\zeta\in\Omega_7$,}\\
\Mhom J_8^{-1},& \qquad\textrm{ for $\zeta\in\Omega_8$,}\\
\Mhom J_9^{-1}J_8^{-1},& \qquad\textrm{ for $\zeta\in\Omega_9$,}
\end{array}\right.
\end{equation}
where $J_k$, $k=0,\ldots,9$, denotes the jump matrix for $M(\zeta)$
on $\Gamma_k$ in Figure~\ref{fig:modelRHP}. Next we set
\[ \Mhomdef(\zeta) = \widehat\Mhomdef(\zeta)\Theta(\zeta). \]

\begin{rhp}\label{rhp:Mhomdef} The $4\times 4$ matrix-valued function $\Mhomdef(\zeta) $ satisfies the
following RH problem:
\begin{itemize}
\item[\rm (1)] $\Mhomdef(\zeta)$ is analytic for $\zeta\in\mathbb C\setminus\mathbb R$.
\item[\rm (2)] For $x\in\mathbb R$, we have the jump relation
\begin{equation}\label{jumps:Mhomdef}
\Mhomdef_+(x) = \Mhomdef_-(x) \Theta_-^{-1}(x)\begin{pmatrix}
1&0&1&1\\
0&1&1&1\\
0&0&1&0\\
0&0&0&1
\end{pmatrix}\Theta_+(x),
\end{equation}
where the orientation of the real axis is from left to right.
\item[\rm (3)] As $\zeta\to\infty$, we have
\begin{equation}
\label{Mhomdef:asymptotics} \Mhomdef(\zeta) = 
O(\zeta^{-3/4}).
\end{equation}
\item[\rm (4)] $\Mhomdef(\zeta)$ is bounded near $\zeta=0$.
\end{itemize}
\end{rhp}

\begin{proof}
The fact that $\Mhomdef(\zeta)$ does not have jumps on $\Gamma_k$,
$k=1,...,4,6,...,9$, follows immediately from \eqref{M3hat}. The
jump matrix on the real axis \eqref{jumps:Mhomdef} follows from the
relations
$$ J_8J_9J_0J_1J_2 = (J_3J_4J_5J_6J_7)^{-1} =
\begin{pmatrix}
1&0&1&1\\
0&1&1&1\\
0&0&1&0\\
0&0&0&1
\end{pmatrix},
$$
also note that we reverse the orientation of the negative real axis.
As for the asymptotics \eqref{Mhomdef:asymptotics}, one should check
that $\Theta^{-1}J_1J_2\Theta$ is bounded for $\zeta\in\Omega_0$,
$\Theta^{-1}J_2\Theta$ is bounded for $\zeta\in\Omega_1$, and so on.
Equivalently, $e^{\theta_1(\zeta)}$ should be bounded for
$\zeta\in\Omega_0$, $e^{\theta_1(\zeta)-\theta_2(\zeta)}$ should be
bounded for $\zeta\in\Omega_1$, and so on. These statements are
easily checked from \eqref{def:theta1}--\eqref{def:theta2} and
\eqref{def:angles}.
\end{proof}

Next, we define a new $4\times 4$ matrix-valued function
$\Mhomtris(\zeta)$ by
\begin{equation}\label{def:Mhomtris}
\Mhomtris(\zeta) = \left\{\begin{array}{ll}\Mhomdef(\zeta),&
\zeta\textrm{ in
lower half plane of $\mathbb C$,}\\ \Mhomdef(\zeta)\begin{pmatrix}0&-I_2 \\
I_2& 0\end{pmatrix},& \zeta\textrm{ in upper half plane of $\mathbb
C$.}\end{array}\right.
\end{equation}
Then $\Mhomtris(\zeta)$ has the jump
\begin{equation}\label{jumps:Mhomtris} \Mhomtris_+(x) =
\Mhomtris_-(x)J(x),\qquad x\in\mathbb R,
\end{equation}
 with
\begin{align}\nonumber
J(x) &:= \Theta_-^{-1}(x)\begin{pmatrix}
1&0&1&1\\
0&1&1&1\\
0&0&1&0\\
0&0&0&1
\end{pmatrix}\Theta_+(x)\begin{pmatrix}0&-I_2 \\
I_2& 0\end{pmatrix}
\\
\label{def:Jx} &= \Theta_-^{-1}(x)\begin{pmatrix}
1&1&-1&0\\
1&1& 0&-1\\
1&0&0&0\\
0&1&0&0
\end{pmatrix}\Theta_-^{-H}(x),
\end{align}
and where the superscript $^{-H}$ denotes the inverse Hermitian conjugate.
Similarly, we will use $^{H}$ to denote the Hermitian conjugate.

Now define a new $4\times 4$ matrix-valued function $Q(\zeta) =
\Mhomtris(\zeta)(\Mhomtris(\overline{\zeta}))^{H}$. Then $Q(\zeta)$
is analytic in the upper half plane of $\mathbb C$ and it decays
with a power $\zeta^{-3/2}$ as $\zeta\to\infty$. A standard argument
based on contour deformation and Cauchy's theorem shows that
$$ \int_{\mathbb R}Q_+(x)\ dx = 0.
$$
Hence,
$$ \int_{\mathbb R} \Mhomtris_-(x)J(x)(\Mhomtris_-(x))^H\ dx = 0.
$$
By adding this relation to its Hermitian conjugate, we find
\begin{equation}\label{vanishing:integralzero} \int_{\mathbb R}
\Mhomtris_-(x)\left(J(x)+J^H(x)\right)(\Mhomtris_-(x))^H\ dx = 0.
\end{equation}
But from \eqref{def:Jx} we have that
\begin{align*}  J(x)+J^H(x) &=
2\Theta_-^{-1}\begin{pmatrix}
1&1&0&0\\
1&1& 0&0\\
0&0&0&0\\
0&0&0&0
\end{pmatrix}\Theta_-^{-H}
\\
&= 2\begin{pmatrix} e^{-\theta_{1,-}(x)}\\ e^{-\theta_{2,-}(x)} \\
0\\0
\end{pmatrix}
\begin{pmatrix} e^{-\overline{\theta_{1,-}(x)}}&
 e^{-\overline{\theta_{2,-}(x)}} & 0&0
\end{pmatrix}.
\end{align*}
Substituting this expression in \eqref{vanishing:integralzero}
yields
\begin{equation}\int_{\mathbb R}
\Mhomtris_-(x)\begin{pmatrix} e^{-\theta_{1,-}(x)}\\ e^{-\theta_{2,-}(x)} \\
0\\0
\end{pmatrix}\begin{pmatrix} e^{-\overline{\theta_{1,-}(x)}}& e^{-\overline{\theta_{2,-}(x)}} & 0&0
\end{pmatrix}(\Mhomtris_-(x))^H\ dx = 0,
\end{equation}
which obviously implies that
\begin{equation}\label{eq:M5 fin}
\Mhomtris_-(x)\begin{pmatrix} e^{-\theta_{1,-}(x)}\\ e^{-\theta_{2,-}(x)} \\
0\\0
\end{pmatrix}\equiv 0,\qquad x\in\mathbb R.
\end{equation}
Inserting \eqref{eq:M5 fin} into the RH problem for $\Mhomtris$, we
see that the jump relation \eqref{jumps:Mhomtris}--\eqref{def:Jx}
reduces to
$$ \Mhomtris_+(x) = \Mhomtris_-(x) \Theta_-^{-1}(x)\begin{pmatrix}
0&0&-1&0\\
0&0& 0&-1\\
1&0&0&0\\
0&1&0&0
\end{pmatrix}\Theta_-^{-H}(x).
$$
By tracing back the transformation $\Mhomdef\mapsto\Mhomtris$, we
find
$$ \Mhomdef_+(x) = \Mhomdef_-(x) \Theta_-^{-1}(x)\Theta_+(x).
$$
This now decouples into four scalar RH problems for the individual
column vectors. For each of these scalar RH problems, one can use an
argument based on Carlson's theorem \cite[Page 1406]{DKMVZ1} to
conclude that it has only the zero solution. This then implies that
also $\Mhomdef$ must be the zero matrix. By tracing back the
transformation $\Mhom\mapsto\Mhomdef$, the same conclusion holds for
$\Mhom$. Thus we see that the homogeneous version of the RH problem
for $M$ indeed has only the trivial solution. This ends the proof of
Step 3, and therefore the proof of existence of $M$ by means of the
vanishing lemma. $\bol$

\section{Proof of Theorem~\ref{theorem:Painleve2modelrhp}} \label{subsection:laxpair}
\label{section:proofP2M}

In this section we prove Theorem~\ref{theorem:Painleve2modelrhp} on
the Painlev\'e~II behavior of the numbers $c,\wtil c$ and $d$ in
\eqref{eq:defM1}.

\subsection{Symmetry properties}

We will need a few properties of the RH problem~\ref{rhp:modelM}.
The first property follows from exploiting the symmetry of the
problem. In what follows, we use the elementary permutation matrix
\[ J=\begin{pmatrix}0&1\\1&0\end{pmatrix} \]
and recall that $I_2$ denotes the $2\times 2$ identity matrix.

\begin{lemma}\label{lemma:symmetries}

\begin{enumerate}
\item[\rm (a)]
For any fixed $r_1, r_2>0$ and $s_1,s_2\in\mathbb R$, we have the
symmetry relations
\begin{equation}\label{symmetry:conjugate}
\overline{M(\overline{\zeta})} =
\begin{pmatrix} I_2 & 0 \\ 0 & -I_2 \end{pmatrix}
M(\zeta) \begin{pmatrix} I_2 & 0 \\ 0 & -I_2
\end{pmatrix},
\end{equation}
where the bar denotes complex conjugation and
\begin{equation}\label{symmetry:inversetranspose}
M^{-T}(\zeta) =
\begin{pmatrix} 0 & I_2 \\ -I_2 & 0 \end{pmatrix}
M(\zeta) \begin{pmatrix} 0 & -I_2 \\ I_2 & 0 \end{pmatrix},
\end{equation} where the superscript ${}^{-T}$ denotes the inverse transpose.
\item[\rm (b)]
Denoting with $M(\zeta;r_1, r_2, s_1, s_2)$ the solution $M(\zeta)$
with parameters $r_1, r_2, s_1, s_2$, then we have the relation
\begin{equation}\label{symmetry:special}
M(-\zeta; r_1, r_2, s_1, s_2) =
 \begin{pmatrix} J & 0 \\ 0 & -J \end{pmatrix}
    M(\zeta; r_2, r_1, s_2, s_1)   \begin{pmatrix} J & 0
\\ 0 & -J
\end{pmatrix}.
\end{equation}
\end{enumerate}
\end{lemma}

\begin{proof} One easily checks that the left and right hand sides of
\eqref{symmetry:conjugate} satisfy the same RH problem. Then
\eqref{symmetry:conjugate} follows from the uniqueness of the
solution to this RH problem. The same argument applies to
\eqref{symmetry:inversetranspose} and \eqref{symmetry:special}.
\end{proof}

\begin{corollary}
For any fixed $r_1, r_2>0$ and $s_1, s_2 \in\mathbb R$, we have
\begin{equation*}  \overline{M_1} = \begin{pmatrix} I_2 & 0 \\ 0 & -I_2 \end{pmatrix}
M_1 \begin{pmatrix} I_2 & 0 \\ 0 & -I_2
\end{pmatrix},
\end{equation*}
and
$$  -M_1^T = \begin{pmatrix} 0 & I_2 \\ -I_2 & 0 \end{pmatrix} M_1 \begin{pmatrix}
0 & -I_2
\\ I_2 & 0 \end{pmatrix}.
$$
Consequently, $M_1$ takes the form
\begin{equation}\label{M1:explicit}
M_1 = \begin{pmatrix} a & b & ic & id \\
-\wtil b & -\wtil a & id & i\wtil c \\
ie & if & -a & \wtil b \\
if & i\wtil e & -b & \wtil a
\end{pmatrix},
\end{equation}
where $a,\wtil a,b,\wtil b,c,\wtil c,\ldots$ are real constants
depending parametrically on $r_1, r_2, s_1, s_2$.

We also have
$$
-M_1(r_1, r_2, s_1, s_2) =
\begin{pmatrix} J & 0 \\ 0 & -J \end{pmatrix}
M_1(r_2, r_1, s_2, s_1)  \begin{pmatrix} J & 0 \\ 0 & -J
\end{pmatrix},
$$
and consequently we have in \eqref{M1:explicit} that
\begin{equation}\label{eqs:a til a}
\begin{aligned}
    a(r_1, r_2, s_1,s_2) &= \wtil a(r_2, r_1, s_2,s_1), \quad
    b(r_1, r_2, s_1,s_2) = \wtil b(r_2, r_1, s_2,s_1), \\
    c(r_1, r_2, s_1,s_2) &= \wtil c(r_2, r_1, s_2,s_1), \quad
    e(r_1, r_2, s_1,s_2) = \wtil e(r_2, r_1, s_2,s_1).
\end{aligned}
\end{equation}
\end{corollary}

\begin{proof} This follows from \eqref{M:asymptotics} and
Lemma \ref{lemma:symmetries}.
\end{proof}

\subsection{Lax pair equations}

We next obtain linear differential equations for $M$ with respect to
both $\zeta$ and $s_1, s_2$. This system of differential equations
has a Lax pair form and the compatibility condition of the Lax pair
will then lead to the Painlev\'e~II equation in
Theorem~\ref{theorem:Painleve2modelrhp}.


We start with a differential equation with respect to $\zeta$.

\begin{proposition}\label{prop:diffeq1}
 We have the differential equation
\begin{equation}\label{diffeq1}
\frac{\partial M}{\partial\zeta} = U M,
\end{equation}
where
\begin{equation}\label{def:U}
    U:=\begin{pmatrix} 0 & 0 & i r_1 & 0 \\ 0 & 0 & 0 & i r_2 \\ i(r_1 \zeta-s_1) & 0 & 0 & 0 \\
0 & -i(r_2 \zeta+ s_2) & 0 & 0
\end{pmatrix}+ \left[ M_1, i r_1 E_{3,1} - i r_2 E_{4,2} \right],
\end{equation}
and where $\left[ \cdot, \cdot \right]$ denotes the commutator.
\end{proposition}

\begin{proof}
Since the jump matrices in the RH problem for  $M$ do not depend on
$\zeta$, we have that $\frac{\partial M}{\partial \zeta}$ has the
same jump properties as $M$. It follows that $U := \frac{\partial
M}{\partial \zeta} M^{-1}$ is entire. From the asymptotic behavior
of $M$ in \eqref{M:asymptotics}, we have as $\zeta \to \infty$,
\begin{align}\label{diffeq1:RH}
    U =&~ \frac{1}{4\zeta} \diag ( -1 , -1 , 1 , 1 )
      + \left(I+\frac{M_1}{\zeta}+\frac{M_2}{\zeta^2}+ \cdots\right) \nonumber \\
    &\times \begin{pmatrix} 0 & 0 & i(r_1 -s_1\zeta^{-1}) & 0 \\ 0 & 0 & 0 & i(r_2+s_2\zeta^{-1}) \\
    i(r_1\zeta-s_1) & 0 & 0 & 0 \\
    0 & -i(r_2 \zeta+s_2) & 0 & 0
\end{pmatrix}
\nonumber \\
&\times
\left(I-\frac{M_1}{\zeta}+\frac{M_1^2-M_2}{\zeta^2}+\cdots\right)
+O(\zeta^{-2}).
\end{align}
Thus the entries of $U$ are polynomial in $\zeta$ with degree at
most one. Dropping the non-polynomial terms in \eqref{diffeq1:RH},
we obtain \eqref{def:U}.
\end{proof}

The proof of Proposition~\ref{prop:diffeq1} also yields the
following.

\begin{lemma}\label{lemma:threerelations}
The entries $a,\wtil a,b,\wtil b,\ldots$ of $M_1$ in
\eqref{M1:explicit} satisfy the identities
\begin{align} \label{threeextrarelations:1}
    (2a + c^2) r_1 & = r_2 d^2+ s_1, \\
\label{threeextrarelations:2}
    (2\wtil a +\wtil c^2)r_2 & = r_1 d^2 + s_2, \\
\label{threeextrarelations:3}
    r_2b- r_1 \wtil b &= (r_2 \wtil c - r_1 c)d.
 \end{align}
 \end{lemma}
 \begin{proof}
The coefficient of $\zeta^{-1}$ in \eqref{diffeq1:RH} is equal to
zero. For the upper right $2 \times 2$ block, we then obtain by
using \eqref{M1:explicit} that
\begin{multline*}
    \begin{pmatrix}-is_1 & 0 \\ 0 & is_2 \end{pmatrix} +
    \begin{pmatrix} a & b \\-\wtil b&-\wtil a \end{pmatrix}
    \begin{pmatrix} i r_1 & 0 \\ 0 & i r_2 \end{pmatrix}
    -\begin{pmatrix} i r_1 & 0 \\ 0 & ir_2 \end{pmatrix}
    \begin{pmatrix}-a&\wtil b\\- b&\wtil a\end{pmatrix} \\
    +
    \begin{pmatrix}c&d\\ d&\wtil c\end{pmatrix}
    \begin{pmatrix}i r_1&0 \\ 0&-ir_2\end{pmatrix}
    \begin{pmatrix}c&d\\ d&\wtil c\end{pmatrix}= 0.
\end{multline*}
These are four identities and they give us
\eqref{threeextrarelations:1}--\eqref{threeextrarelations:3}.
\end{proof}

In the special case $r_1 = r_2$ and $s_1 = s_2$, we see from
\eqref{eqs:a til a} that $\wtil a = a$, $\wtil b= b$, and so on. The
equation \eqref{threeextrarelations:3} then reduces to $0=0$, while
\eqref{threeextrarelations:1} and \eqref{threeextrarelations:2} are
the same. So in that case, the system
\eqref{threeextrarelations:1}--\eqref{threeextrarelations:3} reduces
to the single relation
\begin{equation}
 \label{threeextrarelations:symm} (2a + c^2) r= r^2 d^2+s,
    \qquad \text{ with } r = r_1 = r_2, \, s = s_1 = s_2.
\end{equation}

We obtain more differential equations by taking a derivative of $M$
with respect to the parameters $s_1, s_2$.

\begin{proposition}\label{prop:diffeq2}
We have the differential equation
\begin{equation}\label{diffeq2}
\frac{\partial M}{\partial s_j} = V_j M, \qquad j=1,2,
\end{equation}
where
\begin{equation}\label{def:V1Lax}
    V_1 = \begin{pmatrix} 0 & 0 & -2i & 0 \\ 0 & 0 & 0 & 0 \\
    -2i\zeta & 0 & 0 & 0 \\
    0 & 0& 0 & 0
    \end{pmatrix}
    + \left[ M_1, - 2i E_{3,1} \right]
\end{equation}
and
\begin{equation}\label{def:V2Lax}
    V_2 = \begin{pmatrix} 0 & 0 & 0 & 0 \\ 0 & 0 & 0 & 2i \\
    0 & 0 & 0 & 0 \\
    0 & -2i \zeta & 0 & 0
    \end{pmatrix}
    + \left[ M_1, - 2i E_{4,2} \right].
\end{equation}
\end{proposition}

\begin{proof} Since the jumps in the RH problem for $M$ do not
depend on $s_j$, we have that $\partial M/ \partial s_j$ has the
same jumps as $M$ and so $V_j = \left( \partial M/ \partial s_j
\right) M^{-1}$ is entire. As $\zeta \to \infty$, we find for $j=1$,
\begin{align}\label{diffeq2:RH}
    V_1 =
    \left(I+\frac{M_1}{\zeta} + \cdots\right)
    \begin{pmatrix} 0 & 0 & -2i & 0 \\ 0 & 0 & 0 & 0 \\
    -2i\zeta & 0 & 0 & 0 \\
    0 & 0 & 0 & 0
    \end{pmatrix}
    \left(I-\frac{M_1}{\zeta}+ \cdots\right).
\end{align}
Keeping only the polynomial terms in $\zeta$, we obtain
\eqref{def:V1Lax}. The other relation \eqref{def:V2Lax} is proved
similarly.
\end{proof}

\subsection{Compatibility and
proofs of \eqref{d:Painleve2}--\eqref{ctil:Hamiltonian}} The
compatibility condition for the two differential equations
\eqref{diffeq1} and \eqref{diffeq2} is the zero curvature relation
\begin{equation}\label{compat:laxpair}
 \frac{\partial U}{\partial s_j} -
\frac{\partial V_j}{\partial\zeta} + UV_j - V_jU = 0, \qquad j=1,2.
\end{equation}

This is, in view of \eqref{def:U}, \eqref{def:V1Lax} and
\eqref{def:V2Lax},
\begin{align} \label{eq:compat1}
    \frac{\partial}{\partial s_1} \left[ M_1, i r_1 E_{3,1} - i r_2 E_{4,2} \right]
    & = - i E_{3,1} + \left[ V_1, U \right], \\
   \label{eq:compat2}
    \frac{\partial}{\partial s_2} \left[ M_1, i r_1 E_{3,1} - i r_2 E_{4,2} \right]
    & = - i E_{4,2} + \left[ V_2, U \right].
    \end{align}
Both \eqref{eq:compat1} and \eqref{eq:compat2} give us a system of
16 differential equations for the entries of $M_1$.

With the above preparations, we are ready to prove
\eqref{d:Painleve2}--\eqref{ctil:Hamiltonian} in Theorem
\ref{theorem:Painleve2modelrhp} concerning the Painlev\'e~II
behavior of the numbers $c,\wtil c$ and $d$. To this end, we work
with \eqref{eq:compat1} and first derive the differential equation
satisfied by $d$.

The entries in the matrix relation \eqref{eq:compat1} can be obtained from a
lengthy calculation or with the help of a symbolic software package such as
Maple. For the $(1,1)$ and $(2,2)$ entries of \eqref{eq:compat1}, this yields
\begin{equation} \label{eq:compat3a}
    r_1\frac{\partial c}{\partial s_1} = 2r_2 d^2 + 2s_1, \qquad
     \frac{\partial \wtil c}{\partial s_1} = 2d^2,
     \end{equation}
which imply
\begin{equation} \label{eq:compat3}
    \frac{\partial (r_1 c - r_2 \wtil c)}{\partial s_1} = 2s_1.
    \end{equation}
The $(1,2)$ and $(2,1)$ entries of \eqref{eq:compat1} give
expressions for the partial derivative of $d$:
\begin{equation} \label{eq:compat4}
    \frac{\partial d}{\partial s_1} =
2 cd - 2 \wtil b = \frac{r_2}{r_1} (2 \wtil c d - 2 b),
    \end{equation}
and the $(1,4)$ entry gives
\begin{equation} \label{eq:compat5}
    -\frac{\partial( r_1 b + r_2 \wtil b)}{\partial s_1}
    = r_2(4 a d + 4 \wtil a d +2 b \wtil c + 2 \wtil b c) + 2 s_2 d.
    \end{equation}

On the other hand, the identity \eqref{threeextrarelations:3}
implies
\[ r_1 (r_1 b  + r_2 \wtil b)  =
(r_1^2 + r_2^2) b + r_2 (r_1 c - r_2 \wtil c) d, \] which upon
differentiation leads to
\begin{equation} \label{eq:compat6}
    r_1 \frac{\partial (r_1 b + r_2 \wtil b)}{\partial s_1}
    = (r_1^2 + r_2^2) \frac{\partial b}{\partial s_1} + r_2 (r_1 c - r_2 \wtil c) \frac{\partial d}{\partial s_1}
    + 2 r_2 s_1 d,
    \end{equation}
where we have made use of \eqref{eq:compat3}. Combining this with
\eqref{eq:compat5} yields
\begin{equation} \label{eq:compat7}
    (r_1^2 + r_2^2) \frac{\partial b}{\partial s_1} + r_2 (r_1 c - r_2 \wtil c) \frac{\partial d}{\partial s_1}
    + 2 (r_1 s_2 + r_2 s_1) d
    = - r_1 r_2(4 a d + 4 \wtil a d +2 b \wtil c + 2 \wtil b c).
    \end{equation}
We next eliminate $a$ and $\wtil a$ from the right-hand side of
\eqref{eq:compat7} with the help of \eqref{threeextrarelations:1}
and \eqref{threeextrarelations:2}. This gives us
\begin{align} \label{eq:compat8}
&(r_1^2 + r_2^2) \frac{\partial b}{\partial s_1} + r_2 (r_1 c - r_2
\wtil c) \frac{\partial d}{\partial s_1}
    + 2 (r_1 s_2 + r_2 s_1) d \nonumber \\
&= - 2(r_1^2 + r_2^2) d^3 - 2 (r_1 s_2 + r_2 s_1) d
        + 2 r_1 r_2( c^2 d - \wtil b c + \wtil c^2 d -  b \wtil c).
\end{align}
We move the last term in the left-hand side of \eqref{eq:compat8} to the right,
and rewrite the last term in terms of $\partial d/\partial s_1$ by using
\eqref{eq:compat4} and \eqref{threeextrarelations:3}. It then follows that
\begin{align*}
&(r_1^2 + r_2^2) \frac{\partial b}{\partial s_1} + r_2 (r_1 c - r_2
\wtil c)
\frac{\partial d}{\partial s_1} \\
&= -4 (r_1 s_2 + r_2 s_1) d
     - 2(r_1^2 + r_2^2) d^3
 + r_1 (r_2 c   +  r_1 \wtil c) \frac{\partial d}{\partial s_1},
\end{align*}
or equivalently,
\begin{align} \label{eq:compat9}
     \frac{\partial b}{\partial s_1} - \wtil c \frac{\partial d}{\partial s_1} =
    - \frac{4(r_1 s_2 + r_2 s_1)}{r_1^2 + r_2^2} d - 2 d^3.
\end{align}

Taking a derivative of the second identity in \eqref{eq:compat4}
with respect to $s_1$ and using \eqref{eq:compat9} to eliminate
$\partial b/\partial s_1 - \wtil c \partial d/\partial s_1$, we
obtain
\begin{align} \label{eq:compat10}
    \frac{\partial^2 d}{\partial s_1^2} & =
        2 \frac{r_2}{r_1} \left( \frac{\partial \wtil c}{\partial s_1} d +
    \frac{4(r_1 s_2 + r_2 s_1)}{r_1^2 + r_2^2} d + 2 d^3 \right).
\end{align}
This, together with the fact that $\partial \wtil c/\partial s_1 = 2
d^2$ (see \eqref{eq:compat3a}), implies
\begin{align} \label{eq:compat11}
    \frac{\partial^2 d}{\partial s_1^2} & =
        2 \frac{r_2}{r_1} \left( 2 d^3  +
    \frac{4(r_1 s_2 + r_2 s_1)}{r_1^2 + r_2^2} d + 2 d^3 \right) \nonumber \\
    & = 8 \frac{r_2}{r_1}  \frac{(r_1 s_2 + r_2 s_1)}{r_1^2 + r_2^2} d
        + 8 \frac{r_2}{r_1} d^3.
    \end{align}

The differential equation \eqref{eq:compat11} is a scaled and
shifted version of the Painlev\'e II equation. Indeed, we have that
$q$ satisfies $q'' = s q + 2q^3$, if and only if
\[ f(s) = \alpha q(\beta s + \gamma) \]
satisfies
\begin{equation} \label{eq:compat12}
    f'' = \beta^2 (\beta s+ \gamma) f + 2 \frac{\beta^2}{\alpha^2} f^3.
    \end{equation}
Comparing this with \eqref{eq:compat11}, we see that we need
$\alpha$, $\beta$ and $\gamma$ so that
\[ 8 \frac{r_2}{r_1} = 2 \frac{\beta^2}{\alpha^2}, \qquad
    8 \frac{r_2}{r_1}  \frac{r_2}{r_1^2 + r_2^2} =
    \beta^3, \qquad
    8 \frac{r_2}{r_1}  \frac{r_1 s_2}{r_1^2 + r_2^2} =
    \beta^2 \gamma,
 \]
which means
\[ \alpha = \frac{(r_1 r_2)^{1/6} }{(r_1^2 + r_2^2)^{1/3}}, \qquad
    \beta = \frac{2r_2^{2/3}}{r_1^{1/3} (r_1^2 + r_2^2)^{1/3}}, \qquad
    \gamma = \frac{2r_1^{2/3} s_2}{r_2^{1/3} (r_1^2 + r_2^2)^{1/3}}.  \]

Therefore, we have proved that
\begin{equation}\label{eq:d in q}
d = \alpha q(\beta s_1 + \gamma) =
    \frac{(r_1 r_2)^{1/6} }{(r_1^2 + r_2^2)^{1/3}}
    q \left( \frac{2 (r_1 s_2 + r_2 s_1  )}{ (r_1r_2)^{1/3} (r_1^2 + r_2^2)^{1/3}} \right)
\end{equation}
with $q$ being a solution of the Painlev\'e II equation, which is
\eqref{d:Painleve2}. The fact that $q$ is the Hastings-McLeod
solution follows from the asymptotic behavior of $d$ in
\eqref{eq:asy of d}. To see this, we rewrite \eqref{eq:asy of d} in
the following form:
\begin{equation}
d = \frac{(r_1r_2)^{1/6}}{(r_1^2+r_2^2)^{1/3}}\Ai\left( \frac{2 (r_1
s_2 + r_2 s_1  )}{ (r_1r_2)^{1/3} (r_1^2 + r_2^2)^{1/3}} \right)
        (1 + O(s_1^{-3/2})), \qquad s_1 \to +\infty,
\end{equation}
by virtue of the asymptotics of Airy function; cf. \cite{AS}.
Comparing this expressions with \eqref{eq:d in q}, we see that $q(s)
\sim \Ai(s)$ as $s\to+\infty$, and so $q$ is the Hastings-McLeod
solution.

To establish \eqref{c:Hamiltonian} and \eqref{ctil:Hamiltonian}, we
first derive an identity for the $s_1$-derivative of $c$ and $\til
c$, respectively. Noting that $u' = -q^2$, where $u$ is the
Hamiltonian in \eqref{def:Hamiltonian}, the first equation in
\eqref{eq:compat3a} can be written as
\begin{align}\label{eq:c derivative}
    \frac{\partial c}{\partial s_1} & = 2 \frac{r_2}{r_1} d^2 + 2 \frac{s_1}{r_1}
   \nonumber  \\
    & = 2 \frac{r_2}{r_1} \frac{(r_1 r_2)^{1/3} }{(r_1^2 + r_2^2)^{2/3}}
    q^2 \left( \frac{2 (r_2 s_1 + r_1 s_2)}{r_1^{1/3} r_2^{1/3} (r_1^2 + r_2^2)^{1/3}} \right)
    + 2 \frac{s_1}{r_1} \nonumber \\
    & = \frac{\partial}{\partial s_1} \left[
        - \frac{r_2^{2/3} }{r_1^{1/3} (r_1^2 + r_2^2)^{1/3}}
         u\left( \frac{2 (r_2 s_1 + r_1 s_2)}{r_1^{1/3} r_2^{1/3} (r_1^2 + r_2^2)^{1/3}} \right)
         + \frac{s_1^2}{r_1} \right].
\end{align}
Similarly, it follows from the second equation in
\eqref{eq:compat3a} that
\begin{align}\label{eq:tilc derivative}
    \frac{\partial \wtil c}{\partial s_1} & = 2 d^2
     = \frac{\partial}{\partial s_1} \left[
        - \frac{r_1^{2/3} }{r_2^{1/3} (r_1^2 + r_2^2)^{1/3}}
         u\left( \frac{2 (r_2 s_1 + r_1 s_2)}{r_1^{1/3} r_2^{1/3} (r_1^2 + r_2^2)^{1/3}} \right)
         \right].
\end{align}
By integrating \eqref{eq:c derivative} and \eqref{eq:tilc
derivative} with respect to $s_1$, respectively, we then obtain
\eqref{c:Hamiltonian} and \eqref{ctil:Hamiltonian} on account of
\eqref{eq:asy of c}--\eqref{eq:asy of til c} and the fact that
$u(s)=o(1)$ as $s \to +\infty$ if $q$ is the Hastings-McLeod
solution of the Painlev\'e II equation.

This completes the proof of Theorem \ref{theorem:Painleve2modelrhp}.
$\bol$


\part{Non-intersecting Brownian motions at the tacnode}
\label{part:critbeh:tacnode}

\textbf{Remark on notation and conventions:} Throughout the next sections we work under the
assumption of the double scaling limit as described in
Section~\ref{subsection:doublescaling}. Throughout most of Sections \ref{section:xilamfunctions}
and \ref{section:steepestdescent} $n$ will be large but fixed. Most notions depend
on $n$, such as $p_j$, $a_j$, $b_j$, $\alpha_j$, $\beta_j$, $V_j$, $\lambda_j$, and so on,
although this is not indicated in the notation.
Their limit values as $n\to\infty$ are denoted with a star, such as
$p_j^*,a_j^*,b_j^*$. Also recall that the temperature $T = 1$ and
that the time $t=t_{\crit}$ is fixed (independent of $n$) according to
\eqref{tcrit:bis}.

By the translational symmetry of the problem, it will be sufficient to give the
proofs in case where the tacnode is at the origin. That is, we will assume that
\[ x_{\crit} := \beta^*_2 =\alpha_1^*=0. \]
From
\eqref{def:alphaj}--\eqref{def:betaj}, this yields the relations
\begin{equation}\label{particlesatzero}
(1-t)a_1^*+tb_1^*=2\sqrt{p_1^* t(1-t)},\qquad -(1-t)a_2^*-tb_2^*=2\sqrt{p_2^*
t(1-t)},
\end{equation}
since $T = 1$. For ease of notation we also set
\[ \alpha^*:=-\alpha_2^*, \qquad \beta^*:=\beta_1^*. \]
Then the Brownian paths are asymptotically distributed on the two touching intervals $[-\alpha^*,0]$
and $[0,\beta^*]$ where (use \eqref{particlesatzero} to simplify
\eqref{def:alphaj}--\eqref{def:betaj})
\begin{align} \label{def:alphastar}
    \alpha^* :=  4\sqrt{p_2^* t(1-t)}>0, \\
    \label{def:betastar}
    \beta^* := 4\sqrt{p_1^* t(1-t)}>0.
\end{align}
We will use these conventions throughout
Sections~\ref{section:xilamfunctions}--\ref{section:proofsmaintheorems}.

\section{Modified equilibrium problem, $\xi$-functions and $\lambda$-functions}
\label{section:xilamfunctions}

\subsection{Modified equilibrium problem}

In the steepest descent analysis of the RH problem \ref{rhp:Y} for
$Y$, we will use functions that come from an equilibrium problem for
the two external fields
\begin{equation} \label{eq:V1V2}
    V_j(x) = \frac{1}{2t(1-t)}\left(x^2-2 ((1-t) a_j + t b_j) x \right), \qquad j=1,2.
\end{equation}

In the usual equilibrium problem with external field
\eqref{eq:V1V2}, one asks for a measure that minimizes the energy
functional
\[ \iint \log \frac{1}{|x-y|} d\mu(x) d\mu(y) + \int V_j(x) d\mu(x) \]
among all measures on the real line with mass $\int d\mu = p_j$.
Since $V_j$ is quadratic, the solution is a semicircle law with density
\begin{equation} \label{semicircles}
    \frac{1}{2\pi t(1-t)} \sqrt{(\beta_j - x)(x- \alpha_j)},
    \qquad x \in [\alpha_j, \beta_j],
    \end{equation}
with
\begin{align} \label{eq:alphaj2}
    \alpha_j & = (1-t)a_j + tb_j - 2 \sqrt{p_j t(1-t)}, \\
    \beta_j  & = (1-t) a_j + tb_j + 2 \sqrt{p_j t(1-t)};
    \label{eq:betaj2}
\end{align}
see also \eqref{def:alphaj}--\eqref{wignerdensity}.

The limiting situation corresponds to
$a_1^*,a_2^*,b_1^*,b_2^*$. Then $\alpha_2^* < \beta_2^* = 0 =
\alpha_1^* < \beta_1^*$ and the two semicircles \eqref{semicircles}
meet each other in the origin. For finite $n$ however, this may not
be the case. Indeed the two semicircles may be separated, or may
overlap depending on the situation. In what follows we modify the
equilibrium problems in such a way that the two minimizing measures
have supports $[0, \beta]$ and $[\alpha,0]$, respectively, that meet
at $0$. We allow the measures to become negative, so we will be
dealing with signed measures.
The modification of an equilibrium problem to prepare for a  later
steepest descent analysis of a RH problem was first done in \cite{CK}.

We assume that $\alpha_2 < 0 < \beta_1$, which is certainly the case if $n$ is
large enough.

\begin{definition}
In the modified equilibrium problem for $V_1$ we ask to minimize
\[ \iint \log \frac{1}{|x-y|} d\mu(x) d\mu(y) + \int V_1(x) d\mu(x) \]
among all signed measures $\mu$ on $[0,\infty)$ with total mass
$\int d\mu = p_1$ and such that $\mu$ is non-negative on $[\alpha_1,
\infty)$, where $\alpha_1$ is given by \eqref{eq:alphaj2}.

In the modified equilibrium problem for $V_2$ we ask to minimize
\[ \iint \log \frac{1}{|x-y|} d\mu(x) d\mu(y) + \int V_2(x) d\mu(x) \]
among all signed measures $\mu$ on $(-\infty,0]$ with total mass
$\int d\mu = p_2$ and such that $\mu$ is non-negative on
$(-\infty,\beta_2]$, where $\beta_2$ is given by \eqref{eq:betaj2}.
\end{definition}

We can explicitly find the minimizers for these modified equilibrium problems.
\begin{proposition}\label{prop:signedmeasure}
\begin{enumerate}
\item[\rm (a)] The minimizer in the modified equilibrium problem for $V_1$ is the
signed measure $\mu_1$ on $[0,\beta]$ with
\begin{equation} \label{def:beta}
    \beta = \frac{\alpha_1 + \beta_1 + 2 \sqrt{\alpha_1^2 + \beta_1^2 - \alpha_1 \beta_1}}{3}  > 0
    \end{equation}
given by
\begin{equation} \label{def:mu1}
    \frac{d\mu_1}{dx} =
    \frac{x-\delta_1}{2\pi t (1-t)} \sqrt{\frac{\beta-x}{x}}, \qquad x \in [0, \beta], \end{equation}
with
\begin{equation} \label{def:delta1}
    \delta_1 = \frac{\alpha_1 + \beta_1 - \sqrt{\alpha_1^2 + \beta_1^2 - \alpha_1
    \beta_1}}{3}.
    \end{equation}
\item[\rm (b)] The minimizer in the modified equilibrium problem for $V_2$ is the
signed measure $\mu_2$ on $[-\alpha,0]$ with
\begin{equation} \label{def:alpha}
    \alpha = \frac{-\alpha_2 - \beta_2 + 2 \sqrt{\alpha_2^2 + \beta_2^2 - \alpha_2 \beta_2}}{3} > 0
    \end{equation}
given by
\begin{equation} \label{def:mu2}
    \frac{d\mu_2}{dx} =
    \frac{\delta_2-x}{2\pi t(1-t)}  \sqrt{\frac{x+\alpha}{-x}}, \qquad x \in [-\alpha, 0],
    \end{equation}
with
\begin{equation} \label{def:delta2}
    \delta_2 = \frac{\alpha_2 + \beta_2 + \sqrt{\alpha_2^2 + \beta_2^2 - \alpha_2
    \beta_2}}{3}.
    \end{equation}
\end{enumerate}
\end{proposition}
\begin{proof}
Let $g_1$ be the $g$-function associated with $ \mu_1$, i.e.,
\begin{equation}
g_1(z)=\int \log(z-s) d\mu_1(s),\qquad z\in
\mathbb{C}\setminus(-\infty,\beta].
\end{equation}
Clearly, $g_1$ is defined and analytic in $\mathbb{C}\setminus(-\infty,\beta]$.
It then follows from the variational conditions of the equilibrium problem that
\begin{equation}\label{eq:g1++g1-}
        g_{1,+}(x)+g_{1,-}(x)=V_1(x)+ c_1, \qquad x\in(0,\beta),
\end{equation}
for some constant $c_1$ depending on $n$. Differentiating both
sides of \eqref{eq:g1++g1-} with respect to $x$ gives
\begin{equation}
g_{1,+}'(x)+g_{1,-}'(x)=V_1'(x),
\end{equation}
or equivalently,
\begin{equation}\label{eq:v-g=v-g}
\left(\frac{1}{2}V_1'(x)-g_{1}'(x)\right)_+=
-\left(\frac{1}{2}V_1'(x)-g_{1}'(x)\right)_-
\end{equation}
for $x\in (0,\beta)$. This, together with the fact that $g_{1,+}$
and $g_{1,-}$ only differ by a constant for $x\leq 0$, implies that
$\left(\frac{1}{2}V_1'(z)-g_{1}'(z)\right)^2$ is an entire function
in the complex plane. Since this function only has polynomial growth
for $z$ large and $\mu_1$ is supported on $[0, \beta]$, it is
readily seen that
\begin{equation}\label{eq:v1/2-g1}
z\left(\frac{1}{2}V_1'(z)-g_{1}'(z)\right)^2
=\frac{(z-\delta_1)^2(z-\beta)}{4t^2(1-t)^2}
\end{equation}
for some $\delta_1<\beta$. Hence,
\begin{align}
\frac{d\mu_1}{dx} &= \frac{1}{\pi} \Im
\left(\frac{1}{2}V_1'(x)-g_{1}'(x)\right)_+ \\
&= \frac{x-\delta_1}{2\pi t (1-t)} \sqrt{\frac{\beta-x}{x}}, \qquad x \in [0,
\beta],
\end{align}
which is \eqref{def:mu1}.

To obtain the representations of $\beta$ and $\delta_1$, we expand
\eqref{eq:v1/2-g1} as $z\to\infty$. Comparing the coefficients of order
$O(z^2)$ and $O(z)$ on both sides and taking into account
\eqref{eq:alphaj2}--\eqref{eq:betaj2} leads to
\begin{equation}
\begin{aligned}
\beta+2\delta_1&=\alpha_1+\beta_1,
\\
\delta_1(2\beta+\delta_1)&=\alpha_1 \beta_1.
\end{aligned}
\end{equation}
Solving the above algebraic equations for $\delta_1$ and $\beta$
gives us \eqref{def:beta} and \eqref{def:delta1}.

The explicit formula for $\mu_2$ stated in item $(b)$ can be proved
in a similar manner as for $\mu_1$. To that end, we need to use
\begin{equation}\label{eq:v2-g}
z\left( \frac{1}{2} V_2'(z) - g_{2}'(z) \right)^2  =
\frac{(z-\delta_2)^2 (z+\alpha)}{4 t^2 (1-t)^2},
\end{equation}
where
\begin{equation}
g_2(z)=\int \log(z-s) d\mu_2(s),\qquad z\in
\mathbb{C}\setminus(-\infty,0],
\end{equation}
and $\delta_2 > -\alpha$. We omit the details here.

This completes the proof of the proposition.
\end{proof}

\begin{remark} We can check from formula \eqref{def:delta1} that $\delta_1$
has the same sign as $\alpha_1$. For example, if $\alpha_1 > 0$ then
\[ 0 < \tfrac{1}{3} \alpha_1 <  \delta_1 < \tfrac{1}{2} \alpha_1, \]
so in particular $\delta_1 > 0$. Then the density of $\mu_1$ is negative on the
interval $(0, \delta_1)$. Similarly, if $\alpha_1 < 0$ then \[ \tfrac{2}{3}
\alpha_1 < \delta_1 < \tfrac{1}{2} \alpha_1 <0, \]  so in particular $\delta_1
< 0$ and $\mu_1$ is positive on $[0,\beta]$.
\end{remark}

We recall that the constants $\alpha$, $\beta$, $\delta_{1}$ and $\delta_{2}$
in the above proposition all depend on $n$. From
\eqref{doublescaling:a}--\eqref{doublescaling:b},
\eqref{eq:alphaj2}--\eqref{eq:betaj2} and \eqref{particlesatzero}, it follows
that
\begin{align}
&\alpha=\alpha^*+O(n^{-2/3}), \qquad \beta=\beta^*+O(n^{-2/3}),
\label{limit:al and beta}\\
&\delta_1 = \frac{L_1(1-t)+L_3t}{2}
n^{-2/3}+O(n^{-4/3}), \label{critlim:del1} \\
&\delta_2 = \frac{L_2(1-t)+L_4t}{2} n^{-2/3}+O(n^{-4/3}),
\label{critlim:del2}
\end{align}
as $n\to \infty$, where $\alpha^*$ and $\beta^*$ are given by
\eqref{def:alphastar} and \eqref{def:betastar}, respectively.

\subsection{The $\xi$-functions}

Let $F_1$ and $F_2$ be the Cauchy transforms of the minimizers
$\mu_1$ and $\mu_2$, i.e.,
\[ F_j(z) = \int \frac{1}{z-x} d\mu_j(x), \quad
z\in\mathbb{C}\setminus \supp(\mu_j), \quad  j = 1,2.\] Clearly,
$g_{j}'(z)=F_j(z)$. In view of \eqref{eq:v1/2-g1} and \eqref{eq:v2-g}, we have
the following useful identities:
\begin{align*}
    \left( \frac{1}{2} V_1'(z) - F_1(z) \right)^2 &
= \frac{(z-\delta_1)^2 (z-\beta)}{4 t^2(1-t)^2 z}, \\
    \left( \frac{1}{2} V_2'(z) - F_2(z) \right)^2 & =
\frac{(z-\delta_2)^2 (z+\alpha)}{4 t^2 (1-t)^2 z},
    \end{align*}
which are valid for every $z \in \mathbb C$.

The $\xi$-functions are defined as follows:
\begin{definition}
We define
\begin{equation} \label{def:xi12}
\begin{aligned}
    \xi_1(z) & =  \frac{1}{2} V_1'(z) - F_1(z) = \frac{z-\delta_1}{2t(1-t)} \left(\frac{z-\beta}{z} \right)^{1/2}, \\
    \xi_2(z) & = \frac{1}{2} V_2'(z) - F_2(z) = \frac{z-\delta_2}{2t(1-t)} \left(\frac{z+\alpha}{z} \right)^{1/2},
\end{aligned}
\end{equation}
where the branch of the square root is taken which is positive for positive $z > \beta$. Then
$\xi_1$ is defined and analytic in $\mathbb C \setminus [0,\beta]$, $\xi_2$ is defined and analytic
in $[-\alpha,0]$.
\end{definition}

The $\xi$-functions have the asymptotic behavior
\begin{align} \label{xi12:asy}
    \xi_k(z) &=  \frac{1}{2}V_k'(z) -\frac{p_k}{z}+O(z^{-2}), \qquad k
    =1,2,
\end{align}
as $z \to \infty$.

\subsection{The $\lambda$-functions}\label{sec:lamda function}

The functions $\lambda_k$ are defined as the following
anti-derivatives of the $\xi_k$-functions:

\begin{definition}
We define
\begin{equation} \label{def:lambda12}
    \lambda_k(z) =  \int_{0}^z \xi_k(s)\ ds, \qquad k =1,2,
\end{equation}
where the contour of integration does not intersect $(0,\infty)$ if
$k=1$ and $(-\infty,0)$ if $k=2$. Then $\lambda_1$ is defined and
analytic in $\mathbb C\setminus [0,\infty)$, $\lambda_2$ is defined
and analytic in $\mathbb C\setminus (-\infty,0]$.
\end{definition}

From \eqref{def:xi12} and \eqref{def:lambda12}, we have the
following explicit expressions for the $\lambda$-functions:
\begin{equation}\label{lambda1}
\begin{aligned}
        \lambda_1(z) & =
        \frac{2z- \beta - 4\delta_1}{8t(1-t)}  (z^2 - \beta z)^{1/2} \\
    & \qquad - p_1 \left(\log \left(z - \tfrac{\beta}{2}
    + (z^2 - \beta z)^{1/2} \right) - \log\left(-\tfrac{\beta}{2}\right)
    \right),
\end{aligned}
\end{equation}
where the logarithm is defined with a branch cut along the positive real axis, so that for example
$\log(-\tfrac{\beta}{2}) = \log(\tfrac{\beta}{2}) + \pi i$,
and
\begin{equation} \label{lambda2}
\begin{aligned}
        \lambda_2(z) & =
        \frac{2z+ \alpha - 4 \delta_2}{8t(1-t)}   (z^2 + \alpha z)^{1/2} \\
        & \qquad - p_2 \left(\log \left(z + \tfrac{\alpha}{2}
        + (z^2 + \alpha z)^{1/2} \right) - \log\left(\tfrac{\alpha}{2}\right)
        \right),
\end{aligned}
\end{equation}
where now the logarithm is defined with a branch cut along the negative real axis.

Integrating \eqref{xi12:asy} (or from \eqref{lambda1}--\eqref{lambda2}), we get
the following asymptotic behavior:
\begin{align}\label{lambda12:asy}
    \lambda_k(z)&= \frac{1}{2} V_k(z) -p_k\log z+\ell_k+O(z^{-1}), \qquad k = 1,2,
\end{align}
as $z \to \infty$, for certain constants $\ell_1,\ell_2$ that can be computed from \eqref{lambda1}--\eqref{lambda2}.

In what follows, we will also need the following inequalities for the
$\lambda$-functions. Here we write $\lambda_{1,\pm}(x)$ for $x>0$ to denote the
boundary values of $\lam_1$ obtained from the upper or lower half plane
respectively, and similarly we define $\lam_{2,\pm}(x)$ for $x<0$.

\begin{lemma} \label{lemma:lambda:inequalities}
We have
\begin{equation} \label{var:ineq1}
    \Re \lambda_{1,+}(x)  =\Re \lambda_{1,-}(x)
        \begin{cases} =0, & x\in [0,\beta] \\
        > 0, & x\in(\beta,\infty), \end{cases}
        \end{equation}
\begin{equation} \label{var:ineq2}
    \Re \lambda_{2,+}(x)=\Re \lambda_{2,-}(x)
    \begin{cases} =0, & x\in [-\alpha,0] \\
        > 0, & x\in(-\infty,-\alpha), \end{cases}
    \end{equation}
\begin{equation} \label{var:ineq3}
    \Im \lambda_{1,+}(x)=-\Im \lambda_{1,-}(x)= \pi \mu_1((0,x]),\quad x\geq 0,
    \end{equation}
\begin{equation} \label{var:ineq4}
    \Im \lambda_{2,+}(x)=-\Im \lambda_{2,-}(x)= \pi \mu_2([x,0)),\quad x\leq 0.
    \end{equation}
\end{lemma}

\begin{proof}
Using \eqref{def:mu1}, \eqref{def:mu2} and \eqref{def:xi12}, we find
\begin{align*}
 \xi_{1,\pm}(x) & = \pm \pi i \frac{d\mu_1}{dx}, \qquad 0 < x < \beta, \\
 \xi_{2,\pm}(x) & = \pm \pi i \frac{d\mu_2}{dx}, \qquad -\alpha < x < 0,
\end{align*}
which after integration, see \eqref{def:lambda12}, leads to the equalities in
\eqref{var:ineq1}--\eqref{var:ineq4}.
From \eqref{def:xi12} it is also clear that $\xi_1(x) > 0$ for $x > \beta$ and
$\xi_2(x) < 0$ for $x < -\alpha$. This leads after integration to the inequalities
in \eqref{var:ineq1}--\eqref{var:ineq2}.
\end{proof}

Note that as a consequence of \eqref{var:ineq1} and \eqref{var:ineq3} we have
\[ \lambda_{1,+}(x)-\lambda_{1,-}(x) = 2\pi i \mu_1((0,x]) = 2\pi i p_1,\qquad x\geq \beta. \]
Since $np_1 = n_1$ is an integer, see \eqref{defp1p2}, we then obtain
\begin{equation} \label{var:ineq5}
    e^{n (\lam_{1,+}(x)-\lam_{1,-}(x))}= 1, \qquad x \geq \beta.
\end{equation}
Similarly we have
\begin{equation} \label{var:ineq6}
    e^{n (\lam_{2,+}(x)-\lam_{2,-}(x))} = 1, \qquad x\leq -\alpha.
\end{equation}

\section{Steepest descent analysis for $Y(z)$}
\label{section:steepestdescent}

In this section we perform the steepest descent analysis of the RH
problem~\ref{rhp:Y} for $Y$. To this
end we apply a series of explicit and invertible transformations
\[ Y \mapsto X \mapsto U \mapsto T \mapsto S \mapsto R \]
of the RH problem. In Section~\ref{section:proofsmaintheorems} we will use
these transformations to prove Theorems~\ref{theorem:kernelpsi} and
\ref{theorem:Painleve2rec}.

\subsection{First transformation: Gaussian elimination}
\label{subsection:gausselim}

In the first transformation we apply Gaussian elimination to the
jump matrices of the RH problem for $Y(z)$. This kind of operation
was also done in \cite{DelKui1}. Note that in \cite{DelKui1} there
were actually two types of Gaussian elimination, since we needed to
make a case distinction between $0<t<t_{\crit}$ and $t_{\crit}<t<1$.
In the present case, however, we are looking precisely at the
critical time $t=t_{\crit}$, and therefore we are able to apply both
kinds of Gaussian elimination simultaneously.

We introduce a  curve $\Gamma_r$ in the right-half plane passing through $0$
at angles $\pm \varphi_2$ with
\[ 0 < \varphi_2 < \pi/3 \]
and extending to infinity and a similar curve $\Gamma_l$ in the left-half plane
with orientation as shown in Figure \ref{fig:jumpsX}. The domains enclosed by
$\Gamma_l$ around the negative real axis, and by $\Gamma_r$ around the positive
real axis are called the global lenses \cite{ABK}.

Now we start from the solution $Y$ of the RH problem \ref{rhp:Y} and
define a new matrix-valued function $X=X(z)$ as follows.
\begin{definition}
We define
\begin{align}
\label{defX1} X& = Y\begin{pmatrix}1&0&0&0\\ \frac{w_{1,2}}{w_{1,1}}&1&0&0  \\
    0&0&1&-\frac{w_{2,2}}{w_{2,1}}\\
    0&0&0&1\end{pmatrix},\quad && \begin{array}{l} \text{in the global lens} \\
    \text{around the positive real axis}, \end{array} \\
\label{defX2} X&= Y\begin{pmatrix}1&\frac{w_{1,1}}{w_{1,2}}&0&0\\ 0&1&0&0  \\
    0&0&1&0\\ 0&0&-\frac{w_{2,1}}{w_{2,2}}&1
    \end{pmatrix},\quad && \begin{array}{l} \text{in the global lens} \\
    \text{around the negative real axis}, \end{array} \\
\label{defX3} X&= Y, && \textrm{elsewhere,}
\end{align}
where we recall the weights $w_{j,k}$, $j,k=1,2$ from
\eqref{gaussianw1}--\eqref{gaussianw2}.
\end{definition}

The matrix-valued function $X$ satisfies a new RH problem, with
jumps on the contour $\mathbb R \cup\Gamma_l\cup\Gamma_r$ with the
jump matrices as shown in Figure \ref{fig:jumpsX}.

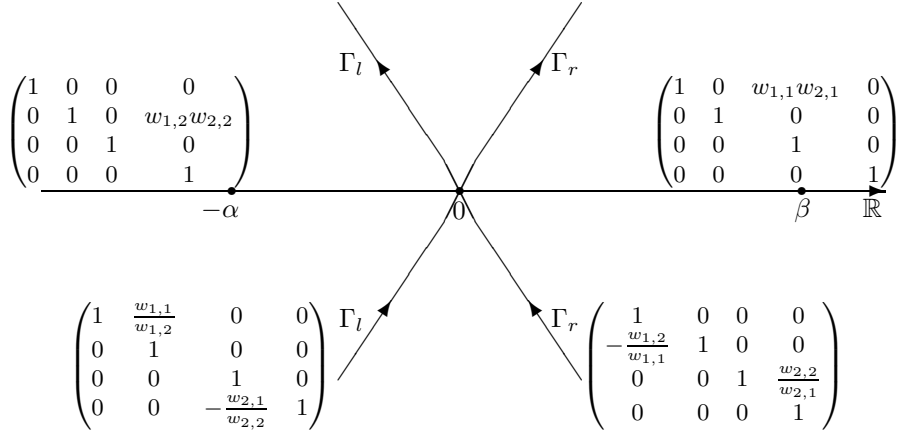
\begin{figure}[t]
\begin{center}
   \setlength{\unitlength}{1truemm}
   \begin{picture}(100,70)(-5,2)
       \put(40,40){\line(-1,-2){2}}
       \put(40,40){\line(-1,2){2}}
       \put(38,36){\line(-2,-3){14}}
       \put(38,44){\line(-2,3){14}}
       \put(24.1,56){$\Gamma_l$}
       \put(24.1,22){$\Gamma_l$}
       \put(40,40){\line(1,-2){2}}
       \put(40,40){\line(1,2){2}}
       \put(42,36){\line(2,-3){14}}
       \put(42,44){\line(2,3){14}}
       \put(52,56){$\Gamma_r$}
       \put(52,22){$\Gamma_r$}
       \put(50,56){\thicklines\vector(2,3){1}}
       \put(30,56){\thicklines\vector(-2,3){1}}
       \put(50,24){\thicklines\vector(-2,3){1}}
       \put(30,24){\thicklines\vector(2,3){1}}

       \put(-15,40){\line(1,0){110}}
       \put(93,36.6){$\mathbb R $}
       \put(95,40){\thicklines\vector(1,0){1}}
       \put(6,36.6){$-\alpha$}
       \put(39,36.3){$0$}
       \put(84,36.6){$\beta$}
       \put(10,40){\thicklines\circle*{1}}
       \put(40,40){\thicklines\circle*{1}}
       \put(85,40){\thicklines\circle*{1}}

       \put(56,16){$\small{\begin{pmatrix}1&0&0&0\\ -\frac{w_{1,2}}{w_{1,1}}&1&0&0  \\
       0&0&1&\frac{w_{2,2}}{w_{2,1}}\\ 0&0&0&1 \end{pmatrix}}$}
       \put(-11.5,16){$\small{\begin{pmatrix}1&\frac{w_{1,1}}{w_{1,2}}&0&0\\ 0&1&0&0  \\
       0&0&1&0\\ 0&0&-\frac{w_{2,1}}{w_{2,2}}&1 \end{pmatrix}}$}
       \put(-20,47){$\small{\begin{pmatrix}1&0&0&0\\ 0&1&0&w_{1,2}w_{2,2}  \\
       0&0&1&0\\ 0&0&0&1 \end{pmatrix}}$}
       \put(65,47){$\small{\begin{pmatrix}1&0&w_{1,1}w_{2,1}&0\\ 0&1&0&0  \\
       0&0&1&0\\ 0&0&0&1 \end{pmatrix}}$}

   \end{picture}
   \vspace{-4mm}
   \caption{The curves $\Gamma_l$ and $\Gamma_r$, and the jump matrices $J_X$
   on $\mathbb R \cup \Gamma_r \cup \Gamma_l$ in the RH problem \ref{rhpforX} for $X$.}
   \label{fig:jumpsX}
\end{center}
\end{figure}

Thus $X$ satisfies the following RH problem.
\begin{rhp} \label{rhpforX}
\textrm{ }\begin{enumerate}
\item[\rm (1)] $X$ is analytic in
$\mathbb C \setminus (\mathbb R  \cup \Gamma_l \cup \Gamma_r)$.
\item[\rm (2)] On $\mathbb R  \cup \Gamma_l \cup \Gamma_r$,
we have that $X_+ = X_- J_X$ with jump matrices $J_X$ as shown  in
Figure \ref{fig:jumpsX}.
\item[\rm (3)] As $z\to\infty$, we have that
\begin{equation}\label{asymptoticconditionX} X(z) =
    (I+O(1/z))\diag(z^{n_1},z^{n_2},z^{-n_1},z^{-n_2}).
\end{equation}
\end{enumerate}
\end{rhp}

Note that we have taken the temperature $T = 1$, so that by \eqref{gaussianw1},
\eqref{gaussianw2} and \eqref{eq:V1V2}
\begin{equation}\label{jump for X on real line}
        w_{1,1} w_{2,1} = e^{-nV_1}, \qquad
        w_{1,2} w_{2,2} = e^{-n V_2},
\end{equation}
which are the non-trivial entries in the jump matrices on the real
line, and
\begin{align}   \label{w11 over w12}
    \frac{w_{1,1}(x)}{w_{1,2}(x)} & = e^{\frac{n}{t} (a_1-a_2) x} =
    e^{-\frac{n}{2} (V_1(x) - V_2(x))} e^{n \kappa x}, \\
    \label{w22 over w21}
    \frac{w_{2,2}(x)}{w_{2,1}(x)} & = e^{- \frac{n}{1-t} (b_1-b_2) x} =
    e^{ \frac{n}{2} (V_1(x) - V_2(x))} e^{n \kappa x},
\end{align}
with a constant
\begin{equation} \label{def:kappa}
    \kappa = \frac{1}{2} \left( \frac{a_1-a_2}{t} - \frac{b_1-b_2}{1-t}\right),
\end{equation}
which appear in the jump matrices on $\Gamma_r$ and $\Gamma_l$.

The constant $\kappa$ depends on $n$, and tends to $0$ as $n \to
\infty$. Indeed we have
\begin{equation}\label{kappa:asy}
    \kappa = \frac{1}{2} \left( \frac{L_1 - L_2}{t} - \frac{L_3 -L_4}{1-t} \right) n^{-2/3},
\end{equation}
on account of \eqref{tcrit:bis} and
\eqref{doublescaling:a}--\eqref{doublescaling:b}.

\subsection{Second transformation: Normalization at infinity}
\label{subsection:normalization}

The next transformation is to normalize the RH problem at infinity.
To this end we use the $\lambda_k$-functions defined in
\eqref{def:lambda12}.

\begin{definition}
We define a new $4\times 4$ matrix-valued function $U = U(z)$ by
\begin{equation}\label{defU}
U=  L^{-n}X \Lambda^n ,
\end{equation}
where $\Lambda = \Lambda(z)$ is given by
\begin{multline} \label{def:Lam}
    \Lambda =
    \diag\left(\exp\left(\lambda_1- \tfrac{1}{2} V_1 \right),
    \exp\left(\lambda_2 - \tfrac{1}{2} V_2 \right),\right. \\
    \left. \exp\left(-\lambda_1 + \tfrac{1}{2} V_1 \right),
    \exp\left(-\lambda_2 + \tfrac{1}{2} V_2 \right) \right),
\end{multline}
and
\begin{equation}
\label{defLmx} L = \diag\left(e^{\ell_1},e^{\ell_2},
e^{-\ell_1},e^{-\ell_2}\right),
\end{equation}
with $\ell_1$ and $\ell_2$ as in \eqref{lambda12:asy}.
\end{definition}

Then $U$ satisfies the following RH problem.
\begin{rhp}  \label{rhpforU}
\textrm{ }
\begin{enumerate}
    \item[\rm (1)] $U$ is analytic in $\mathbb C \setminus ( \mathbb R  \cup \Gamma_l \cup \Gamma_r)$.
  \item[\rm (2)] On $\mathbb R  \cup \Gamma_l \cup \Gamma_r$,
  we have that $U_+ = U_- J_U$ with jump matrices $J_U$ as shown in
  Figure~\ref{fig:jumpsU}.
  \item[\rm (3)] As $z\to\infty$, we have that
\begin{equation}\label{asymptoticconditionU} U(z) =
    I+O(1/z).
\end{equation}
\end{enumerate}
\end{rhp}

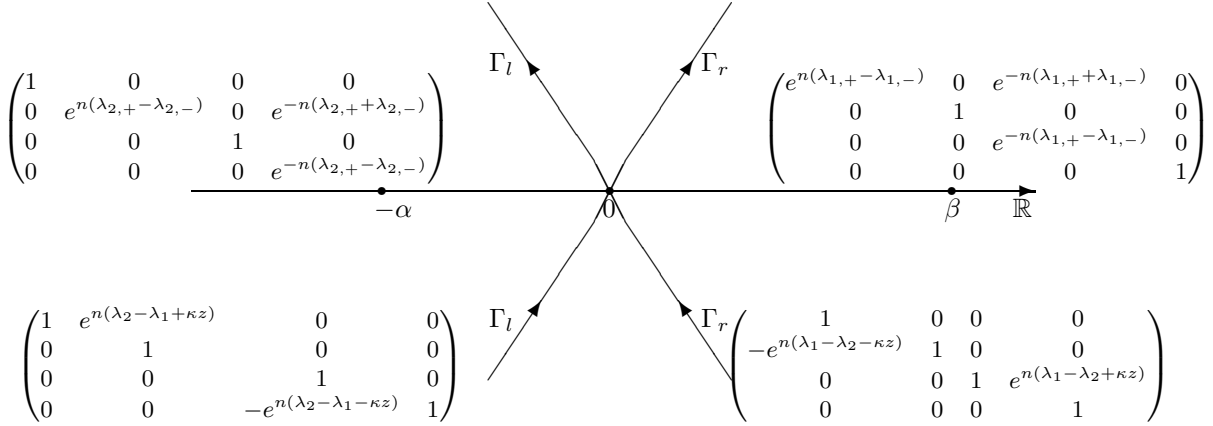
\begin{figure}[t]
\begin{center}
   \setlength{\unitlength}{1truemm}
   \begin{picture}(100,70)(-5,2)
          \put(40,40){\line(-1,-2){2}}
       \put(40,40){\line(-1,2){2}}
       \put(38,36){\line(-2,-3){14}}
       \put(38,44){\line(-2,3){14}}
       \put(24.1,56){$\Gamma_l$}
       \put(24.1,22){$\Gamma_l$}
       \put(40,40){\line(1,-2){2}}
       \put(40,40){\line(1,2){2}}
       \put(42,36){\line(2,-3){14}}
       \put(42,44){\line(2,3){14}}
       \put(52,56){$\Gamma_r$}
       \put(52,22){$\Gamma_r$}
       \put(50,56){\thicklines\vector(2,3){1}}
       \put(30,56){\thicklines\vector(-2,3){1}}
       \put(50,24){\thicklines\vector(-2,3){1}}
       \put(30,24){\thicklines\vector(2,3){1}}

       \put(-15,40){\line(1,0){110}}
       \put(93,36.6){$\mathbb R $}
       \put(95,40){\thicklines\vector(1,0){1}}
       \put(9,36.6){$-\alpha$}
       \put(39,36.6){$0$}
       \put(84,36.6){$\beta$}
       \put(10,40){\thicklines\circle*{1}}
       \put(40,40){\thicklines\circle*{1}}
       \put(85,40){\thicklines\circle*{1}}

       \put(55,16){$\small{\begin{pmatrix}1&0&0&0\\ -e^{n(\lambda_1-\lambda_2 - \kappa z)} &1&0&0  \\
       0&0&1&e^{n(\lambda_1-\lambda_2 + \kappa z)} \\ 0&0&0&1
       \end{pmatrix}}$}
       \put(-38,16){$\small{\begin{pmatrix}1&e^{n(\lambda_2-\lambda_1 + \kappa z)} &0&0\\ 0&1&0&0  \\
       0&0&1&0\\ 0&0&-e^{n(\lambda_2-\lambda_1-\kappa z)} &1
       \end{pmatrix}}$}
       \put(-40,47.5){$\small{\begin{pmatrix}1&0&0&0\\ 0&e^{n(\lambda_{2,+}-\lambda_{2,-})}&0&e^{-n(\lambda_{2,+}+\lambda_{2,-})}  \\
       0&0&1&0\\ 0&0&0&e^{-n(\lambda_{2,+}-\lambda_{2,-})}
       \end{pmatrix}}$}
       \put(60,47.5){$\small{\begin{pmatrix}e^{n(\lambda_{1,+}-\lambda_{1,-})}&0&e^{-n(\lambda_{1,+}+\lambda_{1,-})}&0\\ 0&1&0&0  \\
       0&0&e^{-n(\lambda_{1,+}-\lambda_{1,-})}&0\\ 0&0&0&1
       \end{pmatrix}}$}

   \end{picture}
   \vspace{-6mm}
   \caption{The jump matrices in the RH problem \ref{rhpforU} for $U$.}
   \label{fig:jumpsU}
\end{center}
\end{figure}

The asymptotic condition of $U$ in \eqref{asymptoticconditionU}
follows from \eqref{lambda12:asy} and the jump matrices $J_U$ in
Figure~\ref{fig:jumpsU} follow from \eqref{jump for X on real
line}--\eqref{def:kappa} and straightforward calculations.

By \eqref{var:ineq1}--\eqref{var:ineq6} the jump matrix $J_U$ takes the
following form on $(-\alpha, \beta)$,
\begin{align} \label{JU2}
    J_U & = \begin{pmatrix}
        1 & 0 & 0 & 0 \\ 0 & e^{2n \lambda_{2,+}} & 0 & 1 \\ 0 & 0 & 1 & 0 \\ 0 & 0 & 0 & e^{2n \lambda_{2,-}}
        \end{pmatrix} \qquad \text{ on } (-\alpha, 0) \\
    J_U & = \begin{pmatrix} \label{JU1}
         e^{2n \lambda_{1,+}} & 0 & 1 & 0 \\ 0 & 1 & 0 & 0 \\  0 & 0 & e^{2n \lambda_{1,-}} & 0 \\ 0 & 0 & 0 & 1
        \end{pmatrix} \qquad \text{ on } (0, \beta)
        \end{align}
and the following form on the rest of the real line
\begin{align*}
    J_U & =  I + e^{-2n \Re \lambda_{2,+}} E_{2,4}  \qquad \text{ on } (-\infty,-\alpha), \\
    J_U & =  I + e^{-2n \Re \lambda_{1,+}} E_{1,3}  \qquad \text{ on } (\beta, \infty),
    \end{align*}
which tend to the identity matrix as $n \to \infty$ at an exponential rate
because of the strict inequalities in \eqref{var:ineq1} and \eqref{var:ineq2}.
Recall that $E_{i,j}$ denotes the elementary matrix with $1$ at position
$(i,j)$ and zero elsewhere.

\subsection{Third transformation: Opening of local lenses}
\label{subsection:locallens}

In the next transformation $U\mapsto T$ we open a local lens around each of the
intervals $[-\alpha,0]$ and $[0,\beta]$.

The local lenses around $[0, \beta]$ and $[-\alpha,0]$ are illustrated in
Figure \ref{fig:jumpsT1}. The lips of the lens around $[0,\beta]$ are denoted
by $\Sigma_1$, and the lips of the lens around $[-\alpha,0]$ are denoted by
$\Sigma_2$. We make sure that $\Sigma_1$ and $\Sigma_2$ do not intersect with
$\Gamma_r$ and $\Gamma_l$, except at the origin.

The transformation $U\mapsto T$ is based on the standard factorization (we use
$\lambda_{j,+} + \lambda_{j,-} = 0$)
\begin{align*}
        \begin{pmatrix}  e^{2n \lambda_{j,+}} &  1 \\ 0 &  e^{2n \lambda_{j,-}}
        \end{pmatrix} =
        \begin{pmatrix} 1 & 0 \\ e^{2n \lambda_{j,-}} & 1 \end{pmatrix}
        \begin{pmatrix} 0 & 1 \\ -1 & 0 \end{pmatrix}
        \begin{pmatrix} 1 & 0 \\ e^{2n \lambda_{j,+}} & 1 \end{pmatrix}
        \end{align*}
that can be applied to the non-trivial $2\times 2$ blocks in \eqref{JU2} and \eqref{JU1}.

\begin{definition} \label{def:T}
We define
\begin{align} \label{defT1}
    T(z) & = \begin{cases}
    U(z) (I-e^{2n\lambda_1}E_{3,1}) &
    \textrm{ in upper lens region around $[0,\beta]$,}\\
    U(z)(I+e^{2n\lambda_1}E_{3,1}) &
    \textrm{ in lower lens region around $[0,\beta]$,} \\
    \end{cases} \\
    \label{defT2}
    T(z) & = \begin{cases}
    U(z) (I-e^{2n\lambda_2}E_{4,2}) &
    \textrm{ in upper lens region around $[-\alpha,0]$,}\\
    U(z) (I+e^{2n\lambda_2}E_{4,2}) &
    \textrm{ in lower lens region around $[-\alpha,0]$.}
    \end{cases} \\
        \label{defT3}
        T(z) & = U(z),\qquad \qquad \qquad \qquad \textrm{ outside the lenses.}
\end{align}
\end{definition}

\begin{figure}[t]
\begin{center}
   \setlength{\unitlength}{1truemm}
   \begin{picture}(100,70)(-5,2)
       \put(40,40){\line(-1,-2){2}}
       \put(40,40){\line(-1,2){2}}
       \put(38,36){\line(-2,-3){14}}
       \put(38,44){\line(-2,3){14}}
       \put(24.1,56){$\Gamma_l$}
       \put(24.1,22){$\Gamma_l$}
       \put(40,40){\line(1,-2){2}}
       \put(40,40){\line(1,2){2}}
       \put(42,36){\line(2,-3){14}}
       \put(42,44){\line(2,3){14}}
       \put(52,56){$\Gamma_r$}
       \put(52,22){$\Gamma_r$}
       \put(50,56){\thicklines\vector(2,3){1}}
       \put(30,56){\thicklines\vector(-2,3){1}}
       \put(50,24){\thicklines\vector(-2,3){1}}
       \put(30,24){\thicklines\vector(2,3){1}}

       \put(-15,40){\line(1,0){110}}
       \put(93,36.6){$\mathbb R $}
       \put(95,40){\thicklines\vector(1,0){1}}
       \put(9,36.6){$-\alpha$}
       \put(39,36.6){$0$}
       \put(84,36.6){$\beta$}
       \put(10,40){\thicklines\circle*{1}}
       \put(40,40){\thicklines\circle*{1}}
       \put(85,40){\thicklines\circle*{1}}

       \qbezier(10,40)(25,50)(40,40)
       \qbezier(10,40)(25,30)(40,40)
       \put(25,45){\thicklines\vector(1,0){1}}
       \put(25,35){\thicklines\vector(1,0){1}}
       \qbezier(40,40)(62.5,50)(85,40)
       \qbezier(40,40)(62.5,30)(85,40)
       \put(62.5,45){\thicklines\vector(1,0){1}}
       \put(62.5,35){\thicklines\vector(1,0){1}}
       \put(25,47){$\Sigma_2$}
       \put(25,30){$\Sigma_2$}
       \put(62.5,47){$\Sigma_1$}
       \put(62.5,30){$\Sigma_1$}

   \end{picture}
   \vspace{-14mm}
   \caption{The jump contour $\Sigma_T$ in the RH problem \ref{rhpforT} for  $T$.}
   \label{fig:jumpsT1}
\end{center}
\end{figure}
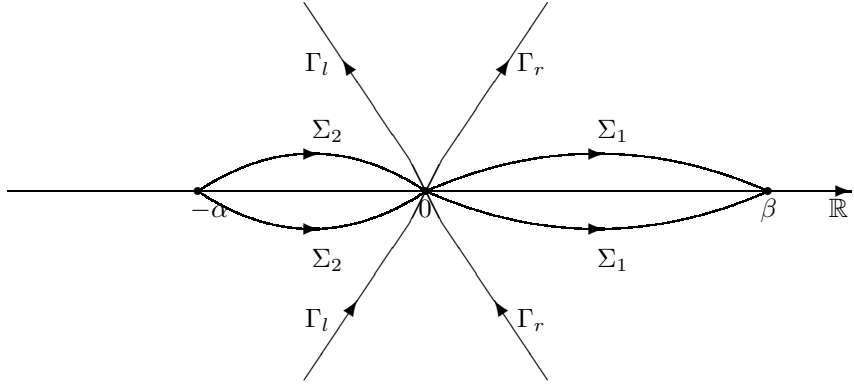

Then the matrix-valued function $T$ is defined and analytic in
$\mathbb C \setminus \Sigma_T$, where
\[ \Sigma_T = \mathbb R \cup \Gamma_l \cup \Gamma_r \cup \Sigma_1 \cup \Sigma_2 \]
and it satisfies the following RH problem.

\begin{rhp}  \label{rhpforT} \textrm{ }
\begin{enumerate}
\item[\rm (1)] $T$ is analytic
in $\mathbb C \setminus \Sigma_T$.
\item[\rm (2)] $T$ satisfies the jump relation $T_+ = T_- J_T$ on $\Sigma_T$ with jump matrices:
\begin{align*}
 J_T &= \begin{pmatrix} 0 & 0 & 1 & 0 \\ 0 & 1 & 0 & 0 \\
    -1 & 0 & 0 & 0 \\ 0 & 0 & 0 & 1 \end{pmatrix}, \quad \textrm{on } (0,\beta),
 \\
    J_T &= \begin{pmatrix} 1 & 0 & 0 & 0 \\ 0 & 0 & 0 & 1 \\
0 & 0 & 1 & 0 \\ 0 & -1 & 0 & 0 \end{pmatrix}, \quad \text{on } (-\alpha,0),
\\
    J_T &= I+e^{2n\lam_1}E_{3,1}, 
    \qquad \text{on }  \Sigma_1,
\\
 J_T & = I+e^{2n\lam_2}E_{4,2},
\qquad \text{on } \Sigma_2,
\end{align*}
and
\[ J_T = J_U, \qquad \text{on } \Gamma_l \cup \Gamma_r \cup (-\infty,-\alpha) \cup (\beta,\infty). \]
 \item[\rm (3)] As $z\to\infty$, we have that
\begin{equation*}T(z) =
    I+O(1/z).
\end{equation*}
\end{enumerate}
\end{rhp}
See Figure~\ref{fig:jumpsT1} for the jump contour $\Sigma_T$  and
Figure~\ref{fig:jumpsT2} for all jumps $J_T$ in a neighborhood of
the origin.

\begin{figure}[t]
\vspace{14mm}
\begin{center}
   \setlength{\unitlength}{1truemm}
   \begin{picture}(100,70)(-5,2)
       \put(40,40){\line(1,0){40}}
       \put(40,40){\line(-1,0){40}}
       \put(40,40){\line(2,1){30}}
       \put(40,40){\line(2,-1){30}}
       \put(40,40){\line(-2,1){30}}
       \put(40,40){\line(-2,-1){30}}
       \put(40,40){\line(2,3){15}}
       \put(40,40){\line(2,-3){15}}
       \put(40,40){\line(-2,3){15}}
       \put(40,40){\line(-2,-3){15}}
       \put(40,40){\thicklines\circle*{1}}
       \put(39.3,36){$0$}
       \put(60,40){\thicklines\vector(1,0){.0001}}
       \put(20,40){\thicklines\vector(1,0){.0001}}
       \put(60,50){\thicklines\vector(2,1){.0001}}
       \put(60,30){\thicklines\vector(2,-1){.0001}}
       \put(20,50){\thicklines\vector(2,-1){.0001}}
       \put(20,30){\thicklines\vector(2,1){.0001}}
       \put(50,55){\thicklines\vector(2,3){.0001}}
       \put(50,25){\thicklines\vector(-2,3){.0001}}
       \put(30,55){\thicklines\vector(-2,3){.0001}}
       \put(30,25){\thicklines\vector(2,3){.0001}}

       \put(60,40.5){}
       \put(60,52.5){$\Sigma_1$}
       \put(49,60){$\Gamma_r$}
       \put(27,60){$\Gamma_l$}
       \put(20.5,50){$\Sigma_2$}
       \put(20,40.5){}
       \put(20,32){$\Sigma_2$}
       \put(30,20){$\Gamma_l$}
       \put(48,20){$\Gamma_r$}
       \put(60,30){$\Sigma_1$}

       \put(81,40){$\small{\begin{pmatrix}0&0&1&0\\ 0&1&0&0\\ -1&0&0&0\\ 0&0&0&1 \end{pmatrix}}$}
       \put(71.5,57){$\small{\begin{pmatrix}1&0&0&0\\ 0&1&0&0\\ e^{2n\lambda_1} & 0 & 1 & 0 \\ 0&0&0&1 \end{pmatrix}}$}
       \put(46,73){$\small{\begin{pmatrix}1&0&0&0\\ -e^{n(\lambda_1-\lambda_2-\kappa z)}&1&0&0\\
                                    0&0&1&e^{n(\lambda_1-\lambda_2+\kappa z)}\\ 0&0&0&1 \end{pmatrix}}$}
       \put(-20,73){$\small{\begin{pmatrix}1&e^{n(\lambda_2-\lambda_1+\kappa z)}&0&0\\ 0&1&0&0\\
                                    0&0&1&0\\ 0&0&-e^{n(\lambda_2-\lambda_1-\kappa z)}&1 \end{pmatrix}}$}
       \put(-21,57){$\small{\begin{pmatrix}1&0&0&0\\ 0&1&0&0\\ 0&0&1&0\\ 0&e^{2n\lambda_2}&0&1 \end{pmatrix}}$}
       \put(-26,40){$\small{\begin{pmatrix}1&0&0&0\\ 0&0&0&1\\ 0&0&1&0\\ 0&-1&0&0 \end{pmatrix}}$}
       \put(-21,23){$\small{\begin{pmatrix}1&0&0&0\\ 0&1&0&0\\ 0&0&1&0\\ 0&e^{2n\lambda_2}&0&1 \end{pmatrix}}$}
       \put(-20,5){$\small{\begin{pmatrix}1&e^{n(\lambda_2-\lambda_1+\kappa z)}&0&0\\ 0&1&0&0\\
                                    0&0&1&0\\ 0&0&-e^{n(\lambda_2-\lambda_1-\kappa z)}&1 \end{pmatrix}}$}
       \put(46,5){$\small{\begin{pmatrix}1&0&0&0\\ -e^{n(\lambda_1-\lambda_2-\kappa z)}&1&0&0\\
                                    0&0&1&e^{n(\lambda_1-\lambda_2+\kappa z)}\\ 0&0&0&1 \end{pmatrix}}$}
       \put(71.5,23){$\small{\begin{pmatrix}1&0&0&0\\ 0&1&0&0\\ e^{2n\lambda_1}&0&1&0\\ 0&0&0&1 \end{pmatrix}}$}
  \end{picture}
   \vspace{6mm}
   \caption{The jump matrices $J_T$ in the RH problem \ref{rhpforT} for  $T$ in
   a neighborhood of $z=0$. A local parametrix in this neighborhood is built by means
   of the RH problem \ref{rhp:modelM} for $M$ as in Figure~\ref{fig:modelRHP}.}
   \label{fig:jumpsT2}
\end{center}
\end{figure}
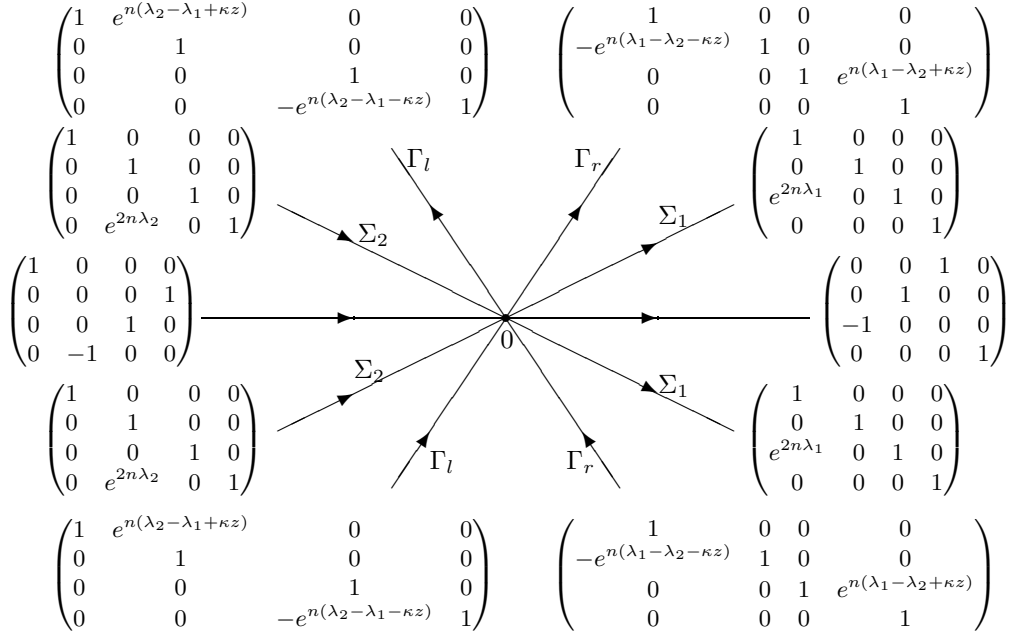

\subsection{Large $n$ behavior} \label{subsec:largen}

Now we take a closer look at the jump matrices $J_T$. The jump matrix is
constant on  $[-\alpha, 0]$ and on $[0,\beta]$. On all parts of $\Sigma_T
\setminus [-\alpha, \beta]$ the jump matrix is equal to the identity matrix
plus one or two non-zero off-diagonal entries. Ideally, we would like to have
that all off-diagonal entries tend to $0$ as $n \to \infty$. However, we cannot
hope for this in a neighborhood of $0$, due to the fact that we modified the
equilibrium problem near $0$. The  exceptional neighborhood of $0$  is
shrinking as $n$ increases at a rate of $O(n^{-2/3})$. Later we will construct
a local parametrix in a larger neighborhood of $0$ of radius $O(n^{-1/3})$.

We show here that outside this shrinking neighborhood the jump matrices $J_T$
on $\Sigma_T \setminus [-\alpha,\beta]$ do indeed tend to the identity matrix
as $n \to \infty$.

The functions $\lambda_1$ and $\lambda_2$ that appear in $J_T$ depend on $n$.
They have limits as $n\to\infty$, which we denote by $\lambda_1^*$ and
$\lambda_2^*$. Since $\beta \to \beta^*$, $\alpha \to \alpha^*$, $\delta_{1,2}
\to 0$ (see \eqref{limit:al and beta}--\eqref{critlim:del2}) and $p_{1,2} \to
p_{1,2}^*$ as $n \to \infty$, we have by \eqref{def:xi12} and
\eqref{def:lambda12},
\begin{align} \label{lambdakstar}
    \lambda_k^*(z)  = \int_0^{z} \xi_k^*(s) ds, \qquad k=1,2,
    \end{align}
where
\begin{equation} \label{xikstar}
    \xi_1^*(z)  = \frac{ z^{1/2} (z-\beta^*)^{1/2}}{2t(1-t)}, \qquad
    \xi_2^*(z)  = \frac{z^{1/2} (z+\alpha^*)^{1/2}}{2t(1-t)},
    \end{equation}
where the branches of the square roots are taken with are positive for
large positive $z$.

We also have
\begin{align}
    \lambda_1^*(z) & =
        \frac{2z- \beta^*}{8t(1-t)}  (z^2 - \beta^* z)^{1/2}
        \nonumber \\
        &\qquad -
         p_1^* \left(\log \left(z - \tfrac{\beta^*}{2} + (z^2 - \beta^* z)^{1/2}
         \right) - \log\left(-\tfrac{\beta^*}{2}\right) \right)
         \label{lambda1star},
        \\
      \lambda_2^*(z) & =
        \frac{2z + \alpha^*}{8t(1-t)}  (z^2 + \alpha^* z)^{1/2}
        \nonumber \\
        & \qquad  -
         p_2^* \left(\log \left(z + \tfrac{\alpha^*}{2} + (z^2 +\alpha^* z)^{1/2}
         \right) - \log\left(\tfrac{\alpha^*}{2}\right) \right)
          \label{lambda2star},
\end{align}
with the appropriate branches of the logarithms.

\begin{proposition}  There exists a constant $C > 0$
such that for every large $n$, we have
\begin{equation}\label{lambda:convergencespeed}
  \left| \lambda_k(z) - \lambda_k^*(z) \right| \leq C n^{-2/3} \max(|z|, |z|^{1/2}),
    \qquad z \in \mathbb C,
\end{equation}
for $k=1,2$.
\end{proposition}

\begin{proof}
We have by \eqref{def:xi12},
\begin{align} \nonumber
    \xi_1(z) - \xi_1^*(z) & =
        - \frac{\delta_1 (z-\beta)^{1/2} }{2t(1-t) z^{1/2}} + \frac{z^{1/2} ((z-\beta)^{1/2} - (z-\beta^*)^{1/2})}{2t(1-t)} \\
        & = - \frac{\delta_1 (z-\beta)^{1/2} }{2t(1-t) z^{1/2}} - \frac{(\beta-\beta^*) z^{1/2}}{2t(1-t)  ((z-\beta)^{1/2} + (z-\beta^*)^{1/2})}
        \nonumber \\
        & = - \frac{\delta_1 (z-\beta)^{1/2}}{2t(1-t) z^{1/2}} -
        \frac{\beta-\beta^*}{2t(1-t)}  \frac{z^{1/2}}{(z-\beta^*)^{1/2}}   \left(1 +
        \frac{(z-\beta)^{1/2}}{(z-\beta^*)^{1/2}}\right)^{-1}.
        \label{eq:xi1diffestimate}
        \end{align}
Note that  $\frac{(z-\beta)^{1/2}}{(z-\beta^*)^{1/2}}$ has non-negative real part, so that
\[ \left|1 + \frac{(z-\beta)^{1/2}}{(z-\beta^*)^{1/2}}\right|^{-1} \leq 1 \]
for every $z$. Then since $\delta_1 = O(n^{-2/3})$ and $\beta - \beta^* =
O(n^{-2/3})$, we can estimate \eqref{eq:xi1diffestimate} as
\begin{align*}
    \xi_1(z) - \xi_1^*(z) = O(n^{-2/3} |z|^{-1/2}) + O(n^{-2/3}|z-\beta^*|^{-1/2}) + O(n^{-2/3}),
            \qquad \text{as } n \to \infty,
     \end{align*}
uniformly for $z \in \mathbb C$. After integration from $0$ to $z$, see \eqref{def:lambda12},
we find \eqref{lambda:convergencespeed} with $k=1$.
The proof for $k=2$ is similar.
\end{proof}

\subsection{Estimate on local lenses}

We need an estimate for $\Re \lambda_1$ on $\Sigma_1$ outside
of the shrinking disk of radius $n^{-1/3}$,
\[ D(0, n^{-1/3}) = \{ z \in \mathbb C \mid |z| < n^{-1/3} \} \]
around $0$ and a fixed disk $D(\beta^*, \varepsilon)$ around $\beta^*$.
The following lemma gives such an estimate, together with a similar estimate
for $\Re \lambda_2$ on $\Sigma_2$.
\begin{lemma} \label{lem:locallensestimates}
\begin{enumerate}
\item[\rm (a)] Let $\varepsilon > 0$. Then there is a constant $c_1 > 0$ such that
for every large enough $n$, we have
\begin{equation} \label{eq:lambda1estimate}
    \Re \lambda_1(z) \leq - c_1 n^{-1/2}, \qquad z \in \Sigma_1 \setminus (D(0,n^{-1/3}) \cup D(\beta^*, \varepsilon)).
    \end{equation}
\item[\rm (b)] Let $\varepsilon > 0$. Then there is a constant $c_2 > 0$ such that
for every large enough $n$, we have
\begin{equation} \label{eq:lambda2estimate}
    \Re \lambda_2(z) \leq - c_2 n^{-1/2}, \qquad z \in \Sigma_2 \setminus (D(0,n^{-1/3}) \cup D(-\alpha^*, \varepsilon)).
    \end{equation}
\end{enumerate}
\end{lemma}

\begin{proof}
For  $\lambda_1^*$ it is easy to show from \eqref{lambdakstar} and \eqref{xikstar}
that $\Re \lambda_1^*(z) < 0$ in the strip $0 < \Re z < \beta^*$, $\Im z \neq 0$,
and
\[ \Re \lambda_1^*(z) \leq  - c_1' \min(|z|^{3/2}, |z-\beta^*|^{3/2}), \qquad z \in \Sigma_1 \]
for some constant $c_1'>0$.
Then  by \eqref{lambda:convergencespeed} we have for another constant $c_1'' > 0$
and for $n$ large enough.
\begin{align*}
    \Re \lambda_1(z) & = \Re \lambda_1^*(z) + \Re (\lambda_1(z) - \lambda_1^*(z))  \\
    &   \leq - c_1'' |z|^{3/2} + C' n^{-2/3} |z|^{1/2}, \qquad z \in \Sigma_1  \setminus D(\beta^*, \varepsilon)
        \end{align*}
        with $C' > 0$.
This leads to \eqref{eq:lambda1estimate} with a suitable constant $c_1 > 0$.

The proof of part (b) is similar.
\end{proof}

\subsection{Estimate on global lenses}

We start by being more precise on the location of the global lenses
$\Gamma_l$ and $\Gamma_r$.

\begin{lemma} \label{lemma:global lenses}
We can (and do) choose the global lenses $\Gamma_r$ in the right-half plane
and $\Gamma_l$ in the left-half plane such that for some constants $c_r, c_l > 0$,
\begin{equation} \label{eq:lambda1min2estimate1}
    \Re(\lambda_1^*(z)-\lambda_2^*(z)) \leq - c_r \min (|z|, |z|^{3/2}), \qquad \text{for } z \in \Gamma_r,
    \end{equation}
 and
\begin{equation} \label{eq:lambda1min2estimate2}
    \Re(\lambda_1^*(z)-\lambda_2^*(z)) \geq c_l \min(|z|, |z|^{3/2}), \qquad \text{for } z \in \Gamma_l.
    \end{equation}
\end{lemma}
\begin{proof}
We consider the curve
\[ \mathcal C : \quad \Re(\lambda_1^*(z) - \lambda_2^*(z)) = 0. \]
By definition $\lambda_1^*(0) = \lambda_2^*(0) = 0$ and so the
curve $\mathcal C $ contains $0$.
By \eqref{xikstar} we have as $z \to 0$,
\begin{align*}
    \xi_1^*(z) & =  \frac{\pm i \sqrt{\beta^*}}{2t(1-t)} z^{1/2} + O(z^{3/2})
                             = \frac{\pm i (p_1^*)^{1/4}}{t^{3/4}(1-t)^{3/4}} z^{1/2} + O(z^{3/2}), \qquad \pm \Im z > 0, \\
    \xi_2^*(z) & = \frac{\sqrt{\alpha^*}}{2t(1-t)} z^{1/2} + O(z^{3/2})
                  = \frac{(p_2^*)^{1/4}}{t^{3/4}(1-t)^{3/4}} z^{1/2} + O(z^{3/2}),
    \end{align*}
    where we used \eqref{def:alphastar} and \eqref{def:betastar}.
After integration, we find
\begin{equation} \label{lambda12starexpansion}
\begin{aligned}
    \lambda_1^*(z) & = \frac{\pm 2i (p_1^*)^{1/4}}{3 t^{3/4}(1-t)^{3/4}} z^{3/2} + O(z^{5/2}), \qquad \pm \Im z > 0, \\
    \lambda_2^*(z) & = \frac{2(p_2^*)^{1/4}}{3t^{3/4}(1-t)^{3/4}} z^{3/2} + O(z^{5/2}),
    \end{aligned}
    \end{equation}
and so
\begin{align*}
    \lambda_1^*(z) - \lambda_2^*(z) & =
    \frac{-2 \left((p_2^*)^{1/4} \mp i (p_1^*)^{1/4} \right)}{3 t^{3/4}(1-t)^{3/4}} z^{3/2} + O(z^{5/2}),
    \qquad \pm \Im z > 0,
\end{align*}
as $z \to 0$.
This implies that the curve $\mathcal C$ makes angles
\[ \pm \left( \frac{\pi}{3} + \frac{2}{3} \arctan \left( \frac{(p_1^*)^{1/4}}{(p_2^*)^{1/4}} \right) \right) \]
with the positive real axis. If $p_1^* = p_2^* = 1/2$, then this is a right angle, and then in fact $\mathcal C$
coincides with the imaginary axis.

As $z \to \infty$, we have by \eqref{lambda1star} and \eqref{lambda2star},
\[ \lambda_1^*(z) - \lambda_2^*(z) = - \frac{(\alpha^* + \beta^*)}{4t(1-t)} z   + O(\log |z|), \]
which implies that $\mathcal C$ gets more parallel to the imaginary axis as $z
\to \infty$ in the sense that $\arg z \to \pm \frac{\pi}{2}$ as $z \to \infty$
with $z \in \mathcal C$.

Thus $\mathcal C$ divides the plane into a part
on the right where $\Re (\lambda_1^*(z) - \lambda_2^*(z)) < 0$
and a part on the left where $\Re (\lambda_1^*(z) - \lambda_2^*(z)) > 0$.
We can then take a smooth curve $\Gamma_r$ in the region to the right of $\mathcal C$,
 making  angles $\pm \varphi_2$ with $0 < \varphi_2 < \pi/3$ at the origin, and extending
 to infinity  in such a way that \eqref{eq:lambda1min2estimate1} holds.

Similarly we can take $\Gamma_l$ to the left of $\mathcal C$ such that \eqref{eq:lambda1min2estimate2}
holds.
\end{proof}

Now we can make the necessary estimates for $\Re (\lambda_1 - \lambda_2 \pm \kappa_n z)$ on
$\Gamma_r$ and $\Gamma_l$ outside  of the shrinking disk $D(0, n^{-1/3})$ of radius $n^{-1/3}$
around $0$.
Here we write $\kappa = \kappa_n$ to emphasize its dependence on $n$, see \eqref{def:kappa}
and \eqref{kappa:asy}.
\begin{lemma} \label{lem:globallensestimates}
\begin{enumerate}
\item[\rm (a)] There is a constant $c_1 > 0$ such that for every large enough $n$, we have
\begin{equation} \label{eq:Gammarestimate}
    \Re (\lambda_1(z) - \lambda_2(z) \pm \kappa z) \leq - c_1 \min(|z|, |z|^{3/2}) \leq -c_1 n^{-1/2}, \qquad z \in \Gamma_r \setminus D(0,n^{-1/3}).
    \end{equation}
\item[\rm (b)] There is a constant $c_2 > 0$ such that for every large enough $n$, we have
\begin{equation} \label{eq:Gammalestimate}
    \Re (\lambda_1(z) - \lambda_2(z) \pm \kappa z) \geq c_2 \min(|z|, |z|^{3/2}) \geq  c_2 n^{-1/2}, \qquad z \in \Gamma_l \setminus D(0,n^{-1/3}).
    \end{equation}
\end{enumerate}
\end{lemma}

\begin{proof}
This follows by combining \eqref{lambda:convergencespeed} with  the estimates
\eqref{eq:lambda1min2estimate1}--\eqref{eq:lambda1min2estimate2} and the fact
that $\kappa = \kappa_n = O(n^{-2/3})$, see \eqref{kappa:asy}. We will not give
details as the proof is similar to the proof of Lemma
\ref{lem:locallensestimates}.
\end{proof}

As a result of Lemmas \ref{lem:locallensestimates} and
\ref{lem:globallensestimates} we have that outside the shrinking disk $D(0,
n^{-1/3})$ the jump matrices $J_T$ on the local and global lenses tend to the
identity matrix. More precisely:

\begin{corollary} \label{cor:JTestimates}
    \begin{enumerate}
    \item[\rm (a)]
    For every $\varepsilon > 0$ there is a constant $c_1 > 0$ such that
    \[ J_T(z) = I + O\left(e^{-c_1 n^{1/2}}\right) \qquad \text{ as } n \to \infty, \]
     uniformly for $z \in (\Sigma_1 \cup \Sigma_2 \cup \Gamma_l \cup \Gamma_r) \setminus \left(D(0,n^{-1/3}) \cup D(-\alpha^*, \varepsilon)
        \cup D(\beta, \varepsilon)\right)$.
    \item[\rm (b)]
    There is a constant $c_2 > 0$ such that
    \[ J_T(z) = I + O\left(e^{-c_2 n |z|}\right) \qquad \text{ as } n \to \infty, \]
    uniformly for  $z \in \Gamma_l \cup \Gamma_r$ with $|z| \geq 1$.
    \item[\rm (c)]
    For every $\varepsilon > 0$ there is a constant $c_3 > 0$ such that
    \[ J_T(z) = I + O\left(e^{-c_3 n |z|^2}\right) \qquad \text{ as } n \to \infty, \]
    uniformly for $z \in (-\infty, -\alpha^* - \varepsilon] \cup [\beta^* + \varepsilon, \infty)$.
    \end{enumerate}
\end{corollary}

\begin{proof}
Part (a) follows from the estimates in Lemmas \ref{lem:locallensestimates} and \ref{lem:globallensestimates}
which apply to the off-diagonal elements in the jump matrix $J_T$ on the lips
of the local and global lenses.

Part (b) also follows directly from Lemma \ref{lem:globallensestimates}.

Part (c) follows from the variational inequalities in \eqref{var:ineq1} and
\eqref{var:ineq2} that in fact can be strengthened to
\begin{align*}
    \Re \lambda_{1,+}(x) & = \Re \lambda_{1,-}(x) \geq c_3 x^2, \qquad x > \beta + \varepsilon, \\
  \Re \lambda_{2,+}(x) & = \Re \lambda_{2,-}(x) \geq c_3 x^2, \qquad x < -\alpha - \varepsilon,
  \end{align*}
for some constant $c_3 > 0$ independent of $n$.
\end{proof}

\subsection{Global parametrix}
\label{subsection:global}

In this section we construct a global parametrix $P^{(\infty)}(z)$ for $T(z)$. The matrix-valued function
$P^{(\infty)}(z)$ satisfies the following RH problem, which we obtain from
the RH problem \ref{rhpforT} for $T$ by ignoring the jumps on the local and global lenses and
on $(-\infty, -\alpha)$ and $(\beta, \infty)$. That is, we keep the jump on $[-\alpha,\beta]$ only.

\begin{rhp} \label{rhpforPinfty} \textrm{ }
\begin{enumerate}
\item[\rm (1)] $P^{(\infty)}(z)$ is analytic for $z\in\mathbb C \setminus([-\alpha,\beta])$.
\item[\rm (2)] $P^{(\infty)}$ satisfies the jumps
\begin{align*}
P^{(\infty)}_+(x) &= P^{(\infty)}_-(x)\begin{pmatrix} 0 & 0 & 1 & 0 \\ 0 & 1 & 0 & 0 \\
-1 & 0 & 0 & 0 \\ 0 & 0 & 0 & 1 \end{pmatrix}, \qquad \textrm{for }
x\in(0,\beta),
\\
P^{(\infty)}_+(x) &= P^{(\infty)}_-(x)\begin{pmatrix} 1 & 0 & 0 & 0 \\ 0 & 0 & 0 & 1 \\
0 & 0 & 1 & 0 \\ 0 & -1 & 0 & 0 \end{pmatrix}, \qquad \textrm{for }
x\in(-\alpha,0).
\end{align*}
  \item[\rm (3)] As $z\to\infty$, we have that
\begin{equation*}P^{(\infty)}(z) =
    I+O(1/z).
\end{equation*}
\end{enumerate}
\end{rhp}

The RH problem for $P^{(\infty)}$ decouples into two $2 \times 2$ RH problems
that can be easily solved. A solution of the RH problem \ref{rhpforPinfty} is
\begin{equation}\label{globalparametrix} P^{(\infty)}(z) =
\frac{1}{2}\begin{pmatrix} 1 & 0 & i & 0 \\ 0 & 1 & 0 & -i \\
i & 0 & 1 & 0 \\
0 & -i & 0 & 1
\end{pmatrix}
\begin{pmatrix}
    \gamma_1(z) &0&0&0 \\ 0&\gamma_2^{-1}(z)&0&0\\
    0&0&\gamma_1^{-1}(z)&0\\ 0&0&0&\gamma_2(z)
\end{pmatrix}
\begin{pmatrix} 1 & 0 & -i & 0 \\ 0 & 1 & 0 & i \\
-i & 0 & 1 & 0 \\
0 & i & 0 & 1
\end{pmatrix},
\end{equation}
where
\begin{equation}\label{def:gamma12}
\gamma_1(z):=\left(\frac{z-\beta}{z}\right)^{1/4},\qquad
\gamma_2(z):=\left(\frac{z}{z+\alpha}\right)^{1/4},
\end{equation}
and where we choose the branches of  the $1/4$ powers that
are positive and real for  large enough real $z$.

\subsection{Local parametrices around the non-critical endpoints}
\label{subsection:Airy}

On a small but fixed neighborhood around each of the non-critical endpoints $-\alpha$
and $\beta$, we construct local parametrices $P^{(-\alpha)}$ and
$P^{(\beta)}$ out of Airy functions.
We can in fact take $n$-independent disks $D(\beta^*, \varepsilon_1)$ and $D(-\alpha^*, \varepsilon_2)$
for certain $\varepsilon_1, \varepsilon_2 > 0$ such that
$P^{(\beta)}$ and $P^{(-\alpha)}$ have the same jumps as $T$ has in the
respective disks and such that
    \begin{equation} \label{Airymatching}
\begin{aligned}
    P^{(\beta)}(z) & = (I + O(n^{-1})) P^{(\infty)}(z), \qquad \text{uniformly for } |z-\beta^*| = \varepsilon_1, \\
    P^{(-\alpha)}(z) & = (I+O(n^{-1})) P^{(\infty)}(z), \qquad \text{uniformly for } |z+\alpha^*| = \varepsilon_2,
    \end{aligned}
    \end{equation}
as $n \to \infty$.

The construction with Airy functions is well-known in the literature and we do not give details, cf.\ also
Section \ref{subsec:localparametrices} of this paper.

\subsection{Local parametrix around the origin}
\label{subsection:local}

\subsubsection{Statement of the local RH problem}

In this section we construct a local parametrix around the origin, with the
help of the model RH problem for $M(\zeta)$. We want to solve the following RH
problem.

\begin{rhp} \label{rhpforP0} We look for $P^{(0)}$
satisfying the following:
\begin{enumerate}
\item[\rm (1)] $P^{(0}(z)$ is analytic for $z\in D(0, n^{-1/3}) \setminus \Sigma_T$,
where $D(0,n^{-1/3})$ denotes the disk of radius $n^{-1/3}$ around $0$.
\item[\rm (2)] $P^{(0)}$ satisfies the jumps
\begin{align*}
    P^{(0)}_+ = P^{(0)}_- J_T, \qquad \text{on } \Sigma_T \cap D(0, n^{-1/3}),
\end{align*}
where $J_T$ is the jump matrix in the RH problem for $T$.
  \item[\rm (3)] As $n\to\infty$, we have that
\begin{equation}\label{matching condition}
    P^{(0)}(z) =   \left(I+ N(z) + O(n^{-1/3}) \right) P^{(\infty)}(z) \quad \text{uniformly for } |z| = n^{-1/3}
\end{equation}
where
\begin{equation} \label{eq:NisO1}
    N(z) = O(1) \qquad \text{as } n \to \infty \quad \text{uniformly for } |z| = n^{-1/3}.
\end{equation}
\end{enumerate}
\end{rhp}

The matching condition in (3) is posed on the circle $|z| = n^{-1/3}$ which
is shrinking as $n$ increases. It is different from the usual matching
condition which requires
\begin{equation} \label{eq:usualmatching}
    P^{(0)}(z) = \left( I + O(n^{-\delta}) \right) P^{(\infty)}(z)  \quad \text{ as } n\to\infty
    \end{equation}
for some $\delta > 0$. It turns out that in general we cannot achieve \eqref{eq:usualmatching}.
Only in case $\kappa = 0$ we can achieve \eqref{eq:usualmatching} on a circle of
fixed radius.
For $\kappa \neq 0$, we have to control the terms $e^{\pm n \kappa z}$ that appear
in the jump matrix $J_T$ on $\Gamma_l$ and $\Gamma_r$. These terms remain bounded if $|z| = O(n^{-1/3})$
which (partly) explains the radius $n^{-1/3}$ of the shrinking disk.

An essential issue for the further analysis is that the $4 \times 4$ matrix $N(z)$ from \eqref{matching condition}--\eqref{eq:NisO1} is nilpotent
of degree two, i.e., $N^2(z) = 0$, and even more
\[ N(z_1) N(z_2) = 0 \qquad \text{for any } z_1, z_2 \text{ different from $0$}. \]
The matrix-valued function $N(z)$ is also analytic in a punctured neighborhood
of $0$ with a simple pole at $0$, see the explicit formula in Section
\ref{sec:local parametrix}. See also \cite{DesK,KMW2} for a similar feature in
the RH analysis.

\subsubsection{Basic idea for the construction of the parametrix}
\label{main idea}

We will construct $P^{(0)}$ in the form
\[ P^{(0)}(z) = \widehat{P}^{(0)}(z)
    \diag\left(e^{n\lambda_1(z)}, e^{n(\lambda_2(z)
+ \kappa z)},e^{-n\lambda_1(z)},e^{-n(\lambda_2(z) - \kappa z)}\right). \] Then
in order that $P^{(0)}$ has the jump matrices $J_T$, the matrix-valued function
$\widehat{P}^{(0)}$ should have constant jumps on each part of $\Sigma_T \cap
D(0, n^{-1/3})$ and these constant jumps are exactly the same as the jumps in the RH
problem \ref{rhp:modelM} for $M$, cf.\ the jump matrices in Figures
\ref{fig:modelRHP} and \ref{fig:jumpsT2}. The model RH problem \ref{rhp:modelM}
depends on parameters $r_1$, $r_2$, $s_1$, $s_2$. It follows that for any
choice of these parameters, any conformal map $f_n(z)$ that maps the contours
$\Sigma_T \cap D(0,n^{-1/3})$ into the ten rays $\bigcup_{j=0}^9 \Gamma_j$ and any
analytic prefactor $E_n(z)$, the definition
\begin{multline*}
    P^{(0)}(z) = E_n(z) M\left(f_n(z); r_1, r_2, s_1, s_2\right) \\
    \times \diag\left(e^{n\lambda_1(z)}, e^{n(\lambda_2(z) + \kappa z)},e^{-n\lambda_1(z)},e^{-n(\lambda_2(z) - \kappa z)}\right)
    \end{multline*}
    will give us
$P^{(0)}$ that satisfies the jump conditions in the RH problem \ref{rhpforP0}
for $P^{(0)}$. The conformal map $f_n$ and the analytic prefactor $E_n(z)$
should then be chosen in order to have the asymptotic condition \eqref{matching
condition}.

However, there is not enough freedom to achieve this. Since the jumps in the
model RH problem for $M$ do not depend on the parameters $r_1, r_2, s_1, s_2$,
we also let these depend on $z$ and on $n$, and we put
\begin{multline*}
    P^{(0)}(z) = E_n(z) M\left(f_n(z); r_{1,n}(z), r_{2,n}(z), s_{1,n}(z), s_{2,n}(z) \right) \\
    \diag\left(e^{n\lambda_1(z)}, e^{n(\lambda_2(z) + \kappa z)},e^{-n\lambda_1(z)},e^{-n(\lambda_2(z) - \kappa z)}\right),
 \end{multline*}
 where $r_{1,n}(z)$, $r_{2,n}(z)$, $s_{1,n}(z)$, $s_{2,n}(z)$ depend  analytically on $z$.
Now we have these functions at our disposal, as well as $E_n(z)$ and
$f_n(z)$, to achieve the condition (3) in the RH problem
\ref{rhpforP0} for $P^{(0)}$.

It turns out that the right choice for these functions takes the form
\[ f_n(z) = n^{2/3} f(z), \quad r_{j,n}(z) = r_j(z), \quad s_{j,n}(z) = n^{2/3} s_{j}(z), \qquad j=1,2, \]
for certain analytic $f(z)$, $r_1(z)$, $r_2(z)$ that do not depend
on $n$, and analytic $s_1(z)$, $s_2(z)$ that still depend on $n$ but
only in a mild way. We next describe these functions.

\subsubsection{Auxiliary functions}
Recall the functions $\lambda_1$ and $\lambda_2$ from Section~\ref{sec:lamda function}
and their limits $\lambda_1^*$ and $\lambda_2^*$ as $n \to \infty$ from Section~\ref{subsec:largen}.

We define
the following auxiliary functions.
\begin{definition}
\begin{enumerate}
\item[(a)]
We define
\begin{equation}
    \label{def:conformal} f(z) = - (p_1^*)^{-1/6}  \left( \frac{3}{2} \lambda_1^*(z) \right)^{2/3},
\end{equation}
which is a conformal map in a neighborhood of the origin, see also \eqref{f:limit}.
\item[(b)] We define functions
\begin{equation}\label{def:r12}
    r_1 = r_1(z)=(p_1^*)^{1/4}, \qquad
    r_2(z)=  \pm i  (p_1^*)^{1/4}   \frac{\lambda_2^*(z)}{\lambda_1^*(z)}, \quad  \pm \Im z > 0,
\end{equation}
where $r_1$ is a constant and $r_2(z)$ is an analytic function near $0$,
independent from~$n$.
\item[(c)] Finally, we define the $n$-dependent functions
\begin{align}\label{def:s12}
    s_1(z) = \frac{1}{2} \frac{\lambda_1(z)-\lambda_1^*(z)}{(-f(z))^{1/2}},
\qquad
    s_2(z) =  \frac{1}{2} \frac{\lambda_2(z) -
 \lambda_2^*(z)}{f(z)^{1/2}}.
\end{align}
\end{enumerate}
\end{definition}

Note that the functions are defined such that
\begin{align}
    \lambda_1(z) & = \frac{2}{3} r_1(z) (-f(z))^{3/2}+ 2s_1(z)(-f(z))^{1/2}
      \label{P0:jumps1}, \\
    \lambda_2(z) & = \frac{2}{3} r_2(z) (f(z))^{3/2} + 2s_2(z)(f(z))^{1/2}
    \label{P0:jumps2}.
\end{align}

\begin{lemma}\label{lemma:limit of parameters}
There is $r > 0$ such that the functions $f$, $r_2$, $s_1$ and $s_2$ are analytic in the disk $D(0,r) = \{ z \mid |z| < r\}$.
In addition the  following hold.
\begin{enumerate}
\item[\rm (a)] The function $f(z)$ is real for real $z \in D(0,r)$, and
\begin{equation} \label{f:limit}
     f(z) = \frac{z}{\sqrt{t(1-t)}} + O(z^2) \qquad \text{as } z \to 0.
     \end{equation}
\item[\rm (b)] The function $r_2(z)$ is real and positive for
real $z \in D(0,r)$ and
\begin{align} \label{r2:limit}
    r_2(0) = (p_2^*)^{1/4}.
\end{align}
\item[\rm (c)] The functions $s_1$ and $s_2$ depend on $n$ in such a way
that $n^{2/3}s_1$ and $n^{2/3}s_2$ have limits as $n \to \infty$. The limiting
functions are analytic in $D(0,r)$ and satisfy
\begin{equation}
\begin{aligned} \label{s12:limit}
    \lim_{n\to\infty} n^{2/3} s_1(0) & = \frac{(p_1^*)^{1/4}}{2(\sqrt{p_1^*}+\sqrt{p_2^*})}L_5, \\
    \lim_{n\to\infty} n^{2/3} s_2(0) &= -\frac{(p_2^*)^{1/4}}{2(\sqrt{p_1^*}+\sqrt{p_2^*})}L_6,
    \end{aligned}
    \end{equation}
where $L_5$ and $L_6$ are given in \eqref{doublescaling:L5}--\eqref{doublescaling:L6}.
\end{enumerate}
\end{lemma}

\begin{proof}
Parts (a) and (b) follow from the definitions of $\lambda_1^*$ and
$\lambda_2^*$ in \eqref{lambdakstar}--\eqref{xikstar}, see also
\eqref{lambda12starexpansion}.

For part (c) we note that by \eqref{def:s12}, \eqref{def:xi12}--\eqref{def:lambda12}, and \eqref{lambdakstar}--\eqref{xikstar},
\begin{align} \nonumber
    \lim_{n \to \infty} n^{2/3} s_1(z) & =
    \frac{1}{2 (-f(z))^{1/2}} \lim_{n \to \infty} \left[ n^{2/3} \int_0^z (\xi_1(s) - \xi_1^*(s)) ds \right] \\
    & \nonumber
       = - \frac{\lim_{n \to \infty} \left[ n^{2/3} (\beta - \beta^*) \right]}{2t(1-t) (-f(z))^{1/2}}
                \int_0^z \left(\frac{s}{s-\beta^*}\right)^{1/2} ds \\
         & \quad        - \frac{\lim_{n \to \infty} \left[ n^{2/3} \delta_1 \right]}{4t(1-t) (-f(z))^{1/2}}
                \int_0^z \left(\frac{s-\beta^*}{s} \right)^{1/2} ds, \label{s1limit}
            \end{align}
which by \eqref{f:limit} is indeed an analytic function in a neighborhood of $0$.
Recall that $\beta$ and $\delta_1$ depend on $n$, and that the limits
of $ n^{2/3} (\beta - \beta^*)$ and $n^{2/3} \delta_1$ as $n \to \infty$ exist.

To evaluate \eqref{s1limit} for $z=0$ we only need to consider the second term on the
right-hand side of \eqref{s1limit}. By \eqref{critlim:del1}, \eqref{f:limit}
and \eqref{s1limit} we then  obtain
\[ \lim_{n \to \infty} n^{2/3} s_1(0) =
    \frac{L_1(1-t) + L_3 t}{4 t^{3/4} (1-t)^{3/4}} \sqrt{\beta^*}, \]
    which can be rewritten to the formula given in  \eqref{s12:limit} using
    \eqref{def:betastar}, \eqref{criticalseparation:tris}, \eqref{tcrit:bis}    and \eqref{doublescaling:L5}.

The statements in part (c) dealing with  $s_2$ follow in a similar way.
\end{proof}

Observe that $r_1$ in \eqref{def:r12}, $r_2(0)$ in \eqref{r2:limit}
and the limits of $n^{2/3}s_1(0)$ and $n^{2/3} s_2(0)$ from \eqref{s12:limit}
correspond precisely to the values in
\eqref{doublescaling:r}--\eqref{doublescaling:s}.

\subsubsection{Definition of parametrix}
\label{sec:local parametrix}

We finally define the local parametrix $P^{(0)}$ near the origin as
follows.

\begin{definition} We define
\begin{multline} \label{def:localparametrix}
P^{(0)}(z) = E_n(z) M \left(n^{2/3}f(z); r_1(z),r_2(z), n^{2/3}s_1(z),n^{2/3}s_2(z) \right) \\
    \times
    \diag\left(e^{n\lambda_1(z)},e^{n(\lambda_2(z) + \kappa z)},e^{-n\lambda_1(z)},e^{-n(\lambda_2(z) - \kappa z)}\right),
\end{multline}
where $\kappa$ is given in \eqref{def:kappa} and
\begin{multline} \label{def:En}
    E_n(z) = P^{(\infty)}(z)
    \diag(1,e^{-n\kappa z}, 1, e^{-n\kappa z})\frac{1}{\sqrt{2}}\begin{pmatrix} 1 & 0 & i & 0 \\ 0 & 1 & 0 & -i \\
    i & 0 & 1 & 0 \\
    0 & -i & 0 & 1
    \end{pmatrix}  \\
    \times \diag\left((-n^{2/3}f(z))^{1/4},(n^{2/3}f(z))^{1/4}, (-n^{2/3}f(z))^{-1/4},(n^{2/3}f(z))^{-1/4}\right).
\end{multline}
\end{definition}

\begin{remark}
Note that the parameters $r_2(z)$, $n^{2/3} s_1(z)$, $n^{2/3} s_2(z)$ of $M$ in
\eqref{def:localparametrix} are real for real $z$, but in general non-real if
$z$ is not real. We proved in Theorem \ref{theorem:solvability} that the RH
problem for $M$ is solvable, and thus that $M(\zeta; r_1, r_2, s_1, s_2)$
exists for real parameters (and $r_1 > 0$, $r_2 > 0$). By perturbation
arguments it can then be shown that the RH problem for $M$ is also solvable for
parameters that are sufficiently close to the real line.

We are interested in $z \in D(0,n^{-1/3})$ and for such $z$, the values of
$r_2(z)$, $n^{2/3} s_1(z)$, $n^{2/3} s_2(z)$ come arbitrarily close to the real
axis as $n \to \infty$. Therefore for large enough $n$, $M \left(\zeta;
r_1(z),r_2(z)\right.$, $\left. n^{2/3}s_1(z),n^{2/3}s_2(z) \right)$ exists for
every $z \in D(0, n^{-1/3})$ and the local parametrix
\eqref{def:localparametrix} is well-defined.

Also the asymptotic condition in the RH problem \ref{rhp:modelM} for $M$
will be valid uniformly for $z \in D(0, n^{-1/3})$.
\end{remark}

\begin{lemma} The prefactor $E_n(z)$ in \eqref{def:En} is
analytic in a neighborhood of  the origin.
\end{lemma}

\begin{proof} By using the expression for $P^{(\infty)}(z)$ in
\eqref{globalparametrix}, we can rewrite \eqref{def:En} as
\begin{multline}\label{En:zero}
E_n(z) = \diag(1,e^{-n\kappa z}, 1, e^{-n\kappa z})
\frac{1}{\sqrt{2}}\begin{pmatrix} 1 & 0 & i & 0 \\ 0 & 1 & 0 & -i \\
i & 0 & 1 & 0 \\
0 & -i & 0 & 1
\end{pmatrix}\\ \times
    \diag\left(\gamma_1(z) (-n^{2/3}f(z))^{1/4},\gamma_2^{-1}(z)(n^{2/3}f(z))^{1/4}, \right. \\
    \left.
    \gamma_1^{-1}(z)(-n^{2/3}f(z))^{-1/4},\gamma_2(z)(n^{2/3}f(z))^{-1/4}\right),
\end{multline}
with $\gamma_1$ and $\gamma_2$ as in \eqref{def:gamma12}. Here we used that the
matrix $\diag(1,e^{-n\kappa z}, 1, e^{-n\kappa z})$ commutes with all the other
matrices in \eqref{def:En}. From \eqref{En:zero}, \eqref{f:limit} and
\eqref{def:gamma12}, we then see that $E_n(z)$ is indeed analytic in a
neighborhood of the origin.
\end{proof}

As discussed in Section \ref{main idea}, the matrix-valued function $P^{(0)}$
defined in \eqref{def:localparametrix} satisfies the jump condition in the RH
problem \ref{rhpforP0}. [We modify the contours $\Sigma_T$ if necessary, in
such a way that $f$ maps $\Sigma_T \cap D(0, n^{-1/3})$ into $\bigcup_{j=0}^9
\Gamma_j$.] It then remains to show that  $P^{(0)}$ satisfies the matching
condition $(3)$ in the RH problem \ref{rhpforP0}.

\begin{lemma}
$P^{(0)}$ defined in \eqref{def:localparametrix} satisfies the matching
condition $(3)$ in the RH problem \ref{rhpforP0}.
\end{lemma}

\begin{proof}
The asymptotics \eqref{M:asymptotics} for $M(\zeta)$ can be rewritten as
\begin{multline} \label{M:asymptoticsbis}
        M(\zeta)  = \diag((-\zeta)^{-1/4},\zeta^{-1/4},(-\zeta)^{1/4},\zeta^{1/4})
        \frac{1}{\sqrt{2}}
        \begin{pmatrix} 1 & 0 & -i & 0 \\
            0 & 1 & 0 & i \\
            -i & 0 & 1 & 0 \\
            0 & i & 0 & 1
    \end{pmatrix}   \\
        \times \left(I+\frac{\widetilde M_1}{\zeta^{1/2}}+
            \frac{\widetilde M_2}{\zeta} + O\left(\frac{1}{\zeta^{3/2}}\right) \right)
        \diag\left(e^{-\theta_1(\zeta)},e^{-\theta_2(\zeta)},e^{\theta_1(\zeta)},e^{\theta_2(\zeta)}\right),
\end{multline}
where
\begin{equation*}
    \widetilde M_1 =
        \frac{1}{2}\begin{pmatrix} 1 & 0 & i & 0 \\
        0 & 1 & 0 & -i \\
        i & 0 & 1 & 0 \\
        0 & -i & 0 & 1 \end{pmatrix}
        \begin{pmatrix}  e^{\mp\pi i/4} &0&0&0\\0&1&0&0\\ 0&0&0&0\\0&0&0&0 \end{pmatrix}
        M_1
        \begin{pmatrix}0 &0&0&0\\0&0&0&0\\ 0&0&e^{\mp\pi i/4}&0\\0&0&0&1 \end{pmatrix}
        \begin{pmatrix} 1 & 0 & -i & 0 \\
        0 & 1 & 0 & i \\
        -i & 0 & 1 & 0 \\
            0 & i & 0 & 1  \end{pmatrix}
\end{equation*}
as $\pm\Im \zeta> 0$, and $\widetilde M_2$ takes the following form:
\begin{equation*}
        \widetilde M_2 =
        \frac{1}{2}\begin{pmatrix} 1 & 0 & i & 0 \\
            0 & 1 & 0 & -i \\
            i & 0 & 1 & 0 \\
            0 & -i & 0 & 1 \end{pmatrix}
        \begin{pmatrix}\ast & \ast &0&0\\ \ast& \ast &0&0\\ 0&0&\ast&\ast\\0&0&\ast&\ast
            \end{pmatrix}
            \begin{pmatrix} 1 & 0 & -i & 0 \\
                0 & 1 & 0 & i \\
                -i & 0 & 1 & 0 \\
                0 & i & 0 & 1 \end{pmatrix}
\end{equation*}
with $\ast$ denoting certain unimportant entries.

Recall that $M(\zeta)$ as well as $\widetilde{M}_1$ depend on $r_1, r_2, s_1,
s_2$. Also $\theta_j(\zeta)$ depends on $r_j$ and $s_j$, for $j=1,2$, see
\eqref{def:theta1}--\eqref{def:theta2}. Using the notation $
\theta_j(\zeta;r_j,s_j)$ to denote the dependence on $r_j$ and $s_j$, we have
\begin{equation} \label{thetajrelation}
    \theta_j\left(n^{2/3} f(z); r_j(z),n^{2/3} s_j(z)\right) = n \lambda_j(z)
    \end{equation}
because of \eqref{def:theta1}--\eqref{def:theta2} and
\eqref{P0:jumps1}--\eqref{P0:jumps2}.

Then we have for $z$ on the circle $|z|=n^{-1/3}$, by
\eqref{def:localparametrix}, \eqref{def:En} and \eqref{M:asymptoticsbis},
\begin{multline}
        P^{(0)}(z)(P^{(\infty)})^{-1}(z)  \\
            = P^{(\infty)}(z)\diag(1,e^{-n\kappa z}, 1, e^{-n\kappa z})
            \left(I+\frac{\widetilde M_1(z)}{n^{1/3}f^{1/2}(z)}+O\left(\frac{1}{n^{2/3}f(z)} \right)\right)
        \\  \times  \diag(1,e^{n\kappa z}, 1, e^{n \kappa
        z})(P^{(\infty)})^{-1}(z),
\label{jump on boundary}
\end{multline}
as the exponential factors involving $\theta_j$ and $\lambda_j$ cancel
due to \eqref{thetajrelation}.
In \eqref{jump on boundary} we have written $\widetilde M_1(z)$ for
\[ \widetilde M_1(r_1(z), r_2(z), n^{2/3} s_1(z), n^{2/3} s_2(z)) \]
by a slight abuse of notation.

Using the formula \eqref{globalparametrix} for $P^{(\infty)}(z)$ in \eqref{jump
on boundary}, we obtain
\begin{equation}
        P^{(0)}(z)(P^{(\infty)})^{-1}(z) = I+ N(z)+O(n^{-1/3}),\qquad \textrm{for }|z|=n^{-1/3},
\end{equation}
with
\begin{multline} \label{eq:N}
    N(z) = \frac{1}{2 n^{1/3} f^{1/2}(z)}
            \begin{pmatrix}  1 & 0 & i & 0 \\
                0 & 1 & 0 & -i \\
                i & 0 & 1 & 0 \\
                0 & -i & 0 & 1 \end{pmatrix}
                \diag \begin{pmatrix} e^{\mp \pi i/4} \gamma_1(z) & e^{-n \kappa z} \gamma_2^{-1}(z) & 0 & 0 \end{pmatrix} \\
                \times M_1(z)
                \diag \begin{pmatrix} 0 & 0 & e^{\mp \pi i/4} \gamma_1(z) & e^{n \kappa z} \gamma_2^{-1}(z)  \end{pmatrix}
                    \begin{pmatrix} 1 & 0 & -i & 0 \\
                    0 & 1 & 0 & i \\
                        -i & 0 & 1 & 0 \\
                    0 & i & 0 & 1 \end{pmatrix},
\end{multline}
where
\[ M_1(z) = M_1(r_1(z), r_2(z), n^{2/3} s_1(z), n^{2/3}s_2(z)) \]
is analytic.

Then by  the definition \eqref{eq:defM1} of $M_1$ and \eqref{eq:N},
\begin{multline} \label{eq:N2}
    N(z)=\frac{1}{2n^{1/3}f^{1/2}(z)}\begin{pmatrix} 1 & 0 & i & 0 \\
                0 & 1 & 0 & -i \\
                i & 0 & 1 & 0 \\
                0 & -i & 0 & 1  \end{pmatrix}\\ \times
    \begin{pmatrix}
        0 & 0 & \pm \gamma_1^2(z) c(z) &     i e^{\mp \pi i/4} e^{n\kappa z} \gamma_1(z) \gamma^{-1}_2(z) d(z)  \\
        0 & 0 & i e^{\mp\pi i/4}  e^{-n \kappa z} \gamma_1(z) \gamma^{-1}_2(z) d(z) &   i \gamma_2^{-2}(z) \tilde c(z)\\
        0&0&0&0 \\
        0&0&0&0
    \end{pmatrix}  \begin{pmatrix} 1 & 0 & -i & 0 \\
        0 & 1 & 0 & i \\
    -i & 0 & 1 & 0 \\
        0 & i & 0 & 1  \end{pmatrix}
\end{multline}
with $d(z)$, $c(z)$, $\widetilde c(z)$ given by \eqref{d:Painleve2}--\eqref{ctil:Hamiltonian}
with $r_1$, $r_2$, $s_1$, $s_2$ replaced by $r_1(z)$, $r_2(z)$, $n^{2/3} s_1(z)$, $n^{2/3} s_2(z)$,
respectively. These functions are analytic and remain  bounded as $n \to \infty$.

The combinations
\[ \pm \frac{\gamma_1^2(z)}{f^{1/2}(z)}, \quad e^{\mp \pi i/4} \frac{\gamma_1(z)}{\gamma_2(z) f^{1/2}(z)},
    \quad \frac{1}{\gamma_2^2(z) f^{1/2}(z)}, \qquad \mp \Im z  > 0, \]
        are analytic in a punctured neighborhood of $0$ with a simple pole at $0$.
        On the circle with radius $n^{-1/3}$ they all grow like  $O(n^{1/3})$ as $n \to \infty$.
Since in \eqref{eq:N2} there is also a factor $n^{-1/3}$ we find that
indeed
\[ N(z) = O(1),  \qquad \text{as } n \to \infty \text{ uniformly for } |z| = n^{-1/3}, \]
and this proves the lemma.
\end{proof}

From the proof of the lemma, we also find that $N(z)$ is analytic with a simple pole at $0$.
From \eqref{eq:N} or \eqref{eq:N2} it is also clear that
\begin{equation} \label{eq:Nz1Nz2}
    N(z_1) N(z_2) = 0, \qquad \text{for any } z_1, z_2 \text{ different from } 0.
    \end{equation}

\subsection{Fourth transformation}
\label{subsection:fourthtransfo}

 Using the global
parametrix $P^{(\infty)}$ and the local parametrices $P^{(-\alpha)}$,
$P^{(\beta)}$ and $P^{(0)}$, we define the fourth transformation $T \mapsto S$
as follows.

\begin{definition}
We define
\begin{equation}\label{def:S}
S = \begin{cases}
    T (P^{(0)})^{-1}, &  \text{inside the disk } D(0, n^{-1/3}) \textrm{ around }0, \\
    T (P^{(-\alpha)})^{-1}, & \text{in a fixed disk around } -\alpha^*, \\
    T (P^{(\beta)})^{-1},& \text{in a fixed disk around } \beta^*, \\
    T (P^{(\infty)})^{-1},&  \text{elsewhere}.
\end{cases}
\end{equation}
\end{definition}

Then $S$ is defined and analytic outside of $\Sigma_T$ and the three disks
around $0$, $-\alpha^*$ and $\beta^*$, with an analytic continuation across
those parts of $\Sigma_T$ where the jumps of the parametrices coincide with
those of $T$. What remains are the jumps on a contour $\Sigma_{S}$ that
consists of the three circles around $0$, $-\alpha^*$, and $\beta$, the parts
of the real intervals $(-\infty, -\alpha^*]$ and $[\beta^*, \infty)$ outside of
the disks, and the lips of the local and global lenses outside of the disks.

The circles are oriented clockwise. Then the RH problem for $S$ is

\begin{rhp} \label{rhpforS} \textrm{ }
\begin{enumerate}
\item[\rm (1)] $S$ is analytic in $\mathbb C \setminus \Sigma_{S}$.
\item[\rm (2)] $S$ satisfies the jump relation $S_+ = S_- J_{S}$
on $\Sigma_{S}$ with jump matrices:
\begin{align} \label{def:JS}
    J_{S}  = \begin{cases}
            P^{(0)} (P^{(\infty)})^{-1} &  \text{ on the boundary of } D(0, n^{-1/3}),  \\
            P^{(-\alpha)} (P^{(\infty)})^{-1} & \textrm{ on the boundary of  the disk around } -\alpha^*,  \\
            P^{(\beta)} (P^{(\infty)})^{-1}  &  \textrm{ on the boundary of  the disk around } \beta^*,  \\
            P^{(-\infty)} J_T (P^{(\infty)})^{-1} & \textrm{ elsewhere on } \Sigma_{S}.
            \end{cases}
\end{align}
 \item[\rm (3)] As $z\to\infty$, we have that
\begin{equation*}
        S(z) = I+O(1/z).
\end{equation*}
\end{enumerate}
\end{rhp}

The jump matrix $J_{S}$ is not close to the identity matrix on the circle
around $0$, since by \eqref{def:JS} and \eqref{matching condition} we have
\begin{equation} \label{eq:JS0}
    J_{S} = I + N(z) + O(n^{-1/3}) \qquad \text{as } n \to \infty
    \end{equation}
uniformly for $|z| = n^{-1/3}$, with $N(z) = O(1)$.

The other jump matrices, however, are close to the identity matrix as $n$ gets large.
Indeed from \eqref{def:JS} and the matching condition \eqref{Airymatching} we find that
\begin{equation} \label{eq:JS1}
    J_{S} = I + O(n^{-1}) \qquad \text{as } n \to \infty
    \end{equation}
for $z$ on the fixed circles around $-\alpha^*$ and $\beta^*$.

From Corollary \ref{cor:JTestimates} and \eqref{def:JS}, it follows that
$J_{S}$ has similar estimates as $J_T$ on the rest of $\Sigma_{S}$. Namely, for
certain $c_1, c_2, c_3 >0$ we have
\begin{equation} \label{eq:JS2}
    J_{S}(z) = I + O(e^{-c_1 n^{1/2}})
    \end{equation}
on the lips of the local and global lenses outside the disks,
\begin{equation} \label{eq:JS3}
    J_{S}(z) = I + O(e^{-c_2 n |z|})
    \end{equation}
on the lips of the global lenses with $|z| \geq 1$, and
\begin{equation} \label{eq:JS4}
    J_{S}(z) = I + O(e^{-c_3 n |z|^2})
    \end{equation}
on the real line outside of the disks around $-\alpha^*$ and $\beta^*$.


\subsection{Final transformation}
\label{subsection:finaltransfo}

The jump matrix \eqref{eq:JS0} on the shrinking circle $|z| = n^{-1/3}$ around
the origin does not tend to the identity matrix as $n \to \infty$. The final
transformation $S\mapsto R$ serves to resolve this issue. An important role
will be played by the special structure of the matrix $N$ in \eqref{eq:N}, see
also \cite{DesK,KMW2}.

Recall that $N(z)$ is analytic for $z$ in a punctured neighborhood of $0$ with a
simple pole at $z=0$. Define
\begin{equation} \label{eq:N0}
    N_0 = \Res_{z=0} \, N(z)
    \end{equation}
    as the residue matrix.
Then we have the splitting
\[ N(z) = \left(N(z) - \frac{N_0}{z} \right) + \frac{N_0}{z} \]
where $N(z) - N_0/z$ is analytic inside the disk $D(0, n^{-1/3})$ and $N_0/z$ is
analytic outside.
Note also that because of \eqref{eq:Nz1Nz2}
we have
\begin{equation} \label{eq:NzN0}
    N(z) N_0 = 0, \qquad N^2(z) = 0, \qquad N_0^2 = 0.
    \end{equation}

The transformation $S \mapsto R$ is now defined as follows.
\begin{definition}
We define
\begin{equation}\label{defR}
    R(z) = \begin{cases} \ds
     S(z)\left(I+N(z)- \frac{N_0}{z} \right), & \textrm{for } z \in D(0,n^{-1/3}) \setminus \Sigma_{S}, \\
     \ds    S(z)\left(I-\frac{N_0}{z} \right),& \textrm{for } z \in \mathbb C \setminus (D(0, n^{-1/3}) \cup \Sigma_{S}),
    \end{cases}
\end{equation}
where $N$ and $N_0$ are given in \eqref{eq:N} and \eqref{eq:N0}, respectively.
\end{definition}

Then $R$ is defined and analytic in $\mathbb C \setminus \Sigma_R$ where
$\Sigma_R = \Sigma_{S}$ and  $R$ satisfies a RH problem of the following form.
\begin{rhp}  \label{rhpforR} \textrm{ }
\begin{enumerate}
   \item[\rm (1)] $R$ is analytic in $\mathbb C \setminus \Sigma_R$.
   \item[\rm (2)]  $R$ satisfies the jumps $R_+ = R_- J_R$ on  $\Sigma_{R}$.
   \item[\rm (3)]  $R(z) = I+O(1/z)$ as $z\to\infty$.
\end{enumerate}
\end{rhp}

The jump matrix $J_R$ for $|z|=n^{-1/3}$ is by \eqref{defR} (recall that we use
the clockwise orientation)
\begin{align*}
    J_R(z) & = \left(I + N(z) - \frac{N_0}{z} \right)^{-1} J_{S}(z) \left( I - \frac{N_0}{z} \right) \\
    & =
    \left(I - N(z) + \frac{N_0}{z} \right) \left( I + N(z) + O(n^{-1/3}) \right) \left( I - \frac{N_0}{z} \right)
    \end{align*}
where the calculation of the inverse of $I + N(z) - N_0/z$ was done with the
help of \eqref{eq:NzN0}. Then if we expand this product of three matrices, many
terms cancel due to \eqref{eq:NzN0}. What remains is simply
\begin{equation} \label{eq:JR0}
    J_R(z) = I + O(n^{-1/3}), \qquad |z| =n^{-1/3},
    \end{equation}
where we also use the fact that $N(z)$ as well as $N_0/z$ are $O(1)$ for $|z| = n^{-1/3}$.

The transformation \eqref{defR} does not change the jump matrices on the other
parts of $\Sigma_R$ in an essential way. Thus $J_R$ tends to the identity
matrix on these parts as well, with a rate of convergence that is the same as
that for $J_{S}$, see \eqref{eq:JS1}--\eqref{eq:JS4}.

We have now achieved the goal of the steepest descent analysis of the RH
problem. The jump matrices for $R$ tend to the identity matrix as $n \to
\infty$, uniformly on $\Sigma_R$ as well as in $L^2(\Sigma_R)$. By standard
arguments, see \cite{Dei}, (and also \cite{BK3} for the case of contours that
are varying with $n$), we conclude that
\begin{equation} \label{eq:Restimate}
    R(z)=I+O\left( \frac{1}{n^{1/3}(1+|z|)}\right)
\end{equation}
as $n\to\infty$, uniformly for $z\in\mathbb{C}\setminus\Sigma_{R}$.

The estimate \eqref{eq:Restimate} is the main outcome of our Deift-Zhou
steepest descent analysis for the RH problem \ref{rhp:Y} for~$Y$. We will use
it to prove the main theorems in the next section.


\section{Proofs of Theorems \ref{theorem:kernelpsi} and \ref{theorem:Painleve2rec}}
\label{section:proofsmaintheorems}

\subsection{Proof of Theorem~\ref{theorem:kernelpsi}}
\label{subsection:proofkernel}

In this section we prove Theorem~\ref{theorem:kernelpsi} by
following the subsequent transformations of the RH problem. We start
with the expression \eqref{correlationkernel} for the correlation
kernel $K_n(x,y)$:
\begin{equation} \label{correlationkernel:Y}
    K_n(x,y) = \frac{1}{2\pi i(x-y)}\begin{pmatrix} 0 & 0 & w_{2,1}(y) &
    w_{2,2}(y)\end{pmatrix} Y_{+}^{-1}(y)Y_{+}(x)\begin{pmatrix} w_{1,1}(x)\\
    w_{1,2}(x)\\ 0 \\ 0
\end{pmatrix}.
\end{equation}

We first assume that both $x$ and $y$ are positive and close to 0.
From the first transformation $Y\mapsto X$ in
\eqref{defX1}--\eqref{defX3}, it follows that
\begin{equation} \label{correlationkernel:X}
    K_n(x,y) = \frac{1}{2\pi i(x-y)}\begin{pmatrix} 0 & 0 & w_{2,1}(y) &
    0\end{pmatrix} X_{+}^{-1}(y)X_{+}(x)\begin{pmatrix} w_{1,1}(x)\\
    0\\ 0 \\ 0 \end{pmatrix}.
\end{equation}

Applying the second transformation $X\mapsto U$ \eqref{defU} to
\eqref{correlationkernel:X}, we have
\begin{equation*} \label{correlationkernel:U}
    K_n(x,y) = \frac{e^{n\left(\frac{1-2t}{4t(1-t)}(y^2-x^2)+k(y-x)\right)}}{2\pi i(x-y)}\begin{pmatrix} 0 & 0 & e^{-n\lambda_{1,+}(y)} &
    0\end{pmatrix} U_{+}^{-1}(y)U_{+}(x)\begin{pmatrix} e^{-n\lambda_{1,+}(x)}\\
    0\\ 0 \\ 0 \end{pmatrix},
\end{equation*}
where $k=-\frac{a_1}{2t}+\frac{b_1}{2(1-t)}$. This constant depends on $n$ and
from \eqref{doublescaling:a}, \eqref{doublescaling:b} and \eqref{tcrit:bis} it
follows that
\begin{equation}\label{kstar}
k^*:=\lim_{n\to\infty} k=-\frac{a_1^*}{2t}+\frac{b_1^*}{2(1-t)} =
-\frac{a_2^*}{2t}+\frac{b_2^*}{2(1-t)}
\end{equation}
with the convergence rate $O(n^{-2/3})$.

By the third transformation $U\mapsto T$ in
\eqref{defT1}--\eqref{defT3} and our assumption on $x,y$, it is
readily seen that
\begin{multline} \label{correlationkernel:T}
    K_n(x,y) = \frac{e^{n\left(\frac{1-2t}{4t(1-t)}(y^2-x^2)+k(y-x)\right)}}{2\pi i(x-y)}\begin{pmatrix}
    -e^{n\lambda_{1,+}(y)} & 0 & e^{-n\lambda_{1,+}(y)} &
    0\end{pmatrix} \\
    \times T_{+}^{-1}(y)T_{+}(x) \begin{pmatrix} e^{-n\lambda_{1,+}(x)}\\
    0\\ e^{n\lambda_{1,+}(x)} \\ 0 \end{pmatrix}.
\end{multline}
For $z\in D(0,n^{-1/3})$, we have from \eqref{def:S} and \eqref{defR} that
\begin{align}\label{repre of s}
    T(z) & =S(z)P^{(0)}(z) =R(z)\left(I+N(z)-\frac{N_0}{z} \right)^{-1}P^{(0)}(z) \nonumber \\
        &=R(z)\left(I-N(z)+ \frac{N_0}{z} \right)P^{(0)}(z).
\end{align}
Substituting \eqref{repre of s} and \eqref{def:localparametrix} into
\eqref{correlationkernel:T}, it then follows that
\begin{multline} \label{correlationkernel:P}
    K_n(x,y) = \frac{e^{n\left(\frac{1-2t}{4t(1-t)}(y^2-x^2)+k(y-x)\right)}}{2\pi i(x-y)} \\
    \times \begin{pmatrix} -1 & 0 & 1 &     0\end{pmatrix}
        M^{-1}_+(n^{2/3}f(y);r_1(y),r_2(y),n^{2/3}s_1(y),n^{2/3}s_2(y)) \\
    \times E_n^{-1}(y) \left(I+N(y)-\frac{N_0}{y} \right)    R^{-1}(y)R(x)\left(I-N(x)+\frac{N_0}{x} \right)E_n(x) \\
    \times M_+(n^{2/3}f(x);r_1(x),r_2(x),n^{2/3}s_1(x),n^{2/3}s_2(x))
    \begin{pmatrix} 1\\ 0\\ 1 \\ 0 \end{pmatrix}.
\end{multline}

Now we fix $u,v>0$ and take
\begin{equation}\label{scale of x y}
        x=\frac{u}{cn^{2/3}},\qquad y=\frac{v}{cn^{2/3}},
\end{equation}
where $c = \frac{1}{\sqrt{t(1-t)}}$ with $t$ in \eqref{tcrit:bis}.
Then for $n$ large enough, $x$ and $y$ are inside the disk $D(0,n^{-1/3})$. It is then readily seen that
\begin{equation}
        \lim_{n\to \infty}e^{n\left(\frac{1-2t}{4t(1-t)}(y^2-x^2)\right)}=1,
\end{equation}
and by \eqref{f:limit},
\[ n^{2/3}f(x)\to  u,\qquad n^{2/3}f(y)\to v, \]
as $n\to\infty$. From \eqref{scale of x y} we also find that
\begin{equation}\label{eq:limit of parameters}
\begin{aligned}
&r_2(x) \to r_2,  &&r_2(y) \to r_2,\\
&n^{2/3}s_1(x)\to s_1, &&n^{2/3}s_1(y)\to s_1,\\
&n^{2/3}s_2(x)\to s_2,&&n^{2/3}s_2(y)\to s_2,
\end{aligned}
\end{equation}
as $n\to\infty$ (see Lemma \ref{lemma:limit of parameters}). Note that
$r_1(x)=r_1(y)=r_1$; compare \eqref{r2:limit}--\eqref{s12:limit} with
\eqref{doublescaling:r}--\eqref{doublescaling:s}. Furthermore, we have that
\begin{equation}\label{RinverseR}
R^{-1}(y)R(x)=I+O\left(\frac{x-y}{n^{1/3}}\right) =
I+O\left(\frac{u-v}{n}\right),
\end{equation}
and in view of \eqref{En:zero} we see that $E_n(x)=O(n^{1/6})$,
$E_n(y)=O(n^{1/6})$ and
\begin{equation}\label{EinverseE}
E_n^{-1}(y)E_n(x) = I+O(n^{-1/3}),
\end{equation}
as $n\to\infty$. The constants implied by the $O$-symbols are independent of
$u$ and $v$ when $u$ and $v$ are restricted to compact subsets of the real line.
Thus, a combination of \eqref{RinverseR}--\eqref{EinverseE} and the fact that
$N(z)-N_0/z$ is uniformly bounded with respect to $z$ near the origin and $n$
gives us
\[
\lim_{n\to\infty}E_n^{-1}(y)
    \left(I+N(y)- \frac{N_0}{y} \right)
    R^{-1}(y)R(x)\left(I-N(x)+\frac{N_0}{x} \right)E_n(x) = I.
\]
Inserting this into \eqref{correlationkernel:P}, we then obtain from
\eqref{kstar} and \eqref{scale of x y}--\eqref{eq:limit of
parameters} that
\begin{align}
\label{correlationkernel:M}
    \nonumber &\lim_{n\to\infty} \frac{e^{k^*(u-v)n^{1/3}/c}}
    {cn^{2/3}}K_n\left(\frac{u}{cn^{2/3}},\frac{v}{cn^{2/3}}\right)
    \\&= \frac{1}{2\pi i(u-v)}\begin{pmatrix} -1 & 0 & 1 &
    0\end{pmatrix} M^{-1}_+(u;r_1,r_2,s_1,s_2)M_+(v;r_1,r_2,s_1,s_2)
    \begin{pmatrix} 1\\
    0\\ 1 \\ 0 \end{pmatrix}.
\end{align}
This, together with \eqref{tacnodekernel:pos}, yields
\begin{align}
    \lim_{n \to \infty} \frac{e^{k^*(u-v)n^{1/3}/c}}{c n^{2/3}}
    K_n \left(  \frac{u}{cn^{2/3}}, \frac{v}{cn^{2/3}}\right)
        = K^{tacnode}(u,v; r_1, r_2, s_1, s_2),
\end{align}
which is \eqref{kernel at tacnode}, since $c_2 = k^*/c$.

The case where $x$ and/or $y$ are negative, or equivalently, $u$ and/or $v$ are
negative can be proved in a similar manner. We do not give details here.

This completes the proof of Theorem \ref{theorem:kernelpsi}. $\bol$

\subsection{Proof of Theorem~\ref{theorem:Painleve2rec}}
\label{subsection:proofrec}

In this section we prove Theorem~\ref{theorem:Painleve2rec}. The proof will be
quite similar to the one in \cite{DelKui1}.

Following the transformations \eqref{defX3}, \eqref{defU}, \eqref{defT3},
\eqref{def:S} and \eqref{defR} in the steepest descent analysis, we have
the following representation for  $z$ sufficiently large in the region between
the contours $\Gamma_r$ and $\Gamma_l$:
\begin{align}
    Y(z) &= L^n U(z) \Lambda^{-n}(z)
    \label{tracebacksteepestdescent}
    = L^n R(z)\left(I+N_0/z\right)P^{(\infty)}(z)\Lambda^{-n}(z),
\end{align}
where $L$ and $\Lambda$ are given in \eqref{defLmx} and \eqref{def:Lam}. As in
\cite{DelKui1}, the following lemma is easy to check.
\begin{lemma} \label{lemmaY1}
For the matrix $Y_1$ in \eqref{asymptoticconditionY0} we have
\begin{equation} \label{Y1formula}
    Y_1 = L^n \left(\Lam_1 + P_1^{(\infty)} + N_0 + R_1 \right) L^{-n},
    \end{equation}
where $\Lam_1$, $P_1^{(\infty)}$ and $R_1$ are matrices  from the expansions as
$z \to \infty$,
\begin{align} \label{asymptoticG}
    \Lam^{-n}(z)L^n & =
    \left(I+\frac{\Lam_1}{z}+O\left(\frac{1}{z^2}\right)\right)\diag(z^{n_1},z^{n_2},z^{-n_1},z^{-n_2}), \\
    \label{asymptoticPinfty}
    P^{(\infty)}(z) & = I+\frac{P_1^{(\infty)}}{z}+O\left(\frac{1}{z^2}\right),\\
    \label{asymptoticR0}
    R(z) & = I+\frac{R_1}{z}+ O\left(\frac{1}{z^2}\right).
\end{align}
\end{lemma}

We are only interested in the combinations $(Y_1)_{1,2}(Y_1)_{2,1}$ and
$(Y_1)_{1,4}(Y_1)_{4,1}$
of entries of $Y_1$. Since $L$ is a diagonal matrix, the factors $L^n$ and
$L^{-n}$ in \eqref{Y1formula} will not play a role for these combinations.
Also, since $\Lam_1$ is a diagonal matrix (which is clear from
\eqref{asymptoticG}, since $\Lam(z)$ and $L$ are both diagonal), this does not
play a role either. Therefore we have for $i < j$,
\begin{equation} \label{cijcjiformula}
    (Y_1)_{i,j} (Y_1)_{j,i} =
    \left(P_1^{(\infty)} + N_0+ R_1 \right)_{i,j}
    \left(P_1^{(\infty)} + N_0+ R_1 \right)_{j,i}.
\end{equation}

In what follows we evaluate $P_1^{(\infty)}$, $N_0$ and $R_1$.
\begin{lemma} \label{lemma:Pinfty1}
The matrix $P_1^{(\infty)}$ in \eqref{asymptoticPinfty} can be written as
\begin{equation}
\label{Pinfty1formula} P_1^{(\infty)} =
    i \sqrt{t(1-t)} \begin{pmatrix}
    0 & 0 & \sqrt{p_1} & 0 \\
    0 & 0 & 0 & \sqrt{p_2} \\
    -\sqrt{p_1} & 0 & 0 & 0 \\
    0 & -\sqrt{p_2} & 0 & 0
\end{pmatrix}.
\end{equation}
\end{lemma}

This result and its proof are exactly the same as in \cite{DelKui1}.

\begin{lemma} \label{lemma:R1evaluation}
    The matrix $R_1$ in \eqref{asymptoticR0} satisfies
    \begin{equation}
        R_1 = O(n^{-2/3}),
        \qquad \text{ as } n \to \infty.
    \end{equation}
\end{lemma}

\begin{proof}
The asymptotic behavior of $J_R$ in \eqref{eq:JR0} can be extended to
\begin{equation}\label{jumpR:nextterm}
J_R(z) = I+ J_R^{(1)}(z)n^{-1/3}+O(n^{-2/3}),
\end{equation}
as $n\to\infty$, uniformly for all $z$ on the circle $|z|=n^{-1/3}$, for a
certain matrix-valued function $J_R^{(1)}(z)$. An explicit formula for
$J_R^{(1)}(z)$ can be obtained in terms of the matrices $\widetilde M_2$ and
$\widetilde M_3$ in \eqref{M:asymptoticsbis}. For us, it will be sufficient to
know that this matrix has a Laurent series expansion at the origin:
$$ J_R^{(1)}(z) = C_{3,n}\frac{1}{z^3}+C_{2,n} \frac{1}{z^2}+C_{1,n}\frac{1}{z}+C_{0,n}(z),$$
where the matrix function $C_{0,n}(z)$ is analytic at the origin, and where the
matrices $C_{k,n}$, $k=1,2,3$, are depending on $n$ in such a way that the
limiting matrices
\begin{equation}\label{Rexp:bisa} \lim_{n\to\infty} n^{1/3}C_{k,n} =: C_k
\end{equation}
exist for all $k=1,2,3$.

From \eqref{jumpR:nextterm}, standard considerations \cite{DKMVZ1,KMVV} show
that $R$ itself has an expansion of the form
\begin{equation}\label{Rexp:bisb}
R(z) = I+R^{(1)}(z)n^{-1/3}+O(n^{-2/3}),\end{equation} for $n\to\infty$, where
the leading term $R^{(1)}$ is given by
\begin{equation}\label{Rexp:bisc}
R^{(1)}(z) = \left\{\begin{array}{ll}C_{3,n}\frac{1}{z^3}+C_{2,n}
\frac{1}{z^2}+C_{1,n}\frac{1}{z},& \textrm{if }|z|>n^{-1/3},\\
-C_{0,n}(z),& \textrm{ else}.
\end{array}\right.\end{equation}
Taking into account \eqref{Rexp:bisa}--\eqref{Rexp:bisc}, we see that the
matrix $R_1$ in \eqref{asymptoticR0} must be of order $O(n^{-2/3})$, and the
lemma follows.
\end{proof}

\begin{lemma} \label{lemma:N0evaluation}
    We have
    \begin{equation} \label{R1blocks}
        N_0 = K \begin{pmatrix} A & B \\ C & D \end{pmatrix} n^{-1/3} + O(n^{-2/3}),
        \qquad \text{ as } n \to \infty,
    \end{equation}
    where $K$ is the constant defined in \eqref{doublescaling:K},
    and where the $2 \times 2$ blocks $A$, $B$, $C$, $D$ are
    given by
    \begin{align} \label{defA}
    A & = t (b_1^* - b_2^*) \begin{pmatrix} * & - q(\sigma)
    \\ q(\sigma) & * \end{pmatrix}, \\
    \label{defB}
    B & = i\sqrt{t(1-t)(a_1^*-a_2^*)(b_1^*-b_2^*)}
    \begin{pmatrix} * & q(\sigma) \\ q(\sigma) & * \end{pmatrix}, \\
    \label{defC}
    C & = -i\sqrt{t(1-t)(a_1^*-a_2^*)(b_1^*-b_2^*)}
    \begin{pmatrix} * & q(\sigma) \\ q(\sigma) & * \end{pmatrix}, \\
    \label{defD}
    D & = (1-t)(a_1^*-a_2^*)
    \begin{pmatrix} * & -q(\sigma) \\ q(\sigma) & * \end{pmatrix},
    \end{align}
    with the value $\sigma$ given by \eqref{doublescaling:sigma}, with $q$ denoting the
Hastings-McLeod solution to Painlev\'e~II, and where the entries denoted with
$*$ depend on the Hamiltonian $u$ through the constants $c,\til c$ in
\eqref{c:Hamiltonian}--\eqref{ctil:Hamiltonian}.
\end{lemma}
\begin{proof} Recall that by definition,
$N_0$ is the residue of the matrix $N$ in \eqref{eq:N} at the origin. The lemma
then follows by computing this residue on account of the formulas in
\eqref{tcrit:bis}, \eqref{d:Painleve2}--\eqref{ctil:Hamiltonian},
\eqref{doublescaling:r}--\eqref{doublescaling:s}, \eqref{def:gamma12} and
\eqref{f:limit}, taking also into account Remark~\ref{remark:compare}.
\end{proof}

\begin{remark}\label{remark:compare}
For facility of comparison, we have stated the above lemma in exactly the same
way as in \cite{DelKui1}. We note however that the statement can be
considerably simplified in the present case, since we are now looking exactly
at the critical time $t=t_{\crit}$. This means that the variable $t$ in the
above formulas can be substituted by \eqref{tcrit:bis}. This implies in
particular that in \eqref{defA}--\eqref{defD} we have
$$ t (b_1^* - b_2^*) = \sqrt{t(1-t)(a_1^*-a_2^*)(b_1^*-b_2^*)} = (1-t)(a_1^*-a_2^*).
$$
Further simplification can be obtained by substituting
\eqref{criticalseparation:tris}.
\end{remark}

Lemmas~\ref{lemma:Pinfty1} and \ref{lemma:N0evaluation} have exactly the same
form as the corresponding results in \cite{DelKui1}. (The entries denoted with
$*$ in Lemma~\ref{lemma:N0evaluation} are slightly different when compared to
\cite{DelKui1}, but this turns out to be irrelevant.) Hence
Theorem~\ref{theorem:Painleve2rec} (and also the analogous result for the \lq
diagonal recurrence coefficients\rq) follows from the calculations in
\cite{DelKui1}.

\section*{Acknowledgements}

We thank Pavel Bleher, Tom Claeys, and Alexander Its for useful
discussions.

S.~Delvaux is a Postdoctoral Fellow of the Fund for Scientific Research - Flanders
(Belgium).

A.B.J.~Kuijlaars is supported by K.U.\ Leuven research
grant OT/08/33, FWO-Flanders project G.0427.09, by the Belgian Interuniversity
Attraction Pole P06/02, and by grant MTM2008-06689-C02-01 of the Spanish Ministry
of Science and Innovation.

L.~Zhang is supported by  FWO-Flanders project G.0427.09.

\end{document}